\documentclass[a4paper,11pt]{article}

\usepackage{amssymb,amsmath,amsfonts} \allowdisplaybreaks
\usepackage{a4wide}

\usepackage{graphicx}

\usepackage[linktocpage, pdftex]{hyperref}
    \usepackage{color}
    \definecolor{darkred}{rgb}{0.5,0,0}
    \definecolor{darkgreen}{rgb}{0,0.5,0}
    \definecolor{darkblue}{rgb}{0,0,0.5}

\hypersetup{colorlinks,linkcolor=darkblue,filecolor=darkgreen,urlcolor=darkred,citecolor=darkblue}
\hypersetup{
pdftitle={Almost-additive ergodic theorems for amenable groups},
pdfauthor={Felix Pogorzelski},}

\usepackage{hyperref}
\usepackage{color}
\usepackage{amsthm}
\usepackage{enumerate}
\numberwithin{equation}{section}    
\theoremstyle{plain}
\newtheorem{Theorem}{Theorem}[section]
\newtheorem{Proposition}[Theorem]{Proposition}
\newtheorem{Corollary}[Theorem]{Corollary}
\newtheorem{Lemma}[Theorem]{Lemma}
\theoremstyle{definition}

\newtheorem{Definition}[Theorem]{Definition}
\newtheorem{Example}[Theorem]{Example}

\theoremstyle{remark}
\newtheorem{Remark}[Theorem]{Remark}

\renewcommand{\epsilon}{\varepsilon}
\renewcommand{\phi}{\varphi}
\newcommand{\one}{\mathbf{1}} 

\newcommand{\ato}[2]{\genfrac{}{}{0pt}{2}{#1}{#2}}

\newcommand{\EE}{\mathbb{E}}
\newcommand{\RR}{\mathbb{R}}
\newcommand{\CC}{\mathbb{C}}
\newcommand{\NN}{\mathbb{N}}
\newcommand{\TT}{\mathbb{T}}
\newcommand{\ZZ}{\mathbb{Z}}

\newcommand{\hm}[1]{\textbf{*}\leavevmode{\marginpar{\tiny%
$\hbox to 0mm{\hspace*{-0.5mm}$\leftarrow$\hss}%
\vcenter{\vrule depth 0.1mm height 0.1mm width \the\marginparwidth}%
\hbox to 0mm{\hss$\rightarrow$\hspace*{-0.5mm}}$\\\relax\raggedright #1}}}
\title{Almost-additive ergodic theorems for amenable groups}
\author{Felix Pogorzelski \thanks{Mathematisches Institut, Universit\"at Leipzig, 04109 Leipzig, Germany}}
\date{}

\begin{document}

\maketitle


\begin{abstract}
In this paper we prove a general convergence theorem for almost-additive set functions on unimodular, amenable groups. These mappings take their values in some Banach space. By extending the theory of epsilon-quasi tiling techniques, we set the ground for far-reaching applications in the theory of group dynamics. In particular, we verify the almost-everywhere convergence of abstract approximable bounded, additive processes, as well as a Banach space approximation result for the spectral distribution function (integrated density of states) for random operators on discrete structures in a metric space. Further, we include a Banach space valued version of the Lindenstrauss ergodic theorem for amenable groups.
\end{abstract}

\section{Introduction}

The analysis of the dynamical properties of almost-additive functions has become an essential undertaking in the mathematical fields of ergodic theory, dynamical systems, spectral theory and mathematical physics. In general, one considers mappings $F$ defined on a set of subsets in a locally compact measure space $(G, m)$ and taking its values in some Banach space $X$. Further, the function is supposed to have certain boundedness-, additivity- and invariance properties.  
Inspired by elementary convergence theorems for subadditive sequences with values in $\RR$ or $\CC$, one might raise the question whether the abstract averages $F(U_j)/m(U_j)$ converge in the topology of $X$ along sequences or even along nets which are characteristic for the space $G$. Positive results on this issue for certain dynamical systems may e.g.\@ be found in \cite{Lenz-02, LenzS-06, Klassert-07, Krieger-07, Krieger-10}. In \cite{LenzMV-08}, it is shown that the normalized almost-additive functions converge along canonical exhaustions of $G=\ZZ^d$ by $d$-dimensional boxes. As a consequence, the authors verify the uniform approximation of the integrated density of states for certain finite-range operators by finite-volume analogues. 
Those results were generalized in \cite{LenzSV-10} for a geometrically restricted class of countable, amenable groups. A proof of the convergence in the case of {\em all} countable, amenable groups has recently been given in \cite{PogorzelskiS-11}. In this work, the authors use the theory of $\varepsilon$-quasi tilings, cf.\@ \cite{OrnsteinW-87}, to handle the geometric difficulties that occured in previous papers.    
A more abstract choice for the space $G$ is a unimodular, second countable, amenable group along with its (invariant) Haar measure $m$. This is the framework of the present paper. 

Unlike in \cite{PogorzelskiS-11}, there is no need to assume that $G$ is countable in the context of this work.  
It is well known that for second countable groups, amenability is characterized by the existence of particular averaging sequences $\{U_j\}_{j \in \NN}$ of compact sets, called F{\o}lner sequences. Under mild compactness criteria for orbits induced by $G$-actions on the Banach space $X$, we are able to prove the convergence for the normalized versions of almost-additive functions along F{\o}lner sequences. The corresponding main Theorem~\ref{thm:METSF} can be considered as an abstract mean ergodic theorem for set functions and it generalizes previous results from ergodic theory. In the situation of countable groups, the theorem coincides with the main result of \cite{PogorzelskiS-11} if the function $F$ under consideration is invariant under the action by $G$. \\ 
A first direct application of the main Theorem~\ref{thm:METSF} is given in Section~\ref{sec:BAP}, where we prove the almost-everywhere convergence along strong F{\o}lner sequences for approximable, bounded, additive processes defined on unimodular, amenable groups. The corresponding Theorem~\ref{thm:PWET} complements the hitherto existing knowledge in the field, where one usually works with semigroup contraction actions in $G=\RR^{d+}$, cf.\@ \cite{Sato-99, Sato-03}. Dealing with general amenable groups, our theorem is significantly more general as far as the choices for the space $G$, as well as for the approximating sequences are concerned. However, we have to restrict ourselves slightly by making sure that the dynamics are similar to group actions. In Theorem~\ref{thm:DET}, we prove a dominated ergodic theorem for bounded, additive processes, which is a crucial ingredient for the convergence result. The techniques of the proof are inspired by classical differentiation theorems, see also \cite{AkcogluJ-81, Emilion-85, Sato-98, Sato-99, Sato-03}. \\
As another application, we use Theorem~\ref{thm:METSF} to verify the almost-sure approximation of the integrated density of states for random operators on discrete structures in metric spaces which are quasi isometric to some unimodular, amenable group. With probability one, the limit is attained in $L^p(I)$, where due to reflexivity $1 < p < \infty$ and $I \subset \RR$ is an arbitrary bounded interval. This is the first Banach space convergence result which takes into account the model of {\sc Lenz} and {\sc Veseli{\'c}} (cf.\@ \cite{LenzV-09}) in its original setting. Thus, we extend the considerations from \cite{PogorzelskiS-11}, where the authors had to deal with countable spaces and with a deterministic setting.\\
The paper is organized as follows. In Section~\ref{sec:preliminaries}, we give the preliminaries which are necessary for the understanding of this work. The main issue will be the introduction of the notions of weak and strong F{\o}lner sequences in amenable groups. Furthermore, growth conditions on those sequences are studied. Next, in Section \ref{sec:tilingtheorems}, we develop combinatorial decomposition theorems which will be the major tools in the proof of our main result, Theorem~\ref{thm:METSF}. In this context, we verify the existence of $\varepsilon$-quasi tilings which extend the techniques of \cite{OrnsteinW-87} and \cite{PogorzelskiS-11}. Precisely, we construct so-called uniform decomposition towers (UDT) in Theorem~\ref{thm:UDT} which can be considered as a pair of families of $\varepsilon$-quasi tilings with a high degree for its uniformity and covering properties. The short Section~\ref{sec:absMET} is devoted to a recapitulation of classical mean ergodic theorems of amenable group actions on Banach spaces, cf.\@ \cite{Greenleaf-73}. In the following Section~\ref{sec:MET}, we combine the classical results of the previous section with the combinatorial methods of Section \ref{sec:tilingtheorems} in order to prove Theorem~\ref{thm:METSF} which can be considered as a mean ergodic theorem for almost-additive set functions. In the corresponding proof, we will have to work with uniform decomposition towers. 
Another variant of the main theorem with slightly different assumptions is given in Corollary~\ref{cor:METSF}.
Turning to measure dynamical systems, we treat pointwise ergodic theorems for amenable groups in Section~\ref{sec:PWET}. Extending classical ergodic theory, these investigations are independent of the results of the Sections \ref{sec:tilingtheorems} and \ref{sec:MET}. Precisely, in Theorem~\ref{thm:lindenstr}, we formulate and prove a Banach-space generalization of the {\sc Lindenstrauss} ergodic theorem, see \cite{Lindenstrauss-01}. 
In the following Section~\ref{sec:BAP}, we introduce the notion of bounded, additive processes on amenable groups. After giving some examples, we use the so-called associated dominating process for the proof of an $L^p$-maximal inequality (Theorem~\ref{thm:DET}). Combining this latter statement with the main Theorem~\ref{thm:METSF}, we derive the corresponding Banach space valued pointwise convergence for a class of additive processes in Theorem~\ref{thm:PWET}. This class consists of 
those processes which can be well-approximated by bounded, additive
processes in $L^{\infty}$, cf.\@~Section~\ref{sec:BAP}. 
As a spectral theoretic application of pointwise convergence, we consider random operators defined on randomly chosen point sets in a metric space $(X,d_X)$ possessing a quasi isometry to some amenable group. Combining Corollary~\ref{cor:METSF} with the pointwise ergodic Theorem~\ref{thm:lindenstr}, we show in Theorem~\ref{thm:pointwise} the almost everywhere convergence of finite-dimensional IDS-approximants as elements in $L^p$.      
Finally, we include an appendix in Section~\ref{sec:appendix}, where we give the quite lengthy proofs of Theorem~\ref{thm:LPmax} and  Lemma~\ref{lemma:partition}.  

\medskip

This manuscript is an update of its previous version which with minor changes has appeared in \cite{PogJFA}. It was observed by the author that the pointwise ergodic theorem for bounded, additive processes (here Theorem~\ref{thm:PWET}), needs the additional assumption of approximability. This was pointed out in \cite{PogERR}. Following the content of the latter note, we correct the ergodic theorem and its proof in Section~\ref{sec:BAP}. Some additional minor flaws are corrected as well.

\section{Preliminaries} \label{sec:preliminaries}

Throughout this paper, we will refer to $G$ as a locally compact, second countable and amenable Hausdorff group with neutral element $\operatorname{id}$. Denoting by $\mathcal{B}(G)$ the {\em Borel} $\sigma$-algebra generated by the open subsets in $G$, one finds (up to constants) exactly one regular measure $m_L(\cdot)$ on $\mathcal{B}(G)$, called the {\em left Haar measure} which is invariant under group multiplication by elements from the left, i.e. $m_L(gA) = m_L(A)$ for every $g \in G$ and all $A \in \mathcal{B}(G)$. In this work, we need to restrict ourselves to so-called {\em unimodular} groups, i.e.\@ groups for which the unique Haar measure is both left- {\em and} right-invariant. In this case we simply write $|A|$ for the measure of some set $A \in \mathcal{B}(G)$. When integrating over sets in an unimodular amenable group, we will use the notation $dm_L(g) = dg$. Note that for instance, all discrete and all abelian groups are unimodular. We shall write $\mathcal{F}(G):= \{A \in \mathcal{B}(G) \,|\, \overline{A}\, \mbox{cpt.} \, \}$ for the collection of Borel sets in $G$ with compact closure (precompact sets). For the cardinality of some finite set $A \subseteq G$, we will write $\operatorname{card}(A)$.  \\

In every amenable second countable group, one can find so-called {\em F{\o}lner} sequences, i.e.\@ sequences of compact subsets $(S_n)$ of $G$ which are asymptotically invariant under left-translation by arbitrary compact sets. To give precise definitions, we introduce the concept of the boundary of a compact set $T$ relatively to some compact set $K$. 

\begin{Definition} \label{defi:KBD}
Let $G$ be an amenable and second countable group. Assume that $\emptyset \neq K, T \subset G$ are compact subsets in $G$. We call the set $\partial_K(T)$, defined by
\begin{eqnarray*}
\partial_K(T):= \{g \in G\,|\, Kg \cap T \neq \emptyset \,\wedge\, Kg \cap (G\setminus T) \neq \emptyset\}
\end{eqnarray*}
the \emph{$K$-boundary} of the set $T$.
Furthermore, $T$ is called \emph{$(K,\delta)$-invariant} if
\begin{eqnarray*}
\frac{|\partial_{K}(T)|}{|T|}   < \delta.
\end{eqnarray*}
\end{Definition}       

In the following Lemma, we summarize some nice and useful properties of the relative boundary definition. 

\begin{Lemma} \label{prop:prop}
Let $T,S,K \subset G$ and assume that $g \in G$. Then the following is true.
\begin{enumerate}[(i)]
\item $\partial_K(T) = \partial_K(G\setminus T)$.
\item $\partial_K(S \cup T) \subset \partial_K(S) \cup \partial_K(T)$.
\item $\partial_K(S \setminus T) \subset \partial_K(S) \cup \partial_K(T)$.
\item $\partial_K(T) \subset \partial_L(T)$ if $K \subset L \subset G$.
\item $\partial_K(Tg) = \partial_K(T)g$.
\item $|\partial_K(S \setminus T)| \leq |\partial_K(T)| + |\partial_K(S)|$
\item $\partial_K(TS) \subset \partial_K(T)S$
\end{enumerate}
\end{Lemma}

\begin{proof}
See e.g.\@ \cite{PogorzelskiS-11}, Lemma 2.3. 
\end{proof}

\begin{Definition} \label{defi:FS}
Let $(S_n)$ be a sequence of non-empty compact subsets of an unimodular group $G$. If
\begin{align*}
\lim_{n \rightarrow \infty} \frac{|S_n\triangle K S_n|}{|S_n|} = 0
\end{align*}
for all non-empty, compact $K \subset G$, then $(S_n)$ is called {\em weak F{\o}lner sequence}. If
\begin{align*}
\lim_{n \rightarrow \infty} \frac{|\partial_K(S_n)|}{|S_n|} = 0
\end{align*}
for all non-empty, compact $K \subset G$, then $(S_n)$ is called {\em strong F{\o}lner sequence}.
We say that a (weak or strong) F{\o}lner sequence $(S_n)$ is {\em nested} if $\operatorname{id}\in S_1$ and $S_n\subset S_{n+1}$ for all $n \geq 1$.
\end{Definition}

Note that in second countable groups, amenability is characterized by the existence of weak F{\o}lner sequences, see e.g.\@ \cite{Greenleaf-73}. As the following lemma shows, it is also true that all second countable, unimodular amenable groups possess strong F{\o}lner sequences. 

\begin{Lemma} \label{lemma:folner}
Let $G$ be a second countable, unimodular, amenable group. Then the following statements hold true. 
\begin{enumerate}[(i)]
\item There exists a strong F{\o}lner sequence in $G$.
\item Each strong F{\o}lner sequence is a weak F{\o}lner sequence. 
\item If $G$ is countable, then every weak F{\o}lner sequence is also a strong F{\o}lner sequence.
\item There exists a nested strong F{\o}lner sequence in $G$. 
\end{enumerate}
\end{Lemma}

\begin{proof}
See e.g.\@ \cite{PogorzelskiS-11}, Lemma 2.5.
\end{proof}

It is a well known fact that one cannot expect pointwise ergodic theorems for arbitrary F{\o}lner sequences, cf.\@ \cite{Emerson-74}. Hence it is important for our purposes to impose some growth conditions on the F{\o}lner sequences under consideration. 

\begin{Definition} \label{defi:growth}
Let $G$ be a second countable, unimodular, amenable group and assume that $(S_n)$ is a weak or strong F{\o}lner sequence in $G$. 
\begin{itemize}
\item We say that $(S_n)$ satisfies the {\em Tempelman} condition if there is a constant $C > 0$ such that
\[
\Big| \bigcup_{i \leq N} S_i^{-1}S_N  \Big| \leq C\, |S_N|
\]
for all $N \in \NN$.
\item We say that $(S_n)$ satisfies the {\em Shulman} condition if there is a constant $\tilde{C} > 0$ such that
\[
\Big| \bigcup_{i < N} S_i^{-1}S_N  \Big| \leq \tilde{C}\, |S_N|
\]
for all $N \in \NN$.
In this case, we say that $(S_n)$ is a {\em tempered} F{\o}lner sequence.  
\end{itemize}
\end{Definition} 

\begin{Remark}
It is evident that the Tempelman condition is stronger than the Shulman condition. As has been shown in \cite{Lindenstrauss-01}, tempered weak F{\o}lner sequences always exist in second countable amenable groups, but there are second countable amenable groups that do not possess a F{\o}lner sequence satisfying the Tempelman condition. On the other hand, as the following Theorem shows, there are sufficient conditions on the group for the existence of Tempelman F{\o}lner sequences.  
\end{Remark}

\begin{Theorem}[\cite{Hochman-07}, Theorem 3.4]
If for a countable, abelian, amenable group $G$, we have
\[
r(G) := \sup \{n \in \NN \,\, |\, \,\, G \mbox{ contains a subgroup isomorphic to  }\ZZ^n\} < \infty,
\]
then $G$ possesses at least one Tempelman F{\o}lner sequence. 
\end{Theorem}

\begin{Remark}
The number $r(G)$ is called the {\em abelian rank} of $G$. 
\end{Remark}

\section{Tiling Theorems} \label{sec:tilingtheorems}

This section is devoted to combinatorial decomposition theorems for unimodular, amenable groups by $\varepsilon$-quasi tilings. To do so, we recapitulate the notion of the special tiling property (STP) which has been introduced in \cite{PogorzelskiS-11}. In the latter paper, the authors prove that (STP) is always satisfied (\cite{PogorzelskiS-11}, Theorem~4.5). Further, it is shown there that large F{\o}lner sets in countable groups can be $\varepsilon$-quasi tiled by a uniform family of coverings with desirable properties on average (cf.\@ \cite{PogorzelskiS-11}, Theorem~4.7). Based on the constructions given in \cite{OrnsteinW-87}, we prove in Theorem \ref{thm:UCD} an analogous result which also holds for continuous groups. A new feature here are the quantitative estimates on the average degree of uniformity as well as on the average portions of covered mass of the $\varepsilon$-quasi-tilings in the family. In order to obtain stronger uniformity properties, we further work with {\em pairs} of such families of $\varepsilon$-quasi tilings. More precisely, we introduce the concept of a so-called {\em uniform decomposition tower} (UDT) for amenable groups in Definition \ref{defi:UDT}. In the main Theorem \ref{thm:UDT} of this section we show that the construction of these pairs is always possible in the unimodular situation. We will see in Section 5 of this paper how this construction can be used to prove a Banach space valued ergodic theorem for a special class of mappings defined on $\mathcal{F}(G)$.

As before, we always refer to $G$ as a second countable, unimodular, amenable, locally constant Hausdorff group.

\begin{Definition} \label{defi:epsdis}
Let $A,B \subseteq G$. For a number $0 < \varepsilon < 1$, we say that $A$ and $B$ are $\varepsilon$-disjoint if there are sets $\overline{A} \subseteq A$ and $\overline{B} \subseteq B$ such that
\begin{itemize}
\item $\overline{A} \cap \overline{B} = \emptyset$.
\item $|\overline{A}| \geq (1-\varepsilon)|A|$ and $|\overline{B}| \geq (1-\varepsilon)|B|$. 
\end{itemize}
\end{Definition}

\begin{Definition} \label{defi:alphacov}
Let $A,B \subseteq G$. For a number $0 < \alpha \leq 1$, we say that the set $A$ $\alpha$-covers the set $B$ if
\[
|A \cap B| \geq \alpha\,|B|. 
\]
\end{Definition}

In the following, for any real number $s \in \RR$, we use the notation
\[
\lceil s \rceil := \min\{ n \in \NN \,|\, n \geq s\}.
\]
For $0 < \varepsilon < 1$, we define the number $N(\varepsilon) \in \NN$ as 
\[
N(\varepsilon) := \lceil \log(\varepsilon)/\log(1-\varepsilon) \rceil.
\]


As has been done before in \cite{PogorzelskiS-11}, we introduce the concept of the {\em special tiling property} for amenable groups. 

\begin{Definition}[Special tiling property] \label{defi:STP}
Let $G$ be an amenable group. Then $G$ is said to have the {\em special tiling property} (STP) if for all $0 < \varepsilon < 1/10$, every parameter $0 < \beta < \varepsilon$ and every nested strong F{\o}lner sequence $(S_n)$ there are $N(\varepsilon)$ many sets
\begin{eqnarray*}
\{\operatorname{id}\} \subseteq T_1 \subseteq T_2 \subseteq \dots \subseteq T_{N}, \quad \quad T_i \in \{S_n \,|\, n \geq i\}, \,\, 1 \leq i \leq N, \,\, N:=N(\varepsilon)
\end{eqnarray*} 
as well as a number $\delta_0 > 0$ depending only on $\beta$ such that for every $0 < \delta < \delta_0$ and every $(T_NT_N^{-1}, \delta)$-invariant set $T \in \mathcal{F}(G)$, we can find finite {\em center sets} $C_i^{T}$, $1 \leq i \leq N(\varepsilon)$ such that
\begin{enumerate}[(i)]
\item $T_iC_i^T \subseteq T$ for all $1 \leq i \leq N(\varepsilon)$,
\item $\{T_ic\}_{c \in C_i^T}$ is an $\varepsilon$-disjoint family of sets for all $1 \leq i \leq N(\varepsilon)$,
\item $\{T_iC_i^{T}\}_{i=1}^{N(\varepsilon)}$ is a disjoint family of sets,
\item $\Big| \frac{|T_iC_i^T|}{|T|} - \eta_i(\varepsilon) \Big| < \beta$ for $1 \leq i \leq N(\varepsilon)$, where $\eta_i(\varepsilon):= \varepsilon(1-\varepsilon)^{N(\varepsilon)}$.
\end{enumerate}  
For fixed $0< \varepsilon < 1/10$ and $0 < \beta < \varepsilon$, we say that the {\em basis sets} $T_i$ $\varepsilon$-quasi tile the group $G$ with the parameters $\varepsilon$ and $\beta$ and if for $T \subseteq G$, the properties (i)-(iv) are satisfied, we say that $T$ has the special tiling property (STP) with respect to $(\{T_i\}_{i=1}^{N(\varepsilon)}, (S_n), \varepsilon, \beta)$ and that the set $T$ is $\varepsilon$-quasi tiled (with parameter $\beta$) by the basis sets $T_i$ with finite center sets $C_i^T$, $1 \leq i \leq N$. 
\end{Definition}

In \cite{PogorzelskiS-11}, it was proven that every unimodular amenable group has the special tiling property. Even more is true: for each number $\varepsilon > 0$, the $T_i$-translates of the covering of $T$ can be made {\em disjoint} in such a way that they maintain certain invariance properties with respect to some fixed compact set $L \subset G$. This leads to the following theorem.

\begin{Theorem}[cf.\@ \cite{PogorzelskiS-11}, Theorem 4.5] \label{thm:STP}
Let $G$ be a unimodular amenable group. Then the following assertions hold true:
\begin{enumerate}[(A)]
\item The group $G$ satisfies the special tiling property.
\item Let $\varepsilon, \beta$ and $(S_n)$ be given as in Definition \ref{defi:STP} and assume further that we are given a positive number $0 < \zeta < \varepsilon$, as well as a compact set $L \subseteq G$ containing the unity $ \operatorname{id}$. Then we can find an $\varepsilon$-quasi tiling with basis sets $T_i$, $1 \leq i \leq N:=N(\varepsilon)$ such that all the properties of Definition \ref{defi:STP} hold and such that for $0 < \delta < 6^{-N}\beta/4$, each compact, $(T_NT_N^{-1}, \delta)$-invariant set $T \subseteq G$ can be $\varepsilon$-tiled (with parameter $\beta$) by translates $T_ic$, $c \in C_i^T$ such that the if the $T_i$-translates are $(L,\eta^2)$-invariant $(1 \leq i \leq N(\varepsilon))$, they can be made pairwise disjoint in a way that guarantees that for every $1 \leq i \leq N(\varepsilon)$ and all $c \in C_i^T$, there is some set $T_i^c \subseteq T_i$ with  
\begin{itemize}
\item $|T_i^c| \geq (1-\varepsilon)|T_i|$,
\item $T_i^c$ is $(L,4\zeta)$-invariant,
\item $|\partial_L(T_i^c)| \leq |\partial_L(T_i)| + \zeta\,|T_i|$,
\item $T_iC_i^T = \cup_{c \in C_i^T} T_i^c c$, where the latter union consists of pairwise disjoint sets. 
\end{itemize}
\end{enumerate}
\end{Theorem}

For countable amenable groups, it was shown in \cite{PogorzelskiS-11} (Theorem 4.7) that there is a {\em family} of coverings which possesses a uniform covering property on average. To be precise, it was proven that for each element $u$ of the covered set $T$, the probability for this element being a center set of a covering of the family is equal to some number which only depends on $\varepsilon$ as well as on the tiling set $T_i$. 
For the sake of this paper, we will need a continuous version of a uniform decomposition theorem. 

\begin{Definition}[Uniform Continuous Decompositions] \label{defi:UCD}
Let $G$ be an amenable group. We say that $G$ satisfies the {\em uniform continuous decompositions condition} (UCDC) if for every strong F{\o}lner sequence $(U_k)$ in $G$, the following statements holds true. 
\begin{itemize}
\item For each $0 < \varepsilon \leq 1/10$, $N:=N(\varepsilon):= \lceil \log(\varepsilon)/\log(1-\varepsilon) \rceil$, for arbitrary numbers $0 < \beta, \zeta < 2^{-N}\varepsilon$, for every nested F{\o}lner sequence $(S_n)$, and for each compact set $\operatorname{id} \in L \subseteq G$, the group $G$ is $\varepsilon$-quasi tiled according to Definition \ref{defi:STP} by tiling sets
\[
 \{\operatorname{id}\}  \subseteq T_1 \subseteq T_2 \subseteq \dots \subseteq T_N,
\]
where $T_i \in \{S_n \,|\, n \geq i\}$ for $1 \leq i \leq N(\varepsilon)$ and where the latter basis are all $(L,\zeta^2)$- invariant. 
\item For fixed numbers $\varepsilon, \beta, \zeta$ and for a fixed compact set $\operatorname{id} \in L \subseteq G$, there is some number $K \in \NN$ depending on $\varepsilon$, $\beta$ and the basis sets $T_i$ such that for each $k \geq K$, we find a finite-measure set $\Lambda_k \in \mathcal{B}(G)$ along with a {\em family} 
\[
\{C_i^{\lambda}(U_k)\,|\, \lambda \in \Lambda_k, \, 1 \leq i \leq N\}
\]
of finite center sets for the basis sets $T_i$ such that for each $\lambda \in \Lambda_k$, the set $U_k$ is $\varepsilon$-quasi tiled by the translates $T_ic$, $1 \leq i \leq N$, $c \in C_i^{\lambda}(U_k)$ according to Definition \ref{defi:STP} and moreover, 
\begin{enumerate}[(I)]
\item $\frac{\left| \bigcup_{i=1}^N T_iC_i^{\lambda}(U_k) \right|}{|U_k|} \geq 1 - 4\varepsilon$ for all $\lambda \in \Lambda_k$,
\item For all $1 \leq i \leq N$ and for every Borel set $S \subseteq U_k$,
\begin{eqnarray*}
\left| |\Lambda_k|^{-1} \int_{\Lambda_k} \frac{\operatorname{card}(C_i^{\lambda}(U_k) \cap S)}{|U_k|} \, d\lambda - \frac{\eta_i(\varepsilon)}{|T_i|} \cdot \frac{|S|}{|U_k|} \right| < 4\, \frac{\beta}{|T_i|} +2\varepsilon\cdot \gamma_i ,
\end{eqnarray*} 
where $\eta_i(\varepsilon):= \varepsilon(1-\varepsilon)^{N-i}$ and the $\gamma_i > 0$ can be chosen such that $\sum_{i=1}^N \gamma_i|T_i| \leq 2$. 
\item The translates $T_ic$ ($c \in C_i^{\lambda}(U_k)$, $1 \leq i \leq N$) can be made disjoint such that the resulting sets $T_i^c c$ have the properties listed in the statement (B) of Theorem  \ref{thm:STP} for all $\lambda \in \Lambda_k$. 
\end{enumerate}
\end{itemize}
\end{Definition}

\begin{Remark}
The essential property of Definition \ref{defi:UCD} is given by the inequality (II) in the second item. Considering $\nu(\cdot) := \frac{m_L(\cdot)}{|\Lambda_k|}$ as a probability measure over the set $\Lambda_k$, the inequality shows that on average with respect to $\nu$, the center cets are uniformly distributed in $U_k$ in the sense that their mean frequency of occurrence is almost constant in every part of $U_k$.  
\end{Remark}

In the following theorem, we prove that each unimodular, amenable group is UCDC, i.e.\@ it possesses the uniform continuous decompositions condition. Being a continuous analogue of Theorem 4.7 in \cite{PogorzelskiS-11} for countable amenable groups, significant parts of the proof can be adapted. The main construction is based on Proposition I.3, 6 in \cite{OrnsteinW-87}. However, Theorem \ref{thm:UCD} extends the result of {\sc Ornstein} and {\sc Weiss} by the quantitative estimates for the coverings in Definition \ref{defi:UCD}.

\begin{Theorem}[Uniform Continuous Decompositions] \label{thm:UCD}
Each unimodular, amenable group satisfies the uniform continuous decompositions condition. 
\end{Theorem}

\begin{proof}
Let $0 < \varepsilon \leq 1/10$ and $0< \zeta, \beta < 2^{-N}\varepsilon$, as well as a compact set $\operatorname{id} \in L \subseteq G$ be given, where as usual, $N:=N(\varepsilon):= \lceil \log(\varepsilon) / \log(1- \varepsilon) \rceil$. Assume further that $(U_k)$ is a strong F{\o}lner sequence and choose $0 < \delta_0 < 6^{-N}\beta/16$.  \\

Note that by Theorem \ref{thm:STP}, we can find $(L,\zeta^2)$-invariant basis sets
\[
\{\operatorname{id}\} \subseteq T_1 \subseteq \dots \subseteq T_N,
\] 
taken from a nested F{\o}lner sequence $(S_n)$ that $\varepsilon$-quasi tile the group such that each $(T_NT_N^{-1}, \delta_0)$-invariant set $T\subseteq G$ can be $\varepsilon$-quasi tiled by translates $T_ic$ that can be made disjoint in a way that they keep the claimed invariance properties with respect to the set $L$. \\

We choose $0 < \varepsilon_1 < 1/100$. At various steps of the proof, we will have to make this parameter smaller. This is possible since the corresponding restrictions do not depend on objects developed in the following constructions, but only on $\varepsilon, \beta$ and the basis sets $T_i$. \\
We stick close to the proof of Theorem 4.7 in \cite{PogorzelskiS-11} and like in that paper, we proceed in nine steps. It turns out that we can adapt the steps (1) to (7) for our purposes. Since the differences are rather notational but not conceptional, we just describe the construction and we omit most technical calculations. \\

\begin{enumerate}[(1)]
\item  
We set $M:=\lceil \log(\varepsilon_1)/\log(1-\varepsilon_1)\rceil$ and we find $(T_NT_N^{-1}, \delta_0^2)$-invariant sets $\overline{T}_l \supseteq T_N$, $1 \leq l \leq M$, taken from a nested F{\o}lner sequence $(S_n)$ such that the $\overline{T}_l$ $\varepsilon_1$-quasi tile the group $G$ (cf.\@ Theorem \ref{thm:STP}). Then we can find some integer number $K \in \NN$ such that for each $k \geq K$, the set $T:=U_k$ is $(\overline{T}_l\overline{T}_l^{-1}, 2^{-l}\varepsilon_1)$-invariant for all $1 \leq l \leq M$. Since $\varepsilon_1$ will only depend on $\varepsilon, \beta$, as well as on the basis sets $T_i$, so does the integer $K$.
Further, we choose $\hat{T}$ to be a $(TT^{-1}, \varepsilon_1)$-invariant compact set inheriting all the mentioned invariance properties of $T$. Using Theorem \ref{thm:STP}, we can also make sure that $\hat{T}$ has the special tiling property with respect to $(\{\overline{T}_l\}_{l=1}^M, (S_n), \varepsilon_1, \beta_1)$, where $0 < \beta_1 < 2^{-M}\varepsilon_1$ (For instance, take $\hat{T}:=U_{\tilde{K}}$ for $\tilde{K} \in \NN$ large enough).  
We set $A:= \{g \in G\,|\, Tg \subseteq \hat{T}\}$ and we note that
\begin{eqnarray}\label{eqn:setA}
|A| \geq (1-\varepsilon_1)|\hat{T}|.
\end{eqnarray}

\item We fix an $\varepsilon_1$-quasi tiling of $\hat{T}$ as in Theorem \ref{thm:STP} with basis sets $\overline{T}_l$, $1 \leq l \leq M$, where we make the $\overline{T}_l$-translates actually disjoint such that the resulting disjoint translates $\overline{T}_l^c c$ are still $(T_NT_N^{-1}, 4\delta_0)$-invariant. We note that these disjoint translates $(1-2\varepsilon_1)$-cover the set $\hat{T}$, i.e.
\begin{eqnarray} \label{eqn:cover1}
\frac{\Big| \bigcup_{l=1}^M \bigcup_{c \in \overline{C}_l} \overline{T}_l^c c \Big|}{|\hat{T}|} = \frac{\sum_{l=1}^M \sum_{c \in \overline{C}_l} |\overline{T}_l^c c|}{|\hat{T}|} \geq 1- 2\varepsilon_1.
\end{eqnarray} 

\item Since all the sets $\overline{T}_l^c$ are still $(T_NT_N^{-1}, 4\delta_0)$-invariant for all $1 \leq l \leq M$ and every $c \in \overline{C}_l$, if follows from Theorem \ref{thm:STP} that we can fix in each translate $\overline{T}^c_l c$ an $\varepsilon$-quasi tiling with the basis sets $T_i$ and finite center sets $C_i^{l,c}$ such that
\begin{eqnarray} \label{eqn:cover2}
\left| \frac{|T_iC_i^{l,c}|}{|\overline{T}_l^c c|} - \eta_i(\varepsilon) \right| < \beta
\end{eqnarray}
for every $1 \leq i \leq N$. Further, we set
\[
\hat{C}_i := \bigcup_{l=1}^M \bigcup_{c \in \overline{C}_l} C_i^{l,c}
\]
for $1 \leq i \leq N$ and we note that the $\hat{C}_i$ can be considered as center sets for the basis sets $T_i$ such that the family $\{T_ic\}_{c \in \hat{C}_i}$ is $\varepsilon$-disjoint and such for $1 \leq i < j \leq N$, the elements $T_i\hat{C}_i$ and $T_j\hat{C}_j$ are disjoint. For the covering properties of this $\varepsilon$-quasi tiling, a short computation using Inequality (\ref{eqn:cover2}) shows that
\begin{eqnarray} \label{eqn:cover3}
\left| \bigcup_{i=1}^N T_i\hat{C}_i \right| \geq (1-2\varepsilon_1 - 2\varepsilon)|\hat{T}|,
\end{eqnarray} 
cf.\@ \cite{PogorzelskiS-11}, proof of Theorem 4.7, step (3).

\item The step (4) of the proof of Theorem 4.7 in \cite{PogorzelskiS-11} shows that by imposing a first condition on $\varepsilon_1$ depending on $\varepsilon$ and $\beta$, we have
\begin{eqnarray} \label{eqn:cover4}
\left| \frac{|T_i\hat{C}_i|}{|\hat{T}|} - \eta_i(\varepsilon) \right| < 2\beta
\end{eqnarray}
for all $1 \leq i \leq N$.

\item A short calculation using the $\varepsilon$-disjointness of the $T_i$-translates now shows with Inequality (\ref{eqn:cover4}) that 
\begin{eqnarray} \label{eqn:cover5}
\left| \frac{\operatorname{card}(\hat{C}_i)}{|\hat{T}|} - \frac{\eta_i(\varepsilon)}{|T_i|} \right| < \frac{2\beta}{|T_i|} + \gamma_i\varepsilon,
\end{eqnarray} 
where $\gamma_i := \operatorname{card}(\hat{C}_i)/|\hat{T}|$ for $1 \leq i \leq N$ and $\sum_{i=1}^N \gamma_i|T_i| \leq 2$, cf.\@ \cite{PogorzelskiS-11} proof of Theorem 4.7, step (5). 

\item We recall from step (1) that we chose $A$ to be the collection of elements $a \in G$ such that the translate $Ta$ lies entirely in $\hat{T}$. So for each $a \in A$, we define
\[
X(a) := \frac{\left|Ta \cap \, \left( \hat{T} \setminus \bigcup_{l=1}^M \bigcup_{c \in \overline{C}_l} \overline{T}_l^c c \right) \right|}{|Ta|} = \frac{\left|Ta \setminus \bigcup_{l=1}^M \bigcup_{c \in \overline{C}_l} \overline{T}_l^c c \right|}{|T|}
\]  
and we treat $X$ as a random variable distributed according to the normalized Haar measure over the set $A$. It follows then from the Chebyshev inequality that
\[
\frac{|\{ a \in A \,|\, X(a) > \sqrt{\varepsilon_1} \}|}{|A|} \leq  \frac{1}{\sqrt{\varepsilon_1}} \int_A \frac{\left|Ta \setminus \bigcup_{l=1}^M \bigcup_{c \in \overline{C}_l} \overline{T}_l^c c \right|}{|A|\cdot |T|} \, da.
\]
Using the Inequalities (\ref{eqn:setA}) and (\ref{eqn:cover1}), we obtain by interchanging integrals (Fubini's Theorem),
\begin{eqnarray} \label{eqn:cover6}
[\dots] &\leq& \frac{1}{\sqrt{\varepsilon_1}} \int_{A} |A|^{-1}|T|^{-1} \int_{G} \one_{Ta \setminus \left( \cup_{l=1}^M \cup_{c \in \overline{C_l}} \overline{T}_l^c c \right)}(g) \, dg \, da \nonumber \\
&=& \frac{1}{\sqrt{\varepsilon_1}} \, |A|^{-1}|T|^{-1} \int_{A} \int_{\hat{T} \setminus \left( \cup_{l=1}^M \cup_{c \in \overline{C_l}} \overline{T}_l^c c \right)} \one_{Ta}(g) \, dg \, da \nonumber \\
&=& \frac{1}{\sqrt{\varepsilon_1}} \, |A|^{-1}|T|^{-1} \int_{\hat{T} \setminus \left( \cup_{l=1}^M \cup_{c \in \overline{C_l}} \overline{T}_l^c c \right)} \left( \int_{A} \one_{Ta}(g) \, da \right)  \, dg \nonumber \\
&\leq& \frac{1}{\sqrt{\varepsilon_1}} \, \frac{|\hat{T} \setminus \cup_{l=l}^M \cup_{c \in \overline{C}_l} \overline{T}_l^c c|\cdot |T|}{(1-\varepsilon_1)|\hat{T}|\cdot |T|}  \nonumber \\
&\leq& 6\, \sqrt{\varepsilon_1}. 
\end{eqnarray}
This shows that for most of the $a$'s (up to a portion of $6\sqrt{\varepsilon_1}$), the corresponding translates $Ta$ are $(1- \sqrt{\varepsilon_1})$-covered by the disjoint union
\begin{eqnarray} \label{eqn:union}
\bigcup_{l=1}^M \bigcup_{c \in \overline{C}_l,\, \overline{T}_l^c c \cap Ta \neq \emptyset} \overline{T}_l^c c.
\end{eqnarray}
It follows from this, as well as from the invariance properties of $T$ that we can impose a second restriction on $\varepsilon_1$ depending on $\varepsilon$ to obtain that up to a portion of $6\, \sqrt{\varepsilon_1}$ of the elements $a \in A$, the translates $Ta$ are $(1-3\varepsilon)$-covered by the union $\bigcup_{i=1}^N T_i\hat{C}_i$, cf.\@ the proof of Theorem 4.7 in \cite{PogorzelskiS-11}, step (6). However, for some elements $a \in A$ and some $c \in \hat{C}_i$, the translate $T_i c$ will have non-trivial intersections with both $Ta$ and its complement. In order to approach this difficulty, we introduce the following notions. Define
\begin{eqnarray*}
I(a,l) &:=&  \{ c \in \overline{C}_l \,|\, \overline{T}^{(c)}_l c \subseteq Ta \} \\
\partial(a,l) &:=& \{ c \in \overline{C}_l \,|\, \overline{T}^{(c)}_l c \cap Ta \neq  \emptyset \wedge \overline{T}^{(c)}_l c \cap (G \setminus Ta) \neq \emptyset \}
\end{eqnarray*}
for $a \in A$ and $1 \leq l \leq M$. Further, we set
\begin{eqnarray*}
C_i(a) := \bigcup_{l=1}^M \bigcup_{c \in I(a,l)} C_i^{l,c}
\end{eqnarray*}
for $a \in A$ and $1 \leq i \leq N(\varepsilon)$, where the sets $C_i^{l,c}$ are those defined in step (3). 

\item We are now in position to define the family $\Lambda_k:= \Lambda$, as well as the corresponding center sets $C_i^{\lambda}(U_k):= C_i^{\lambda}$ for $1 \leq i \leq N$ and $\lambda \in \Lambda = \Lambda_k$ for $k \geq K$. Namely, we obtain $\Lambda$ by erasing from the set $A$ the 'bad' elements, i.e.\@
\[
\Lambda := \{\lambda \in A \,|\, X(\lambda) \leq \sqrt{\varepsilon_1}\}.
\]
Note that by Inequality (\ref{eqn:setA}), we have 
\begin{eqnarray} \label{eqn:cover10}
|\Lambda| \geq (1-6\sqrt{\varepsilon_1})(1-\varepsilon_1)|\hat{T}|. 
\end{eqnarray}
For $\lambda \in \Lambda$, we set
\[
C_i^{\lambda} := \{d \in T\,|\, d\lambda \in C_i(\lambda)\} \subseteq \hat{C}_i \lambda^{-1}
\]
for $1 \leq i \leq N$. An easy computation using the previous step shows that for every $\lambda \in \Lambda$, the set $T$ is $(1-4\varepsilon)$-covered by the union $\bigcup_{i=1}^N T_iC_i^{\lambda}$. 
By construction of the $C_i^{\lambda}$, we have indeed $T_i C_i^{\lambda} \subseteq T$ for all $1\leq i \leq N(\varepsilon)$ and every $\lambda \in \Lambda$. 
This shows property (I) of Definition \ref{defi:UCD}.

\item We still have to show the uniform covering property (II) of Definition \ref{defi:UCD}. To do so, we set $\tilde{U}_k := T \setminus \partial_{\overline{T}_M\overline{T}_M^{-1}}(T)$. It follows from the invariance properties of $T$ that $|\tilde{U}_k| \geq (1-\beta)|T|$. 
Further, fix some Borel set $S \subseteq T$ and fix $1 \leq i \leq N$. We show first that there is a constant $\tilde{C}>0$ such that
\begin{eqnarray} \label{eqn:cont_to_show}
\left| \int_{\Lambda} \operatorname{card}(C_i^{\lambda} \cap S)\, d\lambda - \operatorname{card}({\hat{C}_i}) \cdot |S| \right| \leq \tilde{C} \, \sqrt{\varepsilon_1} \, \frac{|T|}{|T_i|}\, |\hat{T}| + \operatorname{card}(\hat{C}_i)\beta |T|.
\end{eqnarray}   
To do so, note first that 
\begin{eqnarray*}
\int_{\Lambda} \operatorname{card}(C_i^{\lambda} \cap S)\, d\lambda \leq \sum_{c \in \hat{C}_i} \int_{\Lambda} \one_S(c\lambda^{-1})\, d\lambda.
\end{eqnarray*}
One immediately obtains from this that
\begin{eqnarray} \label{eqn:contres1}
\int_{\Lambda} \operatorname{card}(C_i^{\lambda} \cap S)\, d\lambda \leq \operatorname{card}(\hat{C}_i) \cdot |S|.
\end{eqnarray}

On the other hand, we can prove the converse inequality in the following way. Since $T\Lambda \subset \hat{T}$, we have $\Lambda \subset T^{-1}\hat{T}$ and with $\hat{C}_i \lambda^{-1} \cap \tilde{U}_k \subseteq C_i^{\lambda}$, we obtain
\begin{eqnarray*} 
\int_{\Lambda} \operatorname{card}(C_i^{\lambda} \cap S)\,d\lambda &\geq& 
\int_{\Lambda} \operatorname{card}(C_i^{\lambda} \cap S \cap \tilde{U}_k)\, d\lambda\\ 
&=& \int_{T^{-1}\hat{T}} \sum_{c \in \hat{C}_i} \one_{S \cap \tilde{U}_k }(c\lambda^{-1})\,d\lambda - \int_{T^{-1}\hat{T} \setminus \Lambda} \sum_{c \in \hat{C}_i} \one_{S\lambda}(c) \, d\lambda \\
&=& \int_{T^{-1}\hat{T}} \sum_{c \in \hat{C}_i} \one_{(S \cap \tilde{U}_k)^{-1}c}(\lambda) \, d\lambda - \int_{T^{-1}\hat{T} \setminus \Lambda} \operatorname{card}(\hat{C}_i \cap S\lambda) \, d\lambda.  
\end{eqnarray*}
It follows from the $\varepsilon$-disjointness of the basis sets $T_i$ that the maximal number of elements in $\hat{C}_i$ which can belong to some translate $S\lambda$ must be bounded by 
\begin{eqnarray} \label{eqn:comb}
2\, \frac{|\partial_{T_iT_i^{-1}}(T\lambda) \cup T\lambda|}{(1-\varepsilon)|T_i|} \leq 2 \,\frac{1+\varepsilon_1}{1-\varepsilon} \, \frac{|T|}{|T_i|},
\end{eqnarray}
where we also used the $(T_iT_i^{-1}, \varepsilon_1)$-invariance of the set $T$. Hence, we can estimate 
\begin{eqnarray} \label{eqn:cover8}
\int_{\Lambda} \operatorname{card}(C_i^{\lambda} \cap S)\, d\lambda &\geq& \sum_{c \in \hat{C_i}} \int_{T^{-1}\hat{T}}  \one_{(S \cap \tilde{U}_k)^{-1}c}(\lambda)\,d\lambda - 2\, \frac{1+\varepsilon_1}{1-\varepsilon} \, \frac{|T|}{|T_i|}\, |T^{-1}\hat{T} \setminus \Lambda|. 
\end{eqnarray}
Moreover, we obtain
\begin{eqnarray} \label{eqn:cover9}
|T^{-1}\hat{T} \setminus \Lambda| &=& |T^{-1}\hat{T}| - |\Lambda| \nonumber \\
&\leq& |\hat{T} \cup \partial_{TT^{-1}}(\hat{T})| - (1-6\,\sqrt{\varepsilon_1})(1-\varepsilon_1)|\hat{T}| \nonumber \\
&\leq&  [1 + \varepsilon_1 - (1-6\,\sqrt{\varepsilon_1})(1-\varepsilon_1) ]\,|\hat{T}| \nonumber \\
&\leq& 8 \sqrt{\varepsilon_1}\,|\hat{T}|
\end{eqnarray}
by the invariance properties of $\hat{T}$. Now making use of $\varepsilon, \varepsilon_1 < 1/2$ and the estimate in Inequality (\ref{eqn:cover9}), the Inequality (\ref{eqn:cover8}) yields

\begin{eqnarray} \label{eqn:contres2}
\int_{\Lambda} \operatorname{card}(C_i^{\lambda} \cap S)\, d\lambda &\geq& \operatorname{card}(\hat{C_i}) \cdot |S \cap \tilde{U}_k| - 2\, \cdot 4 \, \cdot \frac{|T|}{|T_i|} \cdot 8\sqrt{\varepsilon_1} \,|\hat{T}| \nonumber \\
&\geq& \operatorname{card}(\hat{C_i}) \cdot |S \cap \tilde{U}_k| - 64\,\sqrt{\varepsilon_1} \, \frac{|T|}{|T_i|}\,|\hat{T}| \nonumber \\
&\geq& \operatorname{card}(\hat{C_i}) |S| - \operatorname{card}(\hat{C_i}) \beta |T| - 64\,\sqrt{\varepsilon_1} \, \frac{|T|}{|T_i|}\,|\hat{T}|.
\end{eqnarray}
Note that we used here the fact that $|\tilde{U}_k| \geq (1-\beta)|T|$. 
Together with Inequality (\ref{eqn:contres1}), this shows (\ref{eqn:cont_to_show}) with $\tilde{C}=64$.

\item From the considerations above, we can infer that for $\varepsilon_1$ small enough (a restriction depending on $\varepsilon$),
\begin{eqnarray*} 
\left| \frac{\operatorname{card}(\hat{C}_i)}{|\hat{T}|} - \frac{\operatorname{card}(\hat{C}_i)}{|\Lambda|} \right| \, \frac{|S|}{|T|} &\leq& \left( \frac{1}{(1-\varepsilon_1)(1-6\sqrt{\varepsilon_1})} - 1 \right) \frac{\operatorname{card}(\hat{C}_i)}{|\hat{T}|}\nonumber \\
&\leq& \varepsilon \cdot \gamma_i,
\end{eqnarray*}
where, as above, we set $\gamma_i := \operatorname{card}(\hat{C}_i)/|\hat{T}|$.

It follows then from Inequality (\ref{eqn:cover5}) that 
\begin{eqnarray} \label{eqn:Lambda2}
\left| \frac{\operatorname{card}(\hat{C}_i)}{|\Lambda|} - \frac{\varepsilon(1-\varepsilon)^{N-i}}{|T_i|} \right| \, \frac{|S|}{|T|} < 2\,\gamma_i\cdot \varepsilon + \frac{2\beta}{|T_i|}.
\end{eqnarray}

Hence, by combining the Inequalities (\ref{eqn:cont_to_show}) and (\ref{eqn:Lambda2}) we yield
\begin{eqnarray*}
\Bigg| |\Lambda|^{-1} \int_\Lambda \frac{\operatorname{card}(C_i^{\lambda} \cap S)}{|T|} \, d\lambda &-& \frac{\varepsilon(1-\varepsilon)^{N-i}}{|T_i|} \cdot \frac{|S|}{|T|}  \Bigg| < 64 \frac{\sqrt{\varepsilon_1}}{|T_i|} \frac{|\hat{T}|}{|\Lambda|} + 2\,\gamma_i\cdot \varepsilon + \frac{3\beta}{|T_i|} \\
&\leq& \frac{64 \, \sqrt{\varepsilon_1}}{(1-\varepsilon_1)(1-6\sqrt{\varepsilon_1})|T_i|} + 2\,\gamma_i\cdot \varepsilon + \frac{3\beta}{|T_i|}.
\end{eqnarray*}

So what remains to do is to choose $\varepsilon_1$ small enough such that we have finally proven the theorem. 

\end{enumerate}

\end{proof}

\begin{Corollary} \label{cor:UCD}
Let $G$ be a unimodular amenable group, along with parameters $0 < \varepsilon \leq 1/10$, $N(\varepsilon) := N := \lceil \log(\varepsilon)\log(1-\varepsilon) \rceil$, $0 < \beta, \zeta < 2^{-N}\varepsilon$. Further, denote by $(S_n)$ a nested strong F{\o}lner sequence and choose  $L$ as a compact set containing the unity $\operatorname{id}$ of $G$. Let
\[
\{\operatorname{id}\} \subseteq T_1 \subseteq T_2 \subseteq \dots \subseteq T_N
\]  
be a finite sequence sequence of $(L, \zeta^2)$-invariant compact sets taken from the sequence $(S_n)$ that $\varepsilon$-quasi tile the group according to Definiton \ref{defi:STP}. Then is an integer $J \geq N'$ ($N^{'}:= \min\{n \geq N \,|\, S_n = T_N\}$), as well as number $\delta_0 > 0$ depending on $\varepsilon$, $\beta$ and the basis sets $T_i$ such that for every $(S_JS_J^{-1}, \delta_0)$-invariant set $T$, we can find a family $\Lambda \in \mathcal{B}(G)$, $|\Lambda| < \infty$ of uniform $\varepsilon$-quasi tilings of $T$ with finite center sets $C_i^{\lambda}(T)$ for the basis sets $T_i$ ($1 \leq i \leq N$) such that all the covering properties of Definition \ref{defi:UCD} are satisfied for the set $T$.      
\end{Corollary}

\begin{proof}
This is evident by considering the proof of Theorem \ref{thm:UCD}.
\end{proof}

For the continuous ergodic theorem which is proven in Section \ref{sec:MET}, we will need the concept of a so-called {\em $\varepsilon$-quasi tiling tower}. What we mean here is a uniform family of $\varepsilon$-quasi tilings $\Upsilon \subseteq G$ of a very $TT^{-1}$-invariant set $\hat{T}$ as in Theorem \ref{thm:UCD} such that each quasi tiling $y \in \Upsilon$ of this family generates a uniform family of $\varepsilon$-quasi tilings for the set $T$ coming from $\Lambda \subseteq G$, where this set $\Lambda$ is completely independent of $y \in \Upsilon$.

\begin{Definition} \label{defi:UDT}
Let $G$ be an amenable group. We say that $G$ has the {\em uniform decomposition tower condition} (UDTC) if for every strong F{\o}lner sequence $(U_k)$ in $G$, the following statements hold true. 
\begin{itemize}
\item For each $0 < \varepsilon \leq 1/10$, $N:=N(\varepsilon):= \lceil \log(\varepsilon)/\log(1-\varepsilon) \rceil$, for arbitrary numbers $0 < \beta, \zeta < 2^{-N}\varepsilon$, for every nested F{\o}lner sequence $(S_n)$, and for each compact set $\operatorname{id} \in L \subseteq G$, the group $G$ is $\varepsilon$-quasi tiled according to Definition \ref{defi:STP} by tiling sets
\[
 \{\operatorname{id}\}  \subseteq T_1 \subseteq T_2 \subseteq \dots \subseteq T_N,
\]
where $T_i \in \{S_n \,|\, n \geq i\}$ for $1 \leq i \leq N(\varepsilon)$ and where the latter basis sets are also $(L,\eta^2)$-invariant. 
\item For fixed numbers $\varepsilon, \beta, \zeta$ and for a fixed compact set $\operatorname{id} \in L \subseteq G$, there are numbers $K \in \NN$ and $\eta_0 > 0$ depending on $\varepsilon, \beta$ and the basis sets $T_i$ such that for each $k \geq K$ and for every $0 < \eta < \eta_0$, there is some $(U_kU_k^{-1}, \eta)$-invariant set $\hat{U}_k$, as well as measurable sets $\Lambda_k, \Upsilon_k$ of finite measure, such that \begin{enumerate}[(I)]
\item $U_k\Lambda_k \subseteq \hat{U}_k$ and $(1- \beta)|\hat{U}_k| \leq |\Lambda_k| \leq |\hat{U}_k|$,
\item the set $\hat{U}_k$ is uniformly $\varepsilon$-quasi tiled by a family 
\[
\{\hat{C}_i^{y}(\hat{U}_k) \,|\, y \in \Upsilon_k,\, 1 \leq i \leq N\}
\] 
of finite center sets for the basis sets $T_i$ according to Definition \ref{defi:UCD} (second item), 
\item For each $y \in \Upsilon_k$, the set $U_k$ is uniformly $\varepsilon$-quasi tiled by the family
\[
\{C_i^{y,\lambda}({U}_k) \,|\, \lambda \in \Lambda_k,\, 1 \leq i \leq N\}
\] 
of finite center sets for the basis sets $T_i$ according to Definition \ref{defi:UCD} (second item), where 
\[
\tilde{U}_k \cap \hat{C}_i^{y}(\hat{U}_k)\lambda^{-1} \subseteq  C_i^{y,\lambda}(U_k) \subseteq U_k \cap \hat{C}_i^{y}(\hat{U}_k)\lambda^{-1}
\]
for all $1 \leq i \leq N$ and every $\lambda \in \Lambda_k$, and $\tilde{U}_k \subseteq U_k$ is a set with $|\tilde{U}_k| \geq (1-\beta)|U_k|$.
\end{enumerate}
In this situation, we say that the pair $(\Upsilon_k, \Lambda_k)$ is a {\em uniform decomposition tower} for the set $U_k$ with respect to $(\{T_i\}_{i=1}^{N(\varepsilon)}, (S_n),\varepsilon, \beta)$.
\end{itemize}
\end{Definition}
We will need uniform decomposition towers for the proof of the main theorem of this work. The following existence theorem is a strengthening of Theorem \ref{thm:UCD}

\begin{Theorem}[Uniform Decomposition Tower] \label{thm:UDT}
Each unimodular amenable group satisfies the uniform decomposition tower condition. 
\end{Theorem}

\begin{Remark}
One might think that the existence of uniform decomposition towers follows trivially from a repeated application of Theorem \ref{thm:UCD}. However, since we need the sets $\Upsilon_k$ and $\Lambda_k$ to be {\em independent} from each other, sophisticated arguments are necessary. 
\end{Remark}

\begin{proof}
Let $0 < \varepsilon \leq 1/10$ and $0< \zeta, \beta < 2^{-N}\varepsilon$, as well as a compact set $\operatorname{id} \in L \subseteq G$ be given, where $N:=N(\varepsilon):= \lceil \log(\varepsilon) / \log(1- \varepsilon) \rceil$. Assume further that $(U_k)$ is a strong F{\o}lner sequence. 

Note that by Theorem \ref{thm:STP}, we can find $(L,\zeta^2)$-invariant basis sets
\[
\{\operatorname{id}\} \subseteq T_1 \subseteq \dots \subseteq T_N,
\] 
taken from a nested F{\o}lner sequence $(S_n)$ that $\varepsilon$-quasi tile the group. \\

As before (proof of Theorem \ref{thm:UCD}), we choose $0 < \varepsilon_1 < 1/100$, and at various steps of the proof, we reduce this parameter for our purposes (restrictions depending on $\varepsilon, \beta$ and the $T_i$). Set $M:= \lceil \log(\varepsilon_1)/\log(1-\varepsilon_1) \rceil$. \\

\begin{enumerate}

\item Let $J \in \NN$ and $\delta_0 > 0$ be the parameters that can be found according to Corollary \ref{cor:UCD}. Further, choose a finite sequence of $(S_JS_J^{-1}, \delta_0^2/64)$-invariant sets from the sequence $(S_n)$,  
\[
T_N \subseteq S_J \subseteq \overline{T}_1 \subseteq \overline{T}_2 \subseteq \dots \subseteq \overline{T}_M
\]  
that $\varepsilon_1$-quasi tile the group $G$ according to Definition \ref{defi:STP} for any parameter $0 < \beta_1 < 2^{-M}\varepsilon_1$. Now find some integer number $K \in \NN$ depending on $\varepsilon, \beta$ and the $T_i$ such that for each $k \geq K$, the set $T:=U_k$ is $(\overline{T}_l\overline{T}_l^{-1}, 2^{-l}\varepsilon_1)$-invariant for all $1 \leq l \leq M$. Further, we choose $\hat{T}:= \hat{U}_k$, where $\hat{U}_k$ is a $(U_kU_k^{-1}, \varepsilon_1)$-invariant set that has all mentioned invariance properties of $T=U_k$ and which has the special tiling property with respect to $(\{\overline{T}_l\}_{l=1}^M, (S_n), \varepsilon_1, \beta_1)$ (for instance, take $\hat{T} = \hat{U}_k = U_{\tilde{K}}$ for $\tilde{K} \in \NN$ large enough). We now set $\eta_0 := \varepsilon_1/2$ and by showing all claimed results for $\hat{T}=\hat{U}_k$, we will see that this is in fact appropriate.  \\

\item Following Theorem \ref{thm:STP}, we can fix an $\varepsilon_1$-quasi tiling of $\hat{T}$ with basis sets $\overline{T}_l$ and finite center sets $\overline{C}_l$, $1 \leq l \leq M$. It follows from part (B) of Theorem \ref{thm:STP} that we can actually construct {\em disjoint} translates $\overline{T}^c_l c$ which are still $(S_JS_J^{-1}, \delta_0/2)$-invariant for $1 \leq l \leq M$ and $c \in \overline{C}_l$.

\item Hence by Corollary \ref{cor:UCD}, in each translate $\overline{T}_l^c c$ for $1 \leq l \leq M$ and $c \in \overline{C}_l$, we find a uniform family of $\varepsilon$-quasi tilings with $T_i$-center sets $C_i^{y}(l,c)$, where $y$ is taken from a set $\Upsilon(l,c) \in \mathcal{B}(G)$ of finite measure. It is due to their construction that each of these uniform coverings is induced by some background tiling of a set $\widehat{\overline{T}_l^c c}$ which just has to be invariant enough with respect to $\overline{T}^c_l(\overline{T}^c_l)^{-1}$ for $1 \leq l \leq M$ and $c \in \overline{C}_l$. Hence, with no loss of generality, we can work with {\em one} single compact set $\tilde{T} \subset G$ representing the sets $\widehat{\overline{T}_l^c c}$ for all $1 \leq l \leq M$ and every $c \in \overline{C}_l$. To be precise, we make sure that $\tilde{T}$ is $(\overline{T}_l\overline{T}_l^{-1}, \varepsilon_2)$-invariant and that additionally, it is $(\hat{T}\hat{T}^{-1}, \varepsilon_2)$-invariant for some 
\begin{eqnarray*}
0 < \varepsilon_2 < \left( \frac{\varepsilon_1}{8\, \sum_{l=1}^M \operatorname{card}(\overline{C}_l)} \right)^2     
\end{eqnarray*}
and for {\em all} $1 \leq l \leq M$.


\item 
It follows then from the construction in the proof of Theorem \ref{thm:UCD}, Inequality (\ref{eqn:cover10}), that
\begin{eqnarray} \label{eqn:sizeA4}
|\Upsilon(l,c)| &\geq& (1- \varepsilon_2 - 6\sqrt{\varepsilon_2})|\tilde{T}| \nonumber \\
&\geq& (1 - 7\sqrt{\varepsilon_2})\,|\tilde{T}|  \nonumber \\
&\geq& \frac{1-7\sqrt{\varepsilon_2}}{1+\varepsilon_2}\,|\hat{T}^{-1}\tilde{T}| \nonumber \\
&\stackrel{\varepsilon_2 < 1/100}{\geq}& (1- 8\sqrt{\varepsilon_2})\,|\hat{T}^{-1}\tilde{T}| \nonumber \\
&\geq& \left(1 - \frac{\varepsilon_1}{\sum_{l=1}^M \operatorname{card}(\overline{C}_l)} \right)\, |\hat{T}^{-1}\tilde{T}|
\end{eqnarray}
for every $1 \leq l\leq M$ and all $c \in \overline{C}_l$. \\
Define
\begin{eqnarray*}
\Upsilon_k := \Upsilon := \bigcap_{l=1}^M \bigcap_{c \in \overline{C}_l} \Upsilon(l,c).
\end{eqnarray*}
Since $\Lambda(l,c) \subset \hat{T}^{-1}\tilde{T}$ for all $1 \leq l \leq M$ and every $c \in \overline{C}_l$, it follows from elementary measure theory that for all $1 \leq l \leq M$ and all $c \in \overline{C}_l$,
\begin{eqnarray} \label{eqn:sizeA5}
|\Upsilon| \geq (1-\varepsilon_1)\,|\Upsilon(l,c)|.
\end{eqnarray}

\item We now show the uniform covering property for the set $\hat{T}=U_k$, see the second claim of the theorem, statement (II). \\

To do so, fix a Borel set $\hat{S} \subset \hat{T}$, as well as some $1 \leq i \leq N$. Note that by construction, we have 
\begin{eqnarray} \label{eqn:tower1}
&& \Bigg| |\Upsilon(l,c)|^{-1} \int_{\Upsilon(l,c)} \frac{\operatorname{card}[(\hat{S} \cap \overline{T}_l^{'}(c)c) \cap C_i^{y}(l,c)]}{|\overline{T}^{'}_l(c)|} \, dy - \frac{\varepsilon(1-\varepsilon)^{N-i}}{|T_i|}\, \frac{|\hat{S} \cap \overline{T}^{'}_l(c)|}{|\overline{T}^{'}_l(c)|} \Bigg| \nonumber \\ 
&\quad& \quad < 3 \, \frac{\beta}{|T_i|} + 2\, \gamma_i(l,c) \, \varepsilon
\end{eqnarray}
for all $1 \leq l \leq M$ and for every $c \in \overline{C}_l$, where $\sum_{i=1}^N \gamma_i(l,c)\,|T_i| \leq 2$. \\

Moreover, with the combinatorial argument demonstrated in Inequality (\ref{eqn:comb}) and with the Inequality (\ref{eqn:sizeA5}), it is a straight forward exercise to show

\begin{eqnarray} \label{eqn:tower2}
&& \Bigg|  |\Upsilon(l,c)|^{-1} \int_{\Upsilon(l,c)} \frac{\operatorname{card}[(\hat{S} \cap \overline{T}_l^{'}(c)c) \cap C_i^{y}(l,c)]}{|\overline{T}^{'}_l(c)|} \, dy -  \nonumber \\  
&\,\,& |\Upsilon|^{-1} \int_{\Upsilon} \frac{\operatorname{card}[(\hat{S} \cap \overline{T}_l^{'}(c)c) \cap C_i^{y}(l,c)]}{|\overline{T}^{'}_l(c)|} \, dy \Bigg| < \frac{8\varepsilon_1}{|T_i|} + \frac{4[(1-\varepsilon_1)^{-1}-1]}{|T_i|}. 
\end{eqnarray} 

So combining the Inequalities (\ref{eqn:tower1}) and (\ref{eqn:tower2}) with the triangle inequality, we arrive at
\begin{eqnarray} \label{eqn:tower3}
&& \Bigg| |\Upsilon|^{-1} \int_{\Upsilon} \frac{\operatorname{card}[(\hat{S} \cap \overline{T}_l^{'}(c)c) \cap C_i^{y}(l,c)]}{|\overline{T}^{'}_l(c)|} \, dy - \frac{|\hat{S} \cap \overline{T}^{'}_l(c)c|}{|\overline{T}^{'}_l(c)|} \, \frac{\varepsilon(1-\varepsilon)^{N-i}}{|T_i|} \Bigg| \nonumber \\
&\quad& <  3 \, \frac{\beta}{|T_i|} + 2\, \tilde{\gamma}_i \, \varepsilon + \frac{8\varepsilon_1}{|T_i|} + \frac{4[(1-\varepsilon_1)^{-1}-1]}{|T_i|}
\end{eqnarray} 
for all $1 \leq l \leq M$ and every $c \in \overline{C}_l$, where $\tilde{\gamma}_i := \max_{l,c} \, \gamma_i(l,c)$.

We set 
\begin{eqnarray} \label{eqn:CIY}
\hat{C}_i^{y} := \hat{C}_i^{y}(\hat{T}) := \hat{C}_i^{y}(\hat{U}_k) := \bigcup_{l=1}^M \bigcup_{c \in \overline{C}_l} C_i^{y}(l,c). 
\end{eqnarray}
In this way, for each $y \in \Upsilon$, we get an $\varepsilon$-quasi tiling of $\hat{T}$ with $T_i$-centers $\hat{C}_i^{y}$ (cf.\@ steps (3) to (5) of the proof of Theorem \ref{thm:UCD}). 
Moreover, we are in position to prove the uniform covering property of the family $\Upsilon$. Namely, by disjointness of the $\overline{T}^c_l c$, we have

\begin{eqnarray*} 
&& \Bigg| |\Upsilon|^{-1} \int_{\Upsilon} \frac{\operatorname{card}(\hat{S} \cap \hat{C}_i^{y})}{|\hat{T}|} \, dy - \frac{|\hat{S}|}{|\hat{T}|} \cdot \frac{\varepsilon(1- \varepsilon)^{N-i}}{|T_i|} \Bigg| \\
&\leq& \Bigg| \sum_{l=1}^m\sum_{c \in \overline{C}_l} \frac{|\overline{T}^{c}_l c|}{|\hat{T}|} \, |\Upsilon|^{-1} \int_{\Upsilon} \frac{\operatorname{card}[(\hat{S} \cap \overline{T}_l^{c}c  \cap C_i^{y}(l,c)]}{|\overline{T}^{c}_lc|} \, dy - \\
&& \quad \frac{\Big| \hat{S} \cap \bigcup_{l=1}^M \bigcup_{c \in \overline{C}_l} \overline{T}_l^{c} c \Big|}{|\hat{T}|} \cdot \frac{\varepsilon(1- \varepsilon)^{N-i}}{|T_i|}  \Bigg| + \frac{\Big| \hat{T} \setminus \Big( \bigcup_{l=1}^M \bigcup_{c \in \overline{C}_l} \overline{T}^{c}_l c \Big) \Big|}{|\hat{T}|} \cdot \frac{\varepsilon(1- \varepsilon)^{N-i}}{|T_i|} \\
&\leq& \Bigg| \sum_{l=1}^M \sum_{c \in \overline{C}_l} \frac{|\overline{T}^{c}_l c|}{|\hat{T}|} \, |\Upsilon|^{-1} \int_{\Upsilon} \frac{\operatorname{card}[(\hat{S} \cap \overline{T}_l^{c} c) \cap C_i^{y}(l,c)]}{|\overline{T}^{c}_l c|} \, dy  - \\
&\quad& \sum_{l=1}^M \sum_{c \in \overline{C}_l} \frac{|\overline{T}^c_l c|}{|\hat{T}|} \, \frac{|\hat{S} \cap \overline{T}_l^{c} c|}{|\overline{T}_l^{c}c|} \, \frac{\varepsilon(1- \varepsilon)^{N-i}}{|T_i|} \Bigg| + 2\varepsilon_1 \, \frac{\varepsilon(1- \varepsilon)^{N-i}}{|T_i|} \\
&\leq& \sum_{l=1}^M \sum_{c \in \overline{C}_l} \frac{|\overline{T}^{c}_l c|}{|\hat{T}|} \cdot \Bigg| |\Upsilon|^{-1} \int_{\Upsilon} \frac{\operatorname{card}[(\hat{S} \cap \overline{T}_l^{c} c) \cap C_i^{y}(l,c)]}{|\overline{T}^{c}_l|} \, dy  - \frac{|\hat{S} \cap \overline{T}_l^{c}c|}{|\overline{T}_l^{c}c|} \, \frac{\varepsilon(1- \varepsilon)^{N-i}}{|T_i|}  \Bigg| \\
&\quad& + 2\varepsilon_1 \, \frac{\varepsilon(1- \varepsilon)^{N-i}}{|T_i|} \\
&\stackrel{(\ref{eqn:tower3})}{\leq}& 3 \, \frac{\beta}{|T_i|} + 2\, \tilde{\gamma}_i \, \varepsilon + \frac{10\varepsilon_1}{|T_i|} + \frac{4[(1-\varepsilon_1)^{-1}-1]}{|T_i|}.
\end{eqnarray*}
By making $\epsilon_1$ small enough (depending on $\varepsilon$ and $\beta$), this shows the desired estimate. The remaining properties in statement (II) now follow easily.

\item So what is left to do is to verify the statement (III) of the second item of the theorem. We choose $\Lambda=\Lambda_k$ in exactly the same manner as in the proof of Theorem \ref{thm:UCD}, steps~(6) and (7). Note that $\Lambda$ results from considerations concerning the sets $T,\hat{T}$, and the $\overline{T}^c_l c$ for $1 \leq l \leq M$ and $c \in \overline{C}_l$, but not from the tilings constructed above. Therefore, $\Lambda$ is indeed independent of $\Upsilon=\Upsilon_k$.  

By the above construction, for each $y \in \Upsilon$, the set $\hat{T}$ is $\varepsilon$-quasi tiled by the basis sets $T_i$ with corresponding finite center sets $\hat{C}_i^{y}$ for $1 \leq i \leq N$ and all the translates $T_ic$ ($1\leq i \leq N$, $c \in \hat{C}_i^{y}$) are contained in some translate $\overline{T}_l^{d}d$, where $1 \leq l \leq M$ and $d \in \overline{C}_l$. 
Hence, for fixed $y \in \Upsilon$, we are exactly in the situation of the proof of Theorem \ref{thm:UCD}, steps~(6) and (7) and we define
\begin{eqnarray*}
C_i^{y, \lambda}(U_k):=C_i^{y, \lambda}(T)&:=& \{ d \in T\,|\, d\lambda \in {C}_i^{y}(\lambda) \} \\
&=& T \cap \hat{C}^{y}_i \lambda^{-1}
\end{eqnarray*}
for all $1 \leq i \leq N$ and all $\lambda \in \Lambda$, where
\[
C_i^{y}(\lambda) = \bigcup_{l=1}^M \bigcup_{c \in I(\lambda,l)} C_i^y(l,c) \subseteq \hat{C}_i^y \cap U_k.
\] 
By repeating the steps (8) and (9) of the proof of Theorem \ref{thm:UCD}, we get the desired uniform covering. We set $\tilde{U}_k := U_k \setminus \partial_{\overline{T}_M\overline{T}_M^{-1}}(U_k)$. Indeed, $|\tilde{U}_k| \geq (1-\beta)|U_k|$ and $\hat{C}_i^{y}\lambda^{-1} \cap \tilde{U}_k \subseteq C_i^{y,\lambda}$ for all $\lambda \in \Lambda$. \\ 
Since $y \in \Upsilon$ was arbitrary, this shows the statement (III).  


\end{enumerate}

As the validity of statement (I) follows by construction for $\varepsilon_1$ small enough depending on $\beta$, the theorem is proven. 

\end{proof}

\section{Abstract mean ergodicity} \label{sec:absMET}

In this section, we briefly summarize the classical results concerning the norm convergence of abstract ergodic averages in amenable groups. In light of that, we cite the general abstract mean ergodic theorem from \cite{Greenleaf-73}, which will be used in Section \ref{sec:MET} to derive a mean ergodic theorem for Banach space valued set functions.

\begin{Definition} \label{defi:WMA}
Let $G$ be a second countable Hausdorff group and assume that $Z$ is a Banach space. Let $\tilde{S}$ be a linear subspace of $Z$. We then say that $G$ acts weakly measurably on $\tilde{S}$ via uniformly bounded operators $\{T_g\}_{g \in G}$ if there is a constant $A > 0$ and a map
\begin{eqnarray*}
T: G \times \tilde{S} \rightarrow \tilde{S}: (g,f) \mapsto T_g f
\end{eqnarray*}
with the following properties.
\begin{enumerate}[(i)]
\item $T_g: \tilde{S} \rightarrow \tilde{S}$ is a linear operator for each $g \in G$. 
\item $\|T_g\| \leq A$ for all $g \in G$. 
\item $T_ef = f$ for each $f \in \tilde{S}$, where $e$ is the unit element in G.
\item $T_{g_1}(T_{g_2} f) = T_{g_1g_2} f$ for each $f \in \tilde{S}$ and all $g_1, g_2 \in G$.
\item For each $f \in \tilde{S}$ and every $h \in Z^{*}$, the map
\begin{eqnarray*}
\Phi_{f,h}: G \rightarrow \CC: g \mapsto \langle T_gf, h \rangle_{Z,Z^{*}}
\end{eqnarray*}
is measurable with respect to the Borel $\sigma$-algebras on $G$ and $\CC$ respectively.
\end{enumerate}
Moreover, we set $\operatorname{Fix}(T_G):= \{f \in \tilde{S} \, |\, T_gf = f \mbox{ all } g\in G \}$ as the space of elements in $\tilde{S}$ which remain unchanged under the action of all $g \in G$.
\end{Definition}

\begin{Definition} \label{defi:AEA}
Let $G$ be a second countable, amenable group and assume that $G$ acts weakly measurably on a Banach space $\tilde{S}=Z$ as a family of linear, uniformly bounded operators $\{T_g\}_{g \in G}$ as in Definition \ref{defi:WMA}. Then, if $\{F_n\}$ is a F{\o}lner sequence in $G$, we denote for $f \in Z$ the $n$-th abstract ergodic average $A_nf$ as
\begin{eqnarray*}
A_nf := |F_n|^{-1} \int_{F_n} T_{g^{-1}}f \, dm_L(g), \quad n \in \NN.
\end{eqnarray*}
\end{Definition}

\begin{Remark}
Note that in the first instance, the abstract ergodic averages are only defined in a weak sense, i.e.\@ $A_nf \in Z^{**}$ for $f \in Z$.
However, using mild compactness criteria on th $T_g$-orbit of $f$, one can use standard techniques to showf that we have in fact $A_nf \in Z$ by identification. The interested reader may also refer to  \cite{Pogorzelski-10}, Remark following Theorem 4.3).
\end{Remark}

The following mean ergodic theorem is well known, see e.g.\@ \cite{Greenleaf-73} and \cite{Pogorzelski-10}.

\begin{Theorem} \label{thm:MET}
Let $Z$ be a Banach space. Further, assume that some $\sigma$-compact amenable group $G$ acts weakly measurably on $Z$ as a family $\{T_g\}$ of bounded operators on $Z$ with $A:= \sup_{g \in G} \|T_g\| < \infty$. Moreover, for each $f \in Z$, the convex hull ${co}\{T_gf \,|\, g \in G\}$ is assumed to be relatively weakly compact. Then, given a F{\o}lner sequence $\{F_n\}_n$ in $G$, there is bounded projection $P$ on $Z$ such that
\begin{itemize}
\item $\lim_{n \rightarrow \infty}\|A_nf - Pf\|_Z = 0$ for all $f \in Z$. 
\item $Z = \operatorname{ran}(P) \oplus \operatorname{ker}(P)$.
\item $\operatorname{ran}(P) = \operatorname{Fix}(T_G)$.
\item $\operatorname{ker}(P) = \overline{\operatorname{lin}}\{ f - T_gf \,|\, f \in Z, \, g \in G\}  $.
\end{itemize}
\end{Theorem}

\begin{proof}
See \cite{Greenleaf-73}, Theorems 3.1 and 3.2, as well as \cite{Pogorzelski-10}, Theorems 4.2 and 4.3.
\end{proof}

\begin{Remark}
Note that due to the Banach-Alaoglu Theorem, reflexivity of the Banach space $Z$ is a sufficient condition for the compactness criterion in Theorem \ref{thm:MET}. Moreover, it is also true that for single elements $f \in Z$ the convergence $A_nf$ to some $f^{*} \in Z$ in $Z$-norm holds whenever the $T_g$ are defined on a subspace $\tilde{S} \subseteq Z$ containing $\overline{C}_f :=  \overline{co}\{T_g f \,|\, g \in G \}$ and provided the set $\overline{C}_f$ is compact in the weak topology on $Z$. 
\end{Remark}

\section{A mean ergodic theorem for set functions} \label{sec:MET}

In this section, we prove the main result of this paper which will be a general mean ergodic theorem for an abstract class of mappings. \\
More precisely, we examine functions which map from the set $\mathcal{F}(G)$ as the collection of all precompact (and hence finite-measure) Borel subsets of $G$ in some Banach space $Z$.\\
In particular, the focus will be on {\em almost-additive} functions $F:\mathcal{F}(G) \rightarrow Z$, where the additivity of $F$ can be measured by a so-called {\em boundary term}. Assuming also that there is a constant $C>0$ such that $\|F(Q)\| \leq C\,|Q|$ for all $Q \in \mathcal{F}(G)$ ({\em boundedness}), along with a compactness criterion on particular sets $C_{F,Q}$ arising from $F$ and the elements in $\mathcal{F}(G)$, we prove a general mean ergodic theorem (cf.\@ Theorem~\ref{thm:METSF}). More precisely, it is shown that the limit $\lim_{j \rightarrow \infty} F(U_j)/|U_j|$ exists in the Banach space topology for each strong F{\o}lner sequence $\{U_j\}_{j=1}^{\infty}$ in $G$. This result complements the ergodic Theorem 5.5 for countable groups, proven in \cite{PogorzelskiS-11}. In the latter paper, the authors require the existence of the frequencies of patterns in the group, but no compactness criterion on the Banach space. Note further that for countable groups, both theorems coincide in the periodic case, i.e.\@ if the map $F$ is invariant under group translations. \\

We start by giving the basic definitions. 

\begin{Definition}[Boundary term] \label{defi:BT}
A map $b: \mathcal{F}(G) \rightarrow [0, \infty)$ is called {\em boundary term}
if 
\begin{enumerate}[(i)]
\item it is {\em invariant}, i.e. \@ $b(Qg) = b(Q)$ for every $Q \in \mathcal{F}(G)$ and all $g \in G$,
\item it {\em F{\o}lner vanishes}, i.e.\@ $\lim_{j\rightarrow \infty} \frac{b(U_j)}{|U_j|} =0$ for every F{\o}lner sequence $\{U_j\}$, 
\item it is {\em compatible with unions and intersections}, i.e.\@ $b(Q \cap P) \leq b(Q) + b(P)$, $b(Q \cup P) \leq b(Q) + b(P)$ and $b(Q \setminus P) \leq b(Q) + b(P)$ for all $Q,P \in \mathcal{F}(G)$.
\end{enumerate}
\end{Definition} 

\begin{Definition}[$b$-almost additive function] \label{defi:BAAF}
A map $F: \mathcal{F}(G) \rightarrow Z$ is called $b$-almost additive for some boundary term $b:\mathcal{F}(G) \rightarrow [0, \infty)$ if
\begin{enumerate}[(i)]
\item $F$ is {\em bounded}, i.e.\@ there exists some constant $C > 0$ such that
\begin{eqnarray*}
C = \sup_{Q \in \mathcal{F}(G)} \frac{\|F(Q)\|_Z}{|Q|} < \infty.
\end{eqnarray*}
\item $F$ is {\em almost additive with boundary term $b$}, 
i.e.\@ 
\begin{eqnarray*}
\left\| F(Q) - \sum_{k=1}^m F(Q_k) \right\|_Z \leq \sum_{k=1}^m b(Q_k)
\end{eqnarray*}
for any disjoint union $Q = \cup_k Q_k$ of sets in $\mathcal{F}(G)$.
\end{enumerate}
\end{Definition}

Note that as usual, we will have to deal with $\varepsilon$-disjoint unions of sets rather than with disjoint unions. Morever, in the case of continuous groups, one cannot expect the mapping $b$ to be bounded. For instance, for Cantor sets in $\RR$ with zero Lebesgue measure, its boundary with respect to any ball of positive radius has positive (Lebesgue) measure. Hence, it is convenient for our purposes to introduce the concept of {\em tiling-admissible} boundary terms $b$. Those functions possess certain boundedness properties for the sets arising from $\varepsilon$-quasi tilings.  

\begin{Definition}
For $0 < \varepsilon < 1/10$, we call a set $\mathcal{C}$ consisting of finite, $\varepsilon$-disjoint families of sets in $\mathcal{F}(G)$ an $\varepsilon$-admissible collection for some boundary term $b$ with constant $\tilde{D} > 0$ if for each such family $\{Q_k\}_{k=1}^m$ in $\mathcal{C}$, one can find a family $\{\overline{Q}_k\}_{k=1}^m$ of pairwise disjoint sets with 
\begin{itemize}
\item $|\overline{Q}_k| \geq (1-\varepsilon)|Q_k|$,
\item $b(Q_k) \leq \tilde{D} \, |Q_k|$,
\item $b(\overline{Q}_k)\leq  \tilde{D} \, ( b(Q_k) + \varepsilon|Q_k|)$ 
\end{itemize}
for $1 \leq k \leq m$. 
\end{Definition}

\begin{Definition}
Let $b: \mathcal{F}(G) \rightarrow [0,\infty)$ be a boundary term. We say that $b$ is tiling-admissible if there is a constant $\tilde{D} > 0$ such that for all $0 < \varepsilon < 1/10$, there is a finite sequence for compact subsets of $G$, 
\[
\{e \} \subseteq T_1^{\varepsilon} \subseteq T_2^{\varepsilon} \subseteq \dots \subseteq T_{N(\varepsilon)}^{\varepsilon} 
\] 
such that each $\varepsilon$-quasi tiling coming from a uniform decomposition tower (cf.\@ Theorem \ref{thm:UDT}) with basis sets $T_i^{\varepsilon}$ ($1 \leq i \leq N$) is $\varepsilon$-admissible with constant $\tilde{D}$.
\end{Definition}

The following proposition shows that the canonical boundary terms $b(Q):= D\,|\partial_L(Q)|$ are indeed tiling-admissible.  

\begin{Proposition} \label{prop:tiladmis}
Let $D > 0$ be an arbitrary constant. Then for every compact set $L \subseteq G$ with $\operatorname{id} \in L$, the boundary term 
\[
b: \mathcal{F}(G) \rightarrow [0, \infty): b(Q):= D\, |\partial_L(Q)|
\]
is tiling-admissible. 
\end{Proposition}

\begin{proof}
Note that since the basis sets $T_i$ ($1 \leq i \leq N(\varepsilon)$) of any $\varepsilon$-quasi tiling coming from a uniform decomposition tower for some set $T \subseteq G$ are chosen from a F{\o}lner sequence $(S_n)$, we conclude that the boundedness $b(T_ic) \leq \overline{D}\,|T_ic|$ must hold for the constant $\overline{D}:= \sup_{n \in \NN} b(S_n)/|S_n| < \infty$, where $c \in C_i^T$ and the sets $C_i^T$ stand for the centers of the basis sets $T_i$ for $1 \leq i \leq N(\varepsilon)$. \\
The statement (B) of Theorem \ref{thm:STP} and Theorem \ref{thm:UDT} make sure that we can choose the $T_i$ in such a way that one can construct disjoint tiles $T_i^c c$ with $|T_i^c c| \geq (1-\varepsilon)|T_ic|$ and
\[
|\partial_L(T_i^c c)|  \leq |\partial_L(T_ic)| + \varepsilon\,|T_ic| 
\] 
 for $1 \leq i \leq N(\varepsilon)$ and $c \in C_i^T$. 
  Hence, we have proven that $b$ is tiling-admissible with the constant $\tilde{D}:= \max\{1;\overline{D};D\}$.
\end{proof}

With these concepts at hand, we are now able to derive an error estimate for $\varepsilon$-disjoint unions. 

\begin{Proposition} \label{prop:almostadd}
Assume that $F:  \mathcal{F}(G) \rightarrow Z$ is almost additive with boundary term $b:\mathcal{F}(G) \rightarrow [0, \infty)$. 
Further, let $0 < \varepsilon < 1/10$ and denote by $\mathcal{C}$ an $\varepsilon$-admissible collection for $b$ with constant $\tilde{D}$. Then if $\{Q_k\}_{k=1}^m$ is an element in $\mathcal{C}$ such that $\cup_k Q_k \subseteq Q$ and  $|\cup_k Q_k| \geq \alpha\,|Q|$ for some parameter $0 < \alpha \leq 1$, then the following error estimate holds true.
\begin{eqnarray*}
\left\| F(Q) - \sum_{k=1}^m F(Q_k) \right\|_Z \leq C\,(2\varepsilon + 1 - (1-\varepsilon)\alpha)\,|Q| + 10\tilde{D}\varepsilon\,|Q| + b(Q) + (5\tilde{D} + 1) \, \sum_{k=1}^m b(Q_k),
\end{eqnarray*}
where $C$ is the boundedness constant for $F$.  
\end{Proposition}

\begin{proof}
Since $\mathcal{C}$ is $b$-admissible with constant $\tilde{D}$, for each $1 \leq k \leq m$, one finds a set $\overline{Q}_k \subseteq Q_k$ such that the $\overline{Q}_k$ are pairwise disjoint and $b(\overline{Q}_k) \leq \tilde{D} \, b(Q_k) + \tilde{D}\,\varepsilon\,|Q_k|$ for all $1 \leq k \leq m$.
An easy calculation shows that the fact that $Q$ is $\alpha$-covered by the $Q_k$ implies
\begin{eqnarray} \label{eqn:cov}
\left| \bigcup_{k=1}^m \overline{Q}_k \right| \geq (1-\varepsilon)\,\alpha\,|Q|.
\end{eqnarray} 

By the triangle inequality, we get
\begin{eqnarray*}
\left\| F(Q) - \sum_{k=1}^m F(Q_k)  \right\|_Z &\leq& \left\| F(Q) - F\left( \bigcup_{k=1}^m \overline{Q}_k \right) \right\|_Z + \left\| F\left( \bigcup_{k=1}^m \overline{Q}_k \right) - \sum_{k=1}^m F(\overline{Q}_k)  \right\|_Z + \\
& & \quad \sum_{k=1}^m \left\| F(\overline{Q}_k) - F(Q_k) \right\|_Z.
\end{eqnarray*} 
For the first expression, we obtain from the $b$-almost additivity of $F$ that
\begin{eqnarray*}
\left\| F(Q) - F\left( \bigcup_{k=1}^m \overline{Q}_k \right) \right\|_Z &\leq& b\left( \bigcup_{k=1}^m \overline{Q}_k \right) + b\left( Q \setminus \bigcup_{k=1}^m \overline{Q}_k \right) + \left\| F\left( Q \setminus \bigcup_{k=1}^m \overline{Q_k} \right) \right\|_Z.
\end{eqnarray*}
Since $|Q \setminus \cup_k \overline{Q}_k| \leq (1-(1-\varepsilon)\alpha)|Q|$ and as $b$ is compatible with unions, we obtain by using boundedness and Inequality \ref{eqn:cov}
\begin{eqnarray*}
\left\| F(Q) - F\left( \bigcup_{k=1}^m \overline{Q}_k \right) \right\|_Z &\leq& 2\, \sum_{k=1}^m b(\overline{Q}_k) + b(Q) + C \,(1-(1-\varepsilon)\alpha) \, |Q|.
\end{eqnarray*}
Moreover, by the fact that $b(\overline{Q}_k) \leq \tilde{D}\,(b(Q_k) + \varepsilon |Q_k|)$ for $1 \leq k \leq m$, 
\begin{eqnarray*}
\left\| F(Q) - F\left( \bigcup_{k=1}^m \overline{Q}_k \right) \right\|_Z &\leq& 2\tilde{D} \, \sum_{k=1}^m b(Q_k) + b(Q) + C \,(1-(1-\varepsilon)\alpha) \, |Q| + 4\tilde{D}\varepsilon\,|Q|.
\end{eqnarray*}

For the second expression, we use the disjointness of the $\overline{Q}_k$ to get
\begin{eqnarray*}
\left\|  F\left( \bigcup_{k=1}^m \overline{Q}_k \right) - \sum_{k=1}^m F(\overline{Q}_k) \right\|_Z &\leq& \sum_{k=1}^m b(\overline{Q}_k).
\end{eqnarray*}
By the considerations above and with $\varepsilon$-disjointness of the $Q_k$, we arrive at
\begin{eqnarray*}
\left\|  F\left( \bigcup_{k=1}^m \overline{Q}_k \right) - \sum_{k=1}^m F(\overline{Q}_k) \right\|_Z &\leq& \tilde{D} \, \sum_{k=1}^m b(Q_k) + 2\tilde{D}\varepsilon\,|Q|.
\end{eqnarray*}

For the third expression, we compute similarly as before,
\begin{eqnarray*}
\|F(Q_k) - F(\overline{Q}_k)\|_Z &\leq& b(\overline{Q}_k) + b(Q_k \setminus \overline{Q}_k) + \|F(Q_k \setminus \overline{Q}_k)\|_Z \\ 
&\leq& b(Q_k) + 2 \, b(\overline{Q}_k) + C\varepsilon \, |Q_k|
\end{eqnarray*}
for $1 \leq k \leq m$. Taking sums, one obtains with the previous considerations that 
\begin{eqnarray*}
\sum_{k=1}^m \|F(Q_k) - F(\overline{Q}_k)\|_Z &\leq& (2\tilde{D} + 1)\,\sum_{k=1}^m b(Q_k) + 2C\, \varepsilon |Q| + 4\tilde{D}\varepsilon\,|Q|.
\end{eqnarray*}
Summing the partial results up, this proves the claim.
\end{proof}

We are now in position to state and prove the main theorem of this section.

\begin{Theorem}[Mean ergodic theorem for set functions] \label{thm:METSF}
Let $G$ be a unimodular group and assume that $\{U_j\}$ a strong F{\o}lner sequence in $G$. Let $(Z,\|\cdot\|_Z)$ be a Banach space and assume that $\{T_g\}_{g \in G}$ is a family of linear, uniformly bounded operators acting weakly measurably on $Z$. \\
Further, denote by $b$ a tiling-admissible boundary term defined on $\mathcal{F}(G)$ 
and consider a bounded (constant $C$), $b$-almost additive mapping
\begin{eqnarray*}
F: \mathcal{F}(G) \rightarrow (Z, \| \cdot \|_Z)
\end{eqnarray*}
with the additional property that  
for every $Q \in \mathcal{F}(G)$, the set $C_{F,Q}:= \operatorname{co}\{F(Qg) \,|\, g \in G\}$ is relatively weakly compact in $Z$. \\
Then, if the action of the $T_g$ is compatible with $F$ in the sense that 
\[
T_gF(Q) = F(Qg^{-1})
\]
for $Q \in \mathcal{F}(G)$ and $g \in G$,
the following assertions hold true.
\begin{enumerate}[(A)]
\item For each $Q \in \mathcal{F}(G)$, the limit
\begin{eqnarray*}
S(Q) := \mbox{Z-}\lim_{j \rightarrow \infty} |U_j|^{-1} \int_{U_j} F(Qg) \, dm_L(g)
\end{eqnarray*}
exists in $Z$.
\item If for each $\varepsilon > 0$ and $N(\varepsilon):= \lceil \log(\varepsilon)/\log(1- \varepsilon) \rceil$, we have an 
$\varepsilon$-quasi tiling $\{T_i^{\varepsilon}\}_{i=1}^{N(\varepsilon)}$ of the group as in Definition \ref{defi:STP} with $0 < \beta < 2^{-N(\varepsilon)}\varepsilon$, then the following limits exist in $Z$ and are equal:
\begin{eqnarray*}
\overline{F} := \lim_{j \rightarrow \infty} \frac{F(U_j)}{|U_j|} = \lim_{\varepsilon \rightarrow 0} \sum_{i=1}^{N(\varepsilon)} \eta_i(\varepsilon) \, \frac{S(T_i^{\varepsilon})}{|T_i^{\varepsilon}|},
\end{eqnarray*}
where we have $\eta_i(\varepsilon):= \varepsilon(1-\varepsilon)^{N(\varepsilon) - i}$ for $1 \leq i \leq N(\varepsilon)$.
\item The limit $\overline{F}$ is a $T_g$-fixed point, i.e.\@ for all $g \in G$, we have 
\[
T_g\overline{F} = \overline{F}. 
\]
\end{enumerate}
\end{Theorem}

\begin{proof}
For the proof of statement (A), let $Q \in \mathcal{F}(G)$. With the remark following Theorem \ref{thm:MET}, the claim now follows from the relative weak compactness of $C_{F,Q}$, as well as from the fact that $T_{g^{-1}}F(Q)= F(Qg)$ for all $g \in G$.  \\

For the proof for (B), we fix $0 < \varepsilon < 1/10$ and $\beta:= 2^{-N(\varepsilon)}\varepsilon$ and we find $j_0(\varepsilon) \in \NN$ such that for each $j \geq j_0(\varepsilon)$, the set $U_j$ satisfies the invariance condition given by Theorem \ref{thm:UDT}. Further, we set
\begin{eqnarray*}
\Delta(j,\varepsilon) := \left\| \frac{F(U_j)}{|U_j|} - \sum_{i=1}^{N(\varepsilon)} \eta_i(\varepsilon)\, \frac{S(T_i^{\varepsilon})}{|T_i^{\varepsilon}|} \right\|_Z
\end{eqnarray*}
for $\varepsilon > 0$ and $j \in \NN$. In the following, we fix $j \geq j_0(\varepsilon)$.  With $0 < \eta < \eta_0$, where $\eta_0$ is taken from Theorem \ref{thm:UDT} as well, we can find some very $(U_jU_j^{-1}, \eta)$-invariant set $\hat{U}_j$ along with a uniform decomposition tower $(\Upsilon, \Lambda)$ with basis sets $T_i^{\varepsilon}$, $(1 \leq i \leq N(\varepsilon))$ and finite center sets $\hat{C}^{y}_i$ for $\hat{U}_j$, where $y \in \Upsilon$, and $C_i^{y, \lambda}$ for $U_j$, where $\lambda \in \Lambda$, respectively. With no loss of generality, we assume that all the $T_i^{\varepsilon}$ are taken from a subsequence $\{S_{n_k}\}_{k=1}^{\infty}$ of a strong F{\o}lner sequence such that the expressions $b(S_{n_k})/|S_{n_k}|$ converge to zero monotonically as $k \rightarrow \infty$. Additionally, we make sure that $T_i^{\varepsilon} \in \{S_{n_l} \,|\, l \geq i \}$ for all $1 \leq i \leq N$.  \\

We will show that $\lim_{\varepsilon\rightarrow 0} \lim_{j \rightarrow \infty} \Delta(j,\varepsilon) = 0$. So combining the construction of the uniform decomposition tower (Theorem \ref{thm:UDT}, statement (III)) for $(U_j, \hat{U}_j)$ with the triangle inequality, we arrive at
\begin{eqnarray}  \label{eqn:ALLEE}
\Delta(j, \varepsilon) \leq D_1(j, \varepsilon) + D_2(j, \varepsilon) + D_3(j, \varepsilon) + D_4(j, \varepsilon) + D_5(j,\varepsilon)
\end{eqnarray}
with
\begin{eqnarray*}
D_1(j, \varepsilon) := \left\| \frac{F(U_j)}{|U_j|} - \sum_{i=1}^{N(\varepsilon)} |\Upsilon|^{-1}|\Lambda|^{-1} \int_{\Upsilon} \int_{\Lambda} \sum_{c \in C_i^{y,\lambda}} \frac{F(T_i^{\varepsilon}c)}{|U_j|} \, d\lambda\,dy \right\|_Z,
\end{eqnarray*}
\begin{eqnarray*}
D_2(j, \varepsilon) := \sum_{i=1}^{N(\varepsilon)} |\Upsilon|^{-1} |\Lambda|^{-1} \int_{\Upsilon} \int_{\Lambda} \left\| \sum_{c \in \hat{C}_i^y \lambda^{-1} \cap U_j} \frac{F(T_i^{\varepsilon}c)}{|U_j|} - \sum_{c \in C_i^{y,\lambda}} \frac{F(T_i^{\varepsilon}c)}{|U_j|} \right\| \, d\lambda\, dy,
\end{eqnarray*}
\begin{eqnarray*}
D_3(j, \varepsilon) := \left\| \sum_{i=1}^{N(\varepsilon)} |\Upsilon|^{-1}|\Lambda|^{-1} \int_{\Upsilon} \left( \sum_{c \in \hat{C}_i^{y} \cap U_j\lambda} \, \int_{U_j \setminus c\Lambda^{-1}} \frac{T_{\lambda^{-1}}F(T_i^{\varepsilon})}{|U_j|} \, d\lambda \right) \, dy \right\|_Z,
\end{eqnarray*}
\begin{eqnarray*}
D_4(j, \varepsilon) := \left\| \sum_{i=1}^{N(\varepsilon)} \left( |\Upsilon|^{-1} \int_{\Upsilon} \frac{\operatorname{card}(\hat{C}_i^{y})}{|\Lambda|} \, dy \right) \, \left( \int_{U_j} \frac{T_{\lambda^{-1}}F(T_i^{\varepsilon})}{|U_j|} \, d\lambda - S(T_i^{\varepsilon}) \right) \right\|_Z
\end{eqnarray*}
and
\begin{eqnarray*}
D_5(j, \varepsilon) := \left\| \sum_{i=1}^{N(\varepsilon)} \left( |\Upsilon|^{-1} \int_{\Upsilon} \frac{\operatorname{card}(\hat{C}_i^{y})}{|\Lambda|} \, dy \right)\, S(T_i^{\varepsilon}) - \sum_{i=1}^{N(\varepsilon)} \eta_i(\varepsilon) \, \frac{S(T_i^{\varepsilon})}{|T_i^{\varepsilon}|} \right\|_Z.
\end{eqnarray*}
We will now give relevant estimates for these expressions.
Since $U_j$ is $\alpha:= (1-4\varepsilon)$-covered by $\varepsilon$-disjoint translates $\{T_ic\}$, $1 \leq i \leq N$, $c \in C_i^{y, \lambda}$ for each $y \in \Upsilon$ and each $\lambda \in \Lambda$ and by the fact that the boundary term $b$ is tiling admissible for some constant $\tilde{D} \geq 1$, it follows from Proposition \ref{prop:almostadd} that for every $j \geq j_0(\varepsilon)$
\begin{eqnarray*}
D_1(j, \varepsilon) &\leq& (7C + 10\tilde{D}) \, \varepsilon + \frac{b(U_j)}{|U_j|}\\
&\quad& \quad  + (5 \tilde{D} + 1)\, \sum_{i=1}^{N(\varepsilon)} \left( |\Upsilon|^{-1}|\Lambda|^{-1} \int_{\Upsilon} \int_{\Lambda} \frac{\operatorname{card}(C_i^{y, \lambda})}{|U_j|} \, d\lambda\, dy \right) \cdot b(T_i^{\varepsilon}).
\end{eqnarray*}
With Theorem \ref{thm:UDT} and the boundedness of $b$ for the $T_i^{\varepsilon}$, this yields
\begin{eqnarray*}
D_1(j, \varepsilon) &\leq& (7C + 10\tilde{D})\, \varepsilon + \frac{b(U_j)}{|U_j|} \\
& &  \quad + (5\tilde{D} + 1)\, \sum_{i=1}^{N(\varepsilon)} \left( \eta_i(\varepsilon) \, b(T_i^{\varepsilon}) + \tilde{D}\tilde{\gamma}_i|T_i^{\varepsilon}|\, \varepsilon + 4\tilde{D}\, \beta \right)
\end{eqnarray*}
such that with the triangle inequality, as well as with the facts that $\sum_{i} \tilde{\gamma}_i|T_i^{\varepsilon}| \leq 2$ and $\lim_{j\rightarrow \infty} b(U_j)/|U_j| = 0$, one obtains
\begin{eqnarray} \label{eqn:ESTTD1}
\limsup_{j \rightarrow \infty} D_1(j, \varepsilon) &\leq& (7C + 10\tilde{D})\,\varepsilon + \limsup_{j \rightarrow \infty} \frac{b(U_j)}{|U_j|} \nonumber \\
&\quad& \quad +  \, (5\tilde{D} + 1)  \sum_{i=1}^{N(\varepsilon)} \eta_i(\varepsilon) \, b(T_i^{\varepsilon}) + 2\tilde{D}(5\tilde{D} + 1) \, \varepsilon + 8\tilde{D}(5 \tilde{D} + 1) \, \varepsilon \nonumber \\
&\leq& (7C + 10\tilde{D} + 10\tilde{D}(5\tilde{D} + 1)) \, \varepsilon \nonumber \\
&\quad& \quad + \,  (5\tilde{D} + 1) \sum_{i=1}^{N(\varepsilon)} \varepsilon(1-\varepsilon)^{N(\varepsilon) - i} \cdot \frac{b(T_i^{\varepsilon})}{|T_i^{\varepsilon}|} \nonumber \\
&\leq& (7C + 10\tilde{D} + 10\tilde{D}(5\tilde{D} + 1)) \, \varepsilon \nonumber \\
&\quad& \quad + \, (5\tilde{D} + 1) \sum_{i=1}^{N(\varepsilon)} \varepsilon(1-\varepsilon)^{N(\varepsilon) - i} \cdot \frac{b(S_{n_i})}{|S_{n_i}|}.
\end{eqnarray}

We continue with the estimate for $D_2(j, \varepsilon)$. It follows from the property (III) of the definition of the uniform decomposition tower that there is a 
set $\tilde{U}_j \subseteq U_j$ with $|\tilde{U}_j| \geq (1-\beta)|U_j|$ and 
\begin{eqnarray*}
\tilde{U}_j \cap \hat{C}_i^y \lambda^{-1} \subseteq C_i^{y,\lambda} \subseteq U_j \cap \hat{C}_i^y \lambda^{-1}
\end{eqnarray*}
for $1 \leq i \leq N(\varepsilon)$, $y \in \Upsilon$ and $\lambda \in \Lambda$. 
It follows from the construction given in the proof of Theorem~\ref{thm:UDT} that one can choose
\[
\tilde{U}_j := U_j \setminus \partial_{\overline{T}_M\overline{T}_M^{-1}}(U_j)
\]
for some auxiliary quasi tiling $\overline{T}_l$, $1\leq l \leq M$.
Clearly, there is no loss in generality in assuming that the set $U_j$ is 
$(\overline{T}_M\overline{T}_M^{-1}T_N^{\varepsilon}T_N^{\varepsilon\,-1},\beta)$-invariant 
(if not, choose $j$ larger). It follows from this that the set $\tilde{U}_j$ is actually
$(T_N^{\varepsilon}T_N^{\varepsilon\,-1}, 4\beta)$-invariant. Thus, one obtains
\[
\frac{|\partial_{T_i^{\varepsilon}T_i^{\varepsilon\,-1}}(U_j \setminus \tilde{U}_j)|}{|U_j|} < 5\,\beta
\]
for all $1 \leq i \leq N(\varepsilon)$, as well as for large enough $j$.   
Now by the triangle inequality and the boundedness of $F$, we obtain
\begin{eqnarray} \label{eqn:est2soon}
D_2(j, \varepsilon) &\leq& \sum_{i=1}^{N(\varepsilon)} |\Upsilon|^{-1} |\Lambda|^{-1} \int_{\Upsilon} \int_{\Lambda} \sum_{c \in \hat{C}_i^y \lambda^{-1} \cap (U_j \setminus \tilde{U}_j)} \frac{\| F(T_i^{\varepsilon}c) \|}{|U_j|} \, d\lambda \, dy \nonumber \\
&\leq& C\, \sum_{i=1}^{N(\varepsilon)} |\Upsilon|^{-1} |\Lambda|^{-1} \int_{\Upsilon} \int_{\Lambda} \frac{\operatorname{card}(\hat{C}_i^y \cap (U_j \setminus \tilde{U}_j)\lambda)}{|U_j|} \,|T_i^{\varepsilon}|\, d\lambda \, dy.
\end{eqnarray}
For a moment, fix $\lambda \in \Lambda$, as well as $y \in \Upsilon$. 
By $\varepsilon$-disjointness of the translates $T_i^{\varepsilon}c$, $(c\in \hat{C}_i^y)$, we deduce that there are at most
\[
\frac{|(U_j \setminus \tilde{U}_j) \cup \partial_{T_i^{\varepsilon}T_i^{\varepsilon\,-1}}(U_j \setminus \tilde{U}_j)|}{(1-2\varepsilon)\,|T_i^{\varepsilon}|}
\]
many translates $T_i^{\varepsilon}c$ with $c \in (U_j\setminus \tilde{U}_j)\lambda$. Therefore,  
\begin{eqnarray*}
\frac{\operatorname{card}(\hat{C}^y_i \cap (U_j \setminus \tilde{U}_j)\lambda)}
{|U_j|}\,|T_i^{\varepsilon}| &\leq& \frac{|U_j \setminus \tilde{U}_j \cup 
\partial_{T_i^{\varepsilon}T_i^{\varepsilon\, -1}}
(U_j \setminus \tilde{U}_j)|}{|T_i^{\varepsilon}| \, 
|U_j| \, (1-2\varepsilon)}\, |T_i^{\varepsilon}| \\ 
&\stackrel{\varepsilon < 1/4}{\leq}& 2\, (\beta + 5\beta) = 12 \beta  
\end{eqnarray*}
for every $1 \leq i \leq N(\varepsilon)$, large enough $j$, 
every $y \in \Upsilon$ and each $\lambda \in \Lambda$.
For the above estimate we also used that $U_j$ is $(T_i^{\varepsilon}T_i^{\varepsilon\,-1}, \beta)$-invariant 
and that $|\tilde{U}_j| \geq (1-\beta)|U_j|$. With the simple observation 
that $\beta N(\varepsilon) < 2\varepsilon$, we deduce from inequality~\eqref{eqn:est2soon} that
\begin{eqnarray} \label{eqn:ESTTD1.5}
\limsup_{j\to \infty} D_2(j, \varepsilon) \leq C\, \sum_{i=1}^{N(\varepsilon)} 12 \beta   \leq 24C \varepsilon.
\end{eqnarray}

For a good estimate for $D_3(j, \varepsilon)$, the concept of a uniform decomposition tower is crucial. Hence, due to the boundedness of $F$, we have for $j \geq j_0(\varepsilon)$ that

\begin{eqnarray*}
D_3(j, \varepsilon) &\leq& \sum_{i=1}^{N(\varepsilon)} |\Upsilon|^{-1} |\Lambda|^{-1} \int_{\Upsilon} \left( \sum_{c \in \hat{C}_i^{y}} |U_j|^{-1} \int_{U_j \setminus c\Lambda^{-1}} \|T_{u^{-1}}F(T_i^{\varepsilon})\|_Z \, du \right) \, dy \\
&\leq& C\, \sum_{i=1}^{N(\varepsilon)} |\Upsilon|^{-1} |\Lambda|^{-1} \int_{\Upsilon} \left( \sum_{c \in \hat{C}_i^{y}} |U_j|^{-1} \int_{U_j} \one_{\hat{U}_j \setminus u\Lambda} (c) \cdot |T^{\varepsilon}_i| \, du \right) \, dy \\
&\leq& C \, \sum_{i=1}^{N(\varepsilon)} |U_j|^{-1} \int_{U_j} |\Upsilon|^{-1} \left(  \int_{\Upsilon} \frac{\operatorname{card}((\hat{U}_j \setminus u\Lambda) \cap \hat{C}_i^{y})}{|\Lambda|} \, dy \right)   \,du \cdot |T_i^{\varepsilon}| \\
&\leq& 2C \, \sum_{i=1}^{N(\varepsilon)} |U_j|^{-1} \int_{U_j} |\Upsilon|^{-1} \left(  \int_{\Upsilon} \frac{\operatorname{card}((\hat{U}_j \setminus u\Lambda) \cap \hat{C}_i^{y})}{|\hat{U}_j|} \, dy \right)   \,du \cdot |T_i^{\varepsilon}| ,
\end{eqnarray*}
where the last inequality is due to Theorem \ref{thm:UDT}, statement (I) and $\beta << 1/2$. 

By the properties (I) and (II) of Theorem \ref{thm:UDT}, it follows with the boundedness of $F$ and with with $\beta < 2^{-N(\varepsilon)}\varepsilon$, as well as with $\sum_{i} \tilde{\gamma}_i |T_i^{\varepsilon}| \leq 2$ that
\begin{eqnarray*}
D_3(j, \varepsilon, \omega) &\leq& 2 C \, \sum_{i=1}^{N(\varepsilon)} \left( \eta_i(\varepsilon) \cdot \frac{\beta}{|T_i^{\varepsilon}|} + 4\frac{\beta}{|T_i^{\varepsilon}|} + 2\tilde{\gamma}_i \, \varepsilon \right) \, |T_i^{\varepsilon}| \\
&\leq& 2\, C \, (\beta + 8\varepsilon + 4\varepsilon) \\
&\leq& 26\, C \, \varepsilon.
\end{eqnarray*}
Note that here we used that $|\hat{U}_j \setminus u\Lambda| = |\hat{U}_j| - |\Lambda| \leq \beta\, |\hat{U}_j|$ for all $u \in U_j$, cf.\@ the statement (I) of Theorem \ref{thm:UDT}.\\
Consequently, 
\begin{eqnarray} \label{eqn:ESTTD2}
\limsup_{j \rightarrow \infty} D_3(j, \varepsilon) \leq 26\varepsilon \, C.
\end{eqnarray}

For $D_4(j,\varepsilon)$, it is a direct consequence of claim (II) of Theorem \ref{thm:UDT} with $\hat{S}=\hat{T}$ and with $|\Lambda| \geq (1-\beta)|\hat{U}_j|$ $(\beta << 1/2)$ that 
\begin{eqnarray*}
D_4(j, \varepsilon) &\leq& \sum_{i=1}^{N(\varepsilon)} (1-\beta)^{-1} \left( \frac{\eta_i(\varepsilon)}{|T_i^{\varepsilon}|} + \frac{4\beta}{|T_i^{\varepsilon}|} + 2\,\tilde{\gamma}_i \,\varepsilon \right) \, \left\| \int_{U_j} \frac{T_{\lambda^{-1}}F(T_i^{\varepsilon})}{|U_j|} \, d\lambda - S(T_i^{\varepsilon})  \right\|_Z \\
&\leq& 2\,\sum_{i=1}^{N(\varepsilon)} \Bigg[ \frac{\eta_i(\varepsilon)}{|T_i^{\varepsilon}|} \,\left\| |U_j|^{-1} \int_{U_j} T_{\lambda^{-1}}F(T_i^{\varepsilon}) \, d\lambda - S(T_i^{\varepsilon})  \right\|_Z + \\
& & \quad \left( 4\frac{\beta}{|T_i^{\varepsilon}|} + 2\, \tilde{\gamma}_i \, \varepsilon \right) \,  C \cdot 2 \cdot |T_i^{\varepsilon}| \Bigg] \\
&\leq& 2 \, \sum_{i=1}^{N(\varepsilon)} \frac{\eta_i(\varepsilon)}{|T_i^{\varepsilon}|} \, \left\| |U_j|^{-1} \int_{U_j} T_{\lambda^{-1}}F(T_i^{\varepsilon}) \, d\lambda - S(T_i^{\varepsilon})  \right\|_Z  + 48\varepsilon \, C
\end{eqnarray*}
for every $j \geq j_0(\varepsilon)$.
It now follows from the claim (A) that
\begin{eqnarray} \label{eqn:ESTTD3}
\limsup_{j \rightarrow \infty} D_4(\varepsilon, j) \leq 48 \varepsilon \, C.
\end{eqnarray}

Finally, again by using Theorem \ref{thm:UDT}, property (II), with $\hat{S}=\hat{U}_j$, we also get an estimate for $D_5(j, \varepsilon)$. So by the uniform distribution of the $\hat{C}_i^{y}$ and since $|\Lambda| \geq (1-\beta)|\hat{U}_j|$ by the statement (I) of Theorem \ref{thm:UDT}, we obtain
\begin{eqnarray*}
D_5(j, \varepsilon)  &\leq& \sum_{i=1}^{N(\varepsilon)} \left| |\Upsilon|^{-1} \int_{\Upsilon} \frac{\operatorname{card}(\hat{C}_i^{y})}{|\Lambda|} \, dy - \frac{\eta_i(\varepsilon)}{|T_i^{\varepsilon}|}\right| \,  \| S(T_i^{\varepsilon}) \|_Z  \\
&\leq& C\, [(1-\beta)^{-1}-1] \cdot  |\Upsilon|^{-1}\int_{\Upsilon} \, \sum_{i=1}^{N(\varepsilon)}  \frac{\operatorname{card}(\hat{C}_i^{y}) \,|T_i^{\varepsilon}|}{|\hat{U}_j|} \, dy \,+  \\
&\quad& C \, \sum_{i=1}^{N(\varepsilon)} \left( \frac{4\beta}{|T_i^{\varepsilon}|} + 2\tilde{\gamma}_{i} \varepsilon \right) \, |T_i^{\varepsilon}|
\end{eqnarray*}
for $j \geq j_0(\varepsilon)$.
Since the translates $\{T_ic\}$, $c \in \hat{C}_i^{y}$ are $\varepsilon$-disjoint and as $\sum_{i=1}^{N(\varepsilon)} \tilde{\gamma}_i \, |T_i^{\varepsilon}| \leq 2$, we arrive at

\begin{eqnarray*}
D_5(j, \varepsilon) &\leq& C \, \frac{(1-\beta)^{-1} - 1}{1- \varepsilon} + C \, \sum_{i=1}^{N(\varepsilon)} \left( 4 \beta + 2\tilde{\gamma}_i|T_i^{\varepsilon}| \,\varepsilon \right) \nonumber \\
&\stackrel{\beta < 2^{-N(\varepsilon)}\varepsilon}{\leq}& 16C \, \varepsilon
\end{eqnarray*}
for $j \geq j_0(\varepsilon)$ and thus, 
\begin{eqnarray} \label{eqn:ESTTD4}
\limsup_{j \rightarrow \infty} D_5(j, \varepsilon) \leq 16C\,\varepsilon.
\end{eqnarray}

To conclude the statement, we derive from the Inequalities (\ref{eqn:ESTTD1}), (\ref{eqn:ESTTD1.5}) (\ref{eqn:ESTTD2}), (\ref{eqn:ESTTD3}), as well as (\ref{eqn:ESTTD4}) that indeed, 
\begin{eqnarray*}
\lim_{\varepsilon\rightarrow 0} \limsup_{j \rightarrow \infty} \Delta(j,\varepsilon) = 0. 
\end{eqnarray*} 
In particular, this means that
\begin{eqnarray*}
\limsup_{k,l \rightarrow \infty} \left\| \frac{F(U_l)}{|U_l|} - \frac{F(U_k)}{|U_k|} \right\|_Z &\leq& \lim_{\varepsilon \rightarrow 0} \limsup_{k \rightarrow \infty} \Delta(k,\varepsilon)  + \lim_{\varepsilon \rightarrow 0} \limsup_{l \rightarrow \infty}  \Delta(l, \varepsilon) \\
&=& 0 
\end{eqnarray*}
and $F(U_j)/|U_j|$ is a Cauchy sequence in $Z$ and hence converges in $Z$. The representation as the second limit is now an easy consequence of the triangle inequality. \\

For the validity of claim (C), take an arbitrary $g \in G$. Note that for all $0 < \varepsilon < 1/10$ and every $1 \leq i\leq N$, we have $T_gS(T_i^{\varepsilon}) = S(T_i^{\varepsilon})$ by the general mean ergodic Theorem \ref{thm:MET}.  
By the boundedness (continuity) of the operator $T_g$ and the convergence result in claim (B), the following computation finishes our proof:
\begin{eqnarray*}
T_g \overline{F} &=& T_g \left( \lim_{\varepsilon \rightarrow 0} \sum_{i=1}^{N(\varepsilon)} \eta_i(\varepsilon) \, \frac{S(T_i^{\varepsilon})}{|T_i^{\varepsilon}|} \right) = \lim_{\varepsilon \rightarrow 0} \left(\sum_{i=1}^{N(\varepsilon)} \eta_i(\varepsilon) \, \frac{T_gS(T_i^{\varepsilon})}{|T_i^{\varepsilon}|} \right) \\ 
&=& \lim_{\varepsilon \rightarrow 0} \left(\sum_{i=1}^{N(\varepsilon)} \eta_i(\varepsilon) \, \frac{S(T_i^{\varepsilon})}{|T_i^{\varepsilon}|} \right) = \overline{F}. 
\end{eqnarray*}

\end{proof}

There may occur situations where we do not have a weakly measurable operator action satisfying the invariance condition at hand. 
However, assuming the existence of certain abstract limits, we are still able to derive the mean ergodic theorem.  

\begin{Corollary} \label{cor:METSF}
Let $G$ be a unimodular group and assume that $\{U_j\}$ is a strong F{\o}lner sequence in $G$. Let $(Z, \|\cdot\|_Z)$ be a Banach space. Further, denote by $b$ some tiling-admissible boundary term defined on $\mathcal{F}(G)$ and we consider a bounded, $b$-almost additive mapping $\mathcal{F}(G) \rightarrow (Z, \|\cdot\|_Z)$. \\ 
Assume that for each $Q \in \mathcal{F}(G)$ and all elements $h \in Z^{*}$, the mappings 
\[
\psi_{Q,h}: G \rightarrow \CC: g \mapsto \langle F(Qg), h \rangle_{Z,Z^{*}} 
\]
are measurable. \\
Then, if for a positive sequence $\varepsilon_k \rightarrow 0$ and $N(\varepsilon_k):= \lceil \log(\varepsilon_k)/\log(1- \varepsilon_k) \rceil$, we have 
$\varepsilon_k$-quasi tilings $\{T_i^{\varepsilon_k}\}_{i=1}^{N(\varepsilon_k)}$ of the group as in Definition \ref{defi:STP} with $0 < \beta < 2^{-N(\varepsilon_k)}\varepsilon_k$, and if for each $k \in \mathbb{N}$ and every $1 \leq i \leq N(\varepsilon_k)$, the expression
\[
S(T_i^{\varepsilon_k}) := \lim_{j \rightarrow \infty} |U_j|^{-1} \int_{U_j} F(T_i^{\varepsilon_k}g)\, dg
\]
exists in $Z$, then the following limits exist in $Z$ and are equal:
\begin{eqnarray*}
\overline{F} := \lim_{j \rightarrow \infty} \frac{F(U_j)}{|U_j|} = \lim_{k\rightarrow\infty} \sum_{i=1}^{N(\varepsilon_k)} \eta_i(\varepsilon_k) \, \frac{S(T_i^{\varepsilon_k})}{|T_i^{\varepsilon_k}|},
\end{eqnarray*}
where we have $\eta_i(\varepsilon_k):= \varepsilon(1-\varepsilon_k)^{N(\varepsilon_k) - i}$ for $1 \leq i \leq N(\varepsilon_k)$.
\end{Corollary}

\begin{proof}
See the proof of the claim (B) of Theorem \ref{thm:METSF}. 
\end{proof}

\section{Pointwise ergodic theorems} \label{sec:PWET}

For later considerations for the integrated density of states in random models in Section~\ref{sec:IDScont}, we now extend the celebrated {\sc Lindenstrauss} pointwise ergodic theorem (cf.\@ \cite{Lindenstrauss-01}) to abstract ergodic averages given by Definition~\ref{defi:AEA} on Bochner spaces. More precisely, as Banach space under consideration, we choose $Z:=L^p(\Omega, Y)$, where $1 \leq p < \infty$, $\Omega$ is a $\sigma$-finite measure space and $Y$ is an arbitrary reflexive Banach space. As before, an amenable group acts weakly measurably on $Z$ via a family of linear, uniformly bounded operators $\{T_g\}_{g \in G}$. This action is linked with a measure preserving action of the group on the measure space, see Inequality~(\ref{eqn:tmp}). Following the classic proofs of pointwise ergodic theorems (e.g.\@~\cite{Emerson-74,Krengel-85,Lindenstrauss-01}), we use a so-called $L^p$-maximal inequality (cf.\@ Theorem~\ref{thm:LPmax}) to show the almost everywhere convergence in Theorem~\ref{thm:lindenstr}. We will use the latter result to verify the almost sure convergence of the integrated density of states for certain random operators, cf.\@ Theorem~\ref{thm:pointwise}.
     
\begin{Definition}
Let $Y$ be a Banach space and assume that $(\Omega, \mathcal{F}, \mu)$ is a $\sigma$-finite measure space. For $1 \leq p < \infty$, we denote by $L^{p}(\Omega, Y)$ the (Bochner) space of all equivalence classes $f:\Omega \rightarrow Y$ such that each representant $f$ is strongly measurable with respect to $\mathcal{F}$ and
\begin{eqnarray*}
\|f\|_{L^p(\Omega, Y)}:= \left( \int_{\Omega} \|f(\omega)\|^p_Y \, d\mu(\omega) \right)^{1/p} < \infty,
\end{eqnarray*}
i.e.\@ $\|f(\cdot)\|_Y \in L^p(\Omega, \RR) :=  L^p(\Omega, \mathcal{F}, \mu)$.
For $p = \infty$,  we set $L^{\infty}(\Omega, Y)$ as the space of strongly measurable equivalence relations $f$ such that
\begin{eqnarray*}
\|f\|_{L^{\infty}(\Omega, Y)} := \operatorname{ess}\, \sup_{\omega \in \Omega} \|f(\omega)\|_Y < \infty,
\end{eqnarray*}
i.e.\@ $\|f(\cdot)\|_Y \in L^{\infty}(\Omega, \RR):= L^{\infty}(\Omega, \mathcal{F}, \mu)$.
\end{Definition}



\begin{Lemma} \label{lemma:refl}
Let $Y$ be a reflexive Banach space. Then for all $1 < p < \infty$, the space $L^p(\Omega, Y)$ is reflexive. In particular, we have
\begin{eqnarray*}
L^p(\Omega, Y)^{*} \cong L^q(\Omega, Y^{*}),
\end{eqnarray*}
where $1/p + 1/q = 1$ and $Y^{*}$ is the dual space of $Y$.
\end{Lemma}

\begin{proof}
See \cite{GretskyU-72}, Corollary 3.4.
\end{proof}

We are now interested in the special cases where $Z:= L^p(\Omega, Y)$ for some $1 \leq p < \infty$ and $Y$ is a reflexive Banach space. Further, let the group $G$ act weakly measurably on $(\Omega, \mathcal{F}, \mu)$ by measure preserving transformations, where we write $g\omega := g \cdot \omega$ for $g\in G$ and $\omega \in \Omega$. For each $g \in G$, we are given a mapping $T_g: L^p(\Omega, Y) \rightarrow L^p(\Omega, Y)$ such that the collection $\{T_g\}$ acts weakly measurably on $Z$ as a family of uniformly bounded, linear operators (cf.\@ Definition \ref{defi:WMA}). Additionally, we assume that there is some measurable group homomorphism $\varphi: G \rightarrow G$ along with some constant $\kappa > 0$ such that for every $g \in G$ and each $f \in Z$
\[
\|T_g f (\omega) \|_Y \leq \kappa \, \|f(\varphi(g)^{-1}\omega)\|_Y
\]  
for $\mu$-almost every $\omega \in \Omega$. Since the action of $G$ on $\Omega$ preserves the measure $\mu$, it follows that $\sup_{g \in G}\|T_g\|_{L^p(\Omega,Y)} \leq \kappa$ and a short calculation shows 
\begin{eqnarray} \label{eqn:tmp}
\kappa^{-1}\, \| f(\varphi(g)^{-1}\omega) \|_Y \leq \|T_g f (\omega) \|_Y \leq \kappa \, \|f(\varphi(g)^{-1}\omega)\|_Y,
\end{eqnarray}
which implies that $\kappa \geq 1$. Note that this setting includes the 'standard situation', where one has $(T_gf)(\omega) = f(g^{-1}\omega)$ for $f \in L^p(\Omega, Y)$, $g \in G$ and $\omega \in \Omega$. In the following, we denote by $(U_n)$ a tempered, strong F{\o}lner sequence in $G$, cf.\@ Definitions \ref{defi:FS} and \ref{defi:growth}. 

Recall that in Definition \ref{defi:AEA}, we defined the $j$-th $(j \in \NN)$ abstract ergodic average with respect to $(U_j)$ as
\[
A_jf := |U_n|^{-1} \int_{U_n} T_{g^{-1}}f \, dm_L(g)
\]
for $f \in Z$.
Analogously, for $h \in L^p(\Omega, \RR)$, denote by
\[
A_j^{\varphi}h(\omega) := |U_j|^{-1} \int_{U_j} h(\varphi(g)\omega) \, dm_L(g), \quad j \in \NN,\quad \omega \in \Omega 
\]
the $j$-th {\em abstract ergodic average} of $h$ in $L^p(\Omega,\RR)$ with respect to $(U_j)$ and $\varphi$. With Lemma~\ref{lemma:refl}, we can deduce from Theorem \ref{thm:MET} that for $1 < p < \infty$ and for each $f \in L^p(\Omega, Y)$ the $A_jf$ converge to some $f^{*}$ in $L^p(\Omega, Y)$ as $j \rightarrow \infty$, where $T_g f^{*} = f^{*}$ in $L^p(\Omega, Y)$ for all $g \in G$. Equivalently, we have $A_j \rightarrow P$ strongly in $L^p(\Omega, Y)$ as $j \rightarrow \infty$, where $P$ is a bounded projection on $\operatorname{ran}(P) = \operatorname{Fix}(T_G):= \{h \in L^p(\Omega, Y) \,|\, T_gh = h \,\mbox{ all } g \in G\}$. The following lemma shows that although the mean ergodic theorem does in general not hold for the case $p=1$, we can extend the projection $P$ from the case $p=2$ to the space $L^1(\Omega, Y)$.  

\begin{Lemma} \label{lemma:extend}
Assume the situation of Theorem \ref{thm:MET} with $Z=L^2(\Omega, Y)$ and let $P$ be the corresponding mean ergodic projection. Then there is a bounded projection $\tilde{P}$ on $L^1(\Omega, Y)$ which coincides with $P$ on $L^2(\Omega, Y)$ such that
\begin{enumerate}[(i)]
\item $\operatorname{ran}(\tilde{P}) \subseteq \operatorname{Fix}(T_G)$, 
\item the space $L^{*}:= L_0^{*} \cap L^1(\Omega, Y)$ is $L^1(\Omega, Y)$-dense in $\operatorname{ker}(\tilde{P})$, \\
where $L_0^{*} := \overline{\operatorname{lin}}^{\| \cdot \|_{L^2(\Omega, Y)}}\{ h - T_gh \,|\, h \in L^2(\Omega, Y), \, g \in G \}$.
\end{enumerate}
\end{Lemma}  

\begin{proof}
The proof follows easily from standard techniques, cf.\@ \cite{Garsia-70}, proof of Theorem 2.1.1. For a detailed description, see also \cite{Pogorzelski-10}, proof of Theorem 5.3. All the arguments carry over to Bochner spaces. 
\end{proof}

For the proof of pointwise ergodic theorems, it is convenient to work with the so-called {\em maximal operator} with respect to the $A_j^{\phi}$.

\begin{Definition} \label{defi:maxOP}
Let $1 \leq p < \infty$ and assume that $\phi:G \rightarrow G$ be a measurable group homomorphism. The maximal operator $M^{\phi}$ on $L^p(\Omega, \RR)$ is then defined as
\begin{eqnarray*}
M^{\phi}: L^p(\Omega, \RR) \rightarrow L^0(\Omega, \RR): M^{\phi}h(\omega) := \sup_{j \in \NN} |{A}_j^{\phi}h|(\omega),
\end{eqnarray*}
where $L^{0}(\Omega, \RR)$ denotes the space of all measurable functions from $\Omega$ to $\overline{\RR}:= \RR \cup \{-\infty; +\infty\}$.
\end{Definition}

We will prove the Lindenstrauss ergodic theorem by using a so called {\em maximal ergodic theorem}, i.e.\@ we show that the abstract ergodic averages $A_j f$ on $L^p(\Omega, Y)$ satisfy an $L^p$-{\em maximal inequality}.

\begin{Definition}
We say that the abstract ergodic averages $\{A_n\}$ on $L^p(\Omega, Y)$, $1 \leq p < \infty$ satisfy an $L^p$-maximal inequality if there is a constant $K>0$ such that for all $f \in L^p(\Omega, Y)$ and every $\lambda > 0$, the following holds:
\begin{eqnarray*}
\mu\left(\omega \,\Big|\, \sup_{n \in \NN} \|A_n f\|_Y(\omega) > \lambda\right) \leq \frac{K}{\lambda^p} \, \|f\|^p_{L^p(\Omega, Y)}.
\end{eqnarray*}
\end{Definition}

Now using the celebrated result from {\sc Lindenstrauss}, cf.\@ \cite{Lindenstrauss-01}, we can show the following dominated ergodic theorem. As our proof requires only slight modifications, we just give a sketch in the appendix of this paper, cf.\@ Section \ref{sec:appendix}.

\begin{Theorem} \label{thm:LPmax}
Let $G$ be an amenable (second countable Hausdorff) group acting on a $\sigma$-finite measure space by measure preserving transformations. Further, let $\{T_g\}_{g \in G}$ be a family of uniformly bounded operators acting weakly measurably on $L^p(\Omega, Y)$, $(1 \leq p <  \infty)$, where $Y$ is some reflexive Banach space and the Inequality (\ref{eqn:tmp}) holds for some constant $\kappa \geq 1$ and some measurable group homomorphism $\varphi$. Then the abstract ergodic averages $\{A_j\}_{j \in \NN}$ associated with a weak tempered F{\o}lner sequence $(U_j)$ in $G$ satisfy an $L^p$-maximal inequality on $L^p(\Omega,Y)$.  
\end{Theorem}

\begin{proof}
The proof is a modification of the proof in \cite{Lindenstrauss-01}. For a sketch of the changes, see Section \ref{sec:appendix}. 
\end{proof}

\begin{Remark}
It is also shown in the proof that the abstract ergodic averages $\{A_j^{\phi}\}$ satisfy an $L^p$-maximal inequality on $L^p(\Omega,\RR)$ for all $1 \leq p < \infty$, cf.\@ Inequality (\ref{eqn:lin}). 
\end{Remark}

With this $L^p$-maximal inequality at hand, it is now straight forward to derive the pointwise ergodic theorem for the abstract ergodic averages $A_j$, given by Definition \ref{defi:AEA}. 

\begin{Theorem} \label{thm:lindenstr}
Let $G$ be an amenable group which acts on a $\sigma$-finite measure space $(\Omega, \mu)$ by measure preserving transformations. Assume further that for $1 \leq p < \infty$ and some reflexive Banach space $Y$, the group acts weakly measurably on $L^p(\Omega, Y)$ via a family $\{T_g\}_{g \in G}$ of uniformly bounded operators such that the Inequality (\ref{eqn:tmp}) is satisfied. Then, if $(U_j)$ is a strong and tempered F{\o}lner sequence in $G$, for each $f \in L^p(\Omega, Y)$, we find an element $\overline{f} \in L^p(\Omega, Y)$ such that
\[
\lim_{j \rightarrow \infty} \left\| |U_j|^{-1} \int_{U_j} (T_{g^{-1}}f)(\omega) \, dg - \overline{f}(\omega) \right\|_Y = 0
\]
for $\mu$-almost every $\omega \in \Omega$. Moreover, we have $T_g\overline{f} = \overline{f}$ in $L^p(\Omega, Y)$ for all $g \in G$.  
\end{Theorem} 

\begin{proof}
We consider first the case $p > 1$. 
Following the abstract mean ergodic Theorem \ref{thm:MET}, and using the fact that the space $L^p(\cdot) \cap L^{\infty}(\cdot)$ is dense in $L^p(\cdot)$, we obtain the decomposition
\[
L^p(\Omega, Y) = \operatorname{Fix}(T_G) \oplus \overline{L_0}^{L^p},
\]
where $L_0 := {\operatorname{lin}}\{ T_gf - f \,|\, f \in L^p(\cdot) \cap L^{\infty}(\cdot), \, g \in G \}$.
Using the $L^p$-maximal inequality, Theorem \ref{thm:LPmax}, if follows from the Banach principle (see e.g.\@ \cite{EisnerFHN-12}, Chapter 10) that it is enough to verify the pointwise almost everywhere convergence on a dense subspace $D$ of $L^p(\Omega, Y)$. Here, it is convenient to consider $D=\operatorname{Fix}(T_G) + L_0$. By linearity, we can look at those spaces separately. So for $f \in \operatorname{Fix}(T_G)$, the convergence result is trivial. For an element $h:= T_gf -f$ $(f \in L^p(\cdot)\cap L^{\infty}(\cdot),\, g \in G)$, the convergence follows from a simple change of coordinates in the resulting integrals and from the fact that $|U_j \triangle g^{-1}U_j|/|U_j| \stackrel{j\rightarrow \infty}{\rightarrow} 0$, see e.g.\@ \cite{Emerson-74}, proof of Theorem 3. \\
For $p=1$, the abstract mean ergodic Theorem \ref{thm:MET} does not hold in general. However, we can use the Lemma \ref{lemma:extend} to define the $L^2$-mean ergodic projection on the whole space $L^1(\Omega, Y)$. Exploiting the fact that the abstract ergodic averages satisfy an $L^1$- and an $L^2$-maximal inequality (Theorem \ref{thm:LPmax}), similar approximation techniques as in the case $p>1$ lead to the desired result. For a more detailed demonstration, the interested reader may refer e.g.\@ to \cite{Pogorzelski-10}, proof of Theorem 5.3.     
\end{proof}

\section{Bounded, additive processes} \label{sec:BAP}

In the following section, we draw our attention to generalized abstract ergodic averages. More precisely, we deal with so-called bounded, additive processes on $Z$ comprising the objects of Definition \ref{defi:AEA} as a special class. We start by giving some examples. Inspired by differentiation theorems for $\RR^d$ in \cite{AkcogluJ-81, Emilion-85, Sato-98,Sato-99}, we further introduce the concept of the {\em associated dominating process} in Definition \ref{defi:APF}. This enables us to show that all bounded, additive processes satisfy an $L^1$-maximal inequality along Tempelman F{\o}lner sequences, cf.\@ Theorem \ref{thm:DET}. \\
As an application of this $L^1$-maximal inequality, as well as of the abstract mean ergodic Theorem \ref{thm:METSF}, we show in Theorem \ref{thm:PWET} for the class of {\em approximable} bounded, additive processes the pointwise almost everywhere convergence in the case of a finite measure space. In an earlier version of the present paper, it was
stated that the convergence holds true for all bounded, additive processes. However, the author of this work observed a flaw in the previous proof which required us to add the assumption of
approximability, cf.\@ the Corrigendum and Addendum~\cite{PogERR}.  While many bounded, additive processes satisfy this condition, to the
knowledge of the author it is not known whether Poisson point processes over locally compact, unimodular, amenable groups are approximable. Hence, we take back the claim 
to have proven the pointwise convergence for Poisson point processes, cf.\@ the assertion of Corollary~7.12 in the published version \cite{PogJFA}. \\
We draw some links to the literature. 
In \cite{Sato-99, Sato-03}, the author deals with a $\sigma$-finite measure space, as well as with $\mathbb{R}^{d\,+}$ semi-group actions with contraction majorants, where the convergence is shown along $d$-dimensional cubes exhausting the space. In this context, our theorem complements the corresponding, previous result from the literature.
On the one hand, we have to restrict ourselves to group actions which are dominated by invertible ground transformations on a probability space, cf.\@ Inequality \ref{eqn:tmp}. On the other hand, we can show the almost everywhere convergence for dynamics coming from general unimodular groups and along all Tempelman F{\o}lner sequences.

\begin{Definition} \label{defi:AP}
Let $G$ be a unimodular amenable group and denote by $Y$ some reflexive Banach space. For a $\sigma$-finite measure space $(\Omega, \mathcal{F}, \mu)$ with $\mu$-invariant $G$-action, as well as for $1 \leq p < \infty$, we assume that there is a family $\{T_g\}_{g \in G}$ of uniformly bounded, linear operators acting weakly measurably on $L^p(\Omega,Y)$ such that Inequality (\ref{eqn:tmp}) is satisfied for $\kappa$ and for $\varphi$. 
In this situation, we call the map
\[
F: \mathcal{F}(G) \rightarrow L^p(\Omega, Y)
\]  
a {\em bounded additive process} on $G$, if the following statements hold. 
\begin{enumerate}[(i)]
\item $K:= \sup\{ \|F(Q)\|_{L^p(\Omega, Y)}/|Q| \,|\, Q \in \mathcal{F}(G) \} < \infty$, 
\item $F(Q) = \sum_{k=1}^m F(Q_k)$ if $Q \in \mathcal{F}(G)$ is a disjoint union of the $Q_k \in \mathcal{F}(G)$ for $1 \leq k \leq m$,
\item $T_gF(Q)=F(Qg^{-1})$ for all $Q \in \mathcal{F}(G)$ and every $g \in G$. 
\end{enumerate}
\end{Definition}

Let us give some examples first. 

\begin{Example} \label{exa:examples}
\begin{itemize}
\item Assume that $G$ is unimodular and amenable and that it acts on a $\sigma$-finite measure space $(\Omega, \mathcal{F}, \mu)$ by measure preserving transformations. Then for every $f \in L^p(\Omega, \RR)$ $(1 \leq p < \infty)$, the map
\[
\mathcal{F}(G) \rightarrow L^p(\Omega, \RR): F(Q)(\cdot):= \int_Q T_{g^{-1}} f \, dg\, (\cdot)
\]
defines a bounded, additive process for the canonical action $T_gh(\omega) = h(g^{-1}\omega)$ on $L^p(\Omega,\RR)$. 
In this case, we say that the process $F$ is {\em absolutely continuous with density $f$} with respect to $G$.
\item Let $G=\RR^d$ $(d \geq 1)$ and assume that $F: \mathcal{F}(G) \rightarrow L_+^1(\Omega, \RR )$ is a bounded, additive process for a measure preserving action $T_gh(\omega) = h(g^{-1}\omega)$ on the canonical non-negative cone in $L^1(\Omega, \RR)$. It is shown in \cite{AkcogluJ-81} that in this situation, we can write $F=F_1 + F_2$, where $F_1$ is some absolutely continuous process with a non-negative density and where $F_2$ is a {\em singular process}, i.e.\@ a bounded, additive process which does not dominate any absolutely contiuous, non-zero, non-negative process.
\item In the previous example, set $(\Omega, \mathcal{F}, \mu) = (\RR^d, \mathcal{B}(\RR^d), \mathcal{L}_d)$, where $\mathcal{L}_d$ is the usual $d$-dimensional Lebesque measure in the euclidean space. Assume further that the action of $G$ on $\RR^d$ is given by translation, i.e.\@ $T_gf(x) = f(x-g)$ for all $g,x \in \RR^d$ and where $f \in L_+^1(\RR^d)$. Then every singular bounded, additive process with respect to $\{T_g\}$ is of the form $F(Q)(x) = \nu(x + Q)$ $(x \in \RR^d)$, where $\nu$ is a Borel measure which is singular with respect to $\mathcal{L}_d$ (cf.\@ \cite{AkcogluJ-81}, (4.12)). 
\item The same result as in the previous example holds true if we consider $(\Omega,\mathcal{F},\mu)=(\TT^d,\mathcal{B}(\TT^d), \mathcal{L}_d)$, where $\mathcal{L}_d$ is the $d$-dimensional Lebesgue measure on the $d$-dimensional torus $\TT^d$ and $G=\RR^d$ acts by rotations. 
\item Let $G$ be a unimodular, amenable group with Haar measure $m_L$. Moreover, let $(\Omega, \mathcal{F},\mu)$ be a probability space such that $\alpha\,m_L$ $(\alpha > 0)$ is the intensity measure of some $G$-set valued, homogeneous Poisson point process $X$ defined on $(\Omega, \mu)$. Concerning their existence and further information on point processes on locally compact groups, the reader may refer to \cite{Kingman-93, Pogorzelski-10}. Note that there is a $\mu$-preserving action of $G$ from the right on $\Omega$ and we have $X(g\omega):=  X(\omega(g))= X(\omega)g^{-1}$. We assume that the latter action is ergodic, hence we consider a so-called {\em ergodic point process}. 
In this case, the mapping $F(Q)(\omega):= \operatorname{card}(X(\omega) \cap Q)$ $(\omega \in \Omega, Q \in \mathcal{F}(G))$ defines a bounded, additive process for the canonical translations (see above) $\{T_g\}$ on $L^1_+(\Omega, \RR)$. This can readily be checked. 
\begin{itemize}
\item for the boundedness, note that by the distributional properties of the poisson point process, we have
\begin{eqnarray*}
\|F(Q)\|_{L^1(\Omega)} = \EE_{\mu}\left( \operatorname{card}(X(\cdot) \cap Q) \right) = \alpha\,|Q|, \quad (Q \in \mathcal{F}(G)).
\end{eqnarray*}
We conclude that $F$ is bounded with constant $\alpha > 0$.
\item for the additivity, assume that $Q = \bigsqcup_{k=1}^m Q_k$ is a disjoint union of elements in $\mathcal{F}(G)$. Then indeed, 
\begin{eqnarray*}
F(Q) = \operatorname{card}(X(\omega) \cap Q) = \sum_{k=1}^m \operatorname{card}(X(\omega) \cap Q_k) = \sum_{k=1}^m F(Q_k)
\end{eqnarray*}
for all $\omega \in \Omega$.
\item for the invariance, we compute 
\begin{eqnarray*}
F(Qg)(\omega) &=& \operatorname{card}(X(\omega) \cap Qg) = \operatorname{card}(X(\omega)g^{-1} \cap Q)  \\
&=& \operatorname{card}(X(\omega(g)) \cap Q) = F(Q)(g\omega)
\end{eqnarray*}
for $Q \in \mathcal{F}(G)$, $\omega \in \Omega$, $g \in G$. 
\end{itemize}

Note that indeed, if $G$ is not discrete, then the process is not absolutely continuous. 
We give a brief justification which arose in private communication of the author of this work with Xueping Huang. 
So assume that there is some measurable function $f$ on $\Omega$ such that
\begin{eqnarray*}
\operatorname{card}(X(\omega) \cap Q) = \int_Q f(g\omega)\, dg
\end{eqnarray*}
for every $Q \in \mathcal{F}(G)$. Consider an arbitrary open set $Q \in \mathcal{F}(G)$. Integrating the above relation shows that $f \in L^1(\Omega, \mu)$ and that there is a set $\Omega(Q) \subseteq \Omega$ of full measure such that the integral $\int_Q f(g\omega)\,dg$ is finite. Taking a countable cover ${Q_j}$ of the (second countable) group consisting of open sets of finite measure, we find a subset $\overline{\Omega} \subseteq \Omega$ of full measure such that for all $\omega \in \overline{\Omega}$ and every $j \in \NN$, the expression $\int_{Q_j} f(g\omega)\,dg$ is finite. This also demonstrates that for almost all $\omega \in \Omega$ and every compact set $A \subseteq G$,
\begin{eqnarray} \label{eqn:fnt}
\operatorname{card}(X(\omega) \cap A) = \int_A f(g\omega)\, dg <  \infty. 
\end{eqnarray}
Now fix $\omega_0 \in \overline{\Omega}$, along with an arbitrary $x \in X(\omega_0)$. Further, take a decreasing sequence $A_n$ of compact sets such that $\cap_n A_n = \{x\}$. Since $G$ does not possess atoms, $\lim_{n\rightarrow\infty} m_L(A_n) = 0$. Also, $\liminf_{n\to\infty} \operatorname{card}(X(\omega_0) \cap A_n) \geq 1$. On the other hand, it follows from Vitali's theorem that $\lim_{n\to\infty} \int_{A_n} f(g\omega_0) \, dg = 0$. This clearly is a contradiction to the Equality~\ref{eqn:fnt}. Hence, there is no such $f \in L^1(\Omega,\mu)$. Hence, those poisson point processes do not belong to the class of absolutely continuous processes.    
\end{itemize}
\end{Example}

We will see below that with the notion of bounded, additive processes at hand, it is worth dealing with the following $\RR$-valued (in fact non-negative) expressions. 

\begin{Definition} \label{defi:APF}
For a bounded additive process $F$, we define the {\em associated dominating process} $F^{0}$ as
\begin{eqnarray*}
&& F^{0}: \mathcal{F}(G) \rightarrow L^{0}(\Omega, \RR): \\
&& F^{0}(Q)(\omega) := \operatorname{ess}\,\sup \left\{ \sum_{k=1}^m \|F(Q_k)\|_Y(\omega) \, \Big| \, Q = \cup_{k=1}^m Q_k \mbox{ disj.}, \, Q_k \in \mathcal{F}(G), \, 1 \leq k \leq m,\, m\in \NN \right\}.
\end{eqnarray*}
for $Q \in \mathcal{F}(G)$ and for $\omega \in \Omega$.
\end{Definition}

Let us justify the measurability of the $F^{0}(Q)$ first. Note that in the case of discrete groups, we simply have $F^{0}(Q)(\omega):= \sum_{g \in Q} \|F(\{g\})\|_Y(\omega)$ for $\omega \in \Omega$ and for a finite set $Q \subseteq G$. \\ 
If $G$ is non-discrete, the measurability is guaranteed by the following lemma. Its validity has already been stated without proof in \cite{Emilion-85}. For the sake of completeness, we attach a proof in the appendix (Section \ref{sec:appendix}) of this paper.   

\begin{Lemma} \label{lemma:partition}
Let $G$ be a non-discrete, unimodular, amenable group, endowed with some bounded, additive process $F$ as indicated above. Then we can find a sequence of partitions $\{P_m\}_{m \in \NN}$ of $G$, consisting of countably many sets from $\mathcal{F}(G)$ such that the following holds.
\begin{itemize}
\item For each $A \in P_m$, we have $|A| < 2^{-m}$ $(m \in \NN)$. 
\item $P_{m+1}$ is a refinement of $P_m$ for each $m \in \NN$. 
\item For each $Q \in \mathcal{F}(G)$, we have the representation $F^{0}(Q)(\omega) = \lim_{m \rightarrow \infty} F^{0}_{P_m}(Q)(\omega)$ for $\mu$-almost every $\omega \in \Omega$, where $F^{0}_{P_m}(Q)(\cdot):= \sum_{A \in P_m} \|F(A \cap Q)(\cdot)\|_Y$ for $m \in \NN$.  
\end{itemize}
\end{Lemma}

\begin{proof}
See the appendix of this paper, Section \ref{sec:appendix}.  
\end{proof}

\begin{Proposition} \label{prop:FNull}
Let $F$ be some bounded, additive process on a unimodular, amenable group $G$ with values in $L^1(\Omega, Y)$. 
Then the associated dominating process $F^{0}(Q)$ maps as $F^0: \mathcal{F}(G) \rightarrow L^1(\Omega, \RR)$.
Moreover, the following assertions hold true. 
\begin{enumerate}[(i)]
\item $\sup_{Q \in \mathcal{F}(G)} \frac{ \|F(Q)\|_{L^1(Y)}}{|Q|} = \sup_{Q \in \mathcal{F}(G)} \frac{\|F^0(Q)\|_{L^1(\RR)}}{|Q|}$.
\item $F^{0}(Q) = \sum_{l=1}^L F^{0}(Q_l)$ for every disjoint union $Q = \bigsqcup_{l=1}^L Q_l$ in $\mathcal{F}(G)$. 
\end{enumerate} 
\end{Proposition}

\begin{proof}
Set $\tilde{\gamma}:= \sup_{Q} \|F^{0}(Q)\|_{L^p(\Omega, \RR)}/|Q|$ and ${\gamma}:= \sup_{Q} \|F(Q)\|_{L^p(\Omega, Y)}/|Q|$. It follows from Proposition \ref{lemma:partition} and the monotone convergence theorem that $\|F^{0}(Q)\| \leq \gamma \,|Q|$ for every $Q \in \mathcal{F}(G)$. This shows $\tilde{\gamma} \leq \gamma$. For the converse inequality, note that it follows from the definition of $F^{0}$ that $\|F(Q)\| \leq \tilde{\gamma}\,|Q|$ for all $Q \in \mathcal{F}(G)$. Hence $\gamma \leq \tilde{\gamma}$ and the claim is proven. \\
The proof of the fact that $F^{0}(Q) = \sum_{k=1}^m F^{0}(Q_k)$ for disjoint unions $Q=\bigsqcup_{k=1}^m Q_k$ in $\mathcal{F}(G)$ follows immediately from the definition. The details are left to the reader. 
\end{proof}





The following lemma describes how the dominating process $F^{0}$ for some $F$ is related to the $T_g$ action on $L^p(\Omega, Y)$. 

\begin{Lemma} \label{lemma:FNullcomp}
Let $F$ be some bounded, additive process according to Definition \ref{defi:AP} with its associated dominating process $F^{0}$. 
Then, if $\kappa \geq 1$ is the constant and if $\varphi:G \rightarrow G$ is the measurable group homomorphism in Inequality (\ref{eqn:tmp}), we obtain for all $g \in G$ and for every $Q \in \mathcal{F}(G)$,
\[
\kappa^{-1} F^{0}(Q)(\varphi(g)\omega) \leq  F^{0}(Qg)(\omega) \leq \kappa F^{0}(Q)(\varphi(g)\omega)
\] 
for $\mu$-almost every $\omega \in \Omega$.
\end{Lemma}

\begin{proof}
The claim follows easily from the above representation of $F^{0}$ and from the Inequality (\ref{eqn:tmp}). 
\end{proof}

In the following, we prove a dominated ergodic theorem for bounded, additive processes. Using the concept of the associated dominating process, we apply the techniques of \cite{Krengel-85} to derive an $L^1$-maximal inequality.

\begin{Definition}
Let $G$ be some amenable group, along with a weak F{o}lner sequence $(U_j)$ in $G$. Let $F:\mathcal{F}(G) \rightarrow L^1(\Omega, Y)$ be some bounded, additive process as introduced in Definition \ref{defi:AP}. Then we say that $F$ satisfies an $L^1$-maximality condition for the sequence $(U_j)$ if there is a constant $\gamma > 0$ such that for all $\lambda > 0$ and for every $M \in \NN$, we obtain
\[
\mu\left(\left\{ \omega \in \Omega \,\Big|\, \sup_{j \geq M} \frac{\|F(U_j)(\omega)\|_Y}{|U_j|} > \lambda \right\}\right) \leq \frac{\gamma}{\lambda} \, \sup_{j \geq M} \frac{\|F^0(U_j)\|_{L^1(\Omega, \mathbb{R})}}{|U_j|}
\]
\end{Definition}  

For the proof of such maximal inequalities for bounded, additive processes, we need the following combinatorial covering lemma.

\begin{Lemma} \label{lemma:covXX}
Let $G$ be an amenable group, along with a weak F{\o}lner sequence $(U_n)$ with $U_n \subseteq U_{n+1}$ for all $n \in \NN$. Further, let $M \leq N \in \NN$ and assume that we are given $B,F \in \mathcal{F}(G)$ such that $U_NB \subseteq F$. Then, for every map $\theta:B \rightarrow \{M,\dots, N\}$ there exists a finite subset $\tilde{B} \subseteq B$ for which the sets $U_{\theta(b)}b$ $(b \in \tilde{B})$ are disjoint and such that $B \subseteq \bigcup_{b \in \tilde{B}} U_{\theta(b)}^{-1} U_{\theta(b)} b$.   
\end{Lemma}

\begin{proof}
See \cite{Krengel-85}, Lemma 6.4.3.
\end{proof}

Finally, we are in position to show that bounded, additive processes satisfy the $L^1$-maximality condition for every increasing, weak F{\o}lner sequence of some unimodular group $G$ satisfying the Tempelman condition, cf.\@ Definition \ref{defi:growth}.

\begin{Theorem}[Dominated ergodic theorem] \label{thm:DET}
Let $G$ be a unimodular amenable group, along with a weak F{\o}lner sequence $(U_j)$ such that $U_j \subseteq U_{j+1}$ for all $j \in \NN$ and which satisfies the Tempelman condition. Then every bounded, additive process $F: \mathcal{F}(G) \rightarrow L^1(\Omega, Y)$ as of Definition \ref{defi:AP} satisfies an $L^1$-maximality condition. 
\end{Theorem}

\begin{proof}
The proof is a modification of the proof of Theorem 6.4.4 in \cite{Krengel-85}. 
At first, we fix an integer $M \in \NN$ and $\lambda > 0$. Further, let $N \geq M$ be an integer and denote by $\delta$ an arbitrary positive number. Define the compact set $K:= \cup_{l=M}^N U_l$ and since $(U_j)$ is a weak F{\o}lner sequence, we find a compact set $U_{k_N}$ $(k_N \geq N)$  such that $|U_{k_N} \triangle \overline{U}_N| < \delta\,|U_{k_N}|$, where we have set $\overline{U}_N:=K U_{k_N}$. Moreover, define
the sets 
\[
D_{\lambda, M, N} := \{\omega \in \Omega \,|\, \sup_{M \leq l \leq N} \|F(U_l)(\omega)\|_Y / |U_l| > \lambda\}
\]
and for $\omega \in \Omega$, 
\[
B:= B(\omega, \lambda, M, N) := \{g \in U_{k_N} \,|\, \varphi(g)\omega \in D_{\lambda, M, N}\},
\]
where $\varphi: G  \rightarrow G$ is the homomorphism taken from Inequality (\ref{eqn:tmp}).
Then for any element $b \in B$ there must be some number $\theta(b) \in \{M, \dots, N\}$ such that $\|F(U_{\theta(b)})(\varphi(b)\omega)\|_Y > \lambda\,|U_{\theta(b)}|$. This defines a map $\theta: B \rightarrow \{M,\dots, N\}$. By Lemma \ref{lemma:covXX}, we find some finite subset $\tilde{B} \subseteq B$ such that the sets $U_{\theta(b)}b$ are disjoint for $b \in \tilde{B}$ and 
\[
B(\omega, \lambda, M, N) \subseteq \bigcup_{b \in \tilde{B}} U_{\theta(b)}^{-1} U_{\theta(b)} b.
\]
Since the sequence $(U_n)$ satisfies the Tempelman condition for some constant $\tilde{\kappa} > 0$, it follows that
\begin{eqnarray} \label{eqn:setB1}
|B(\omega, \lambda, M, N)| \leq \tilde{\kappa} \,\left| \bigsqcup_{b \in \tilde{B}} U_{\theta(b)}b \right|.
\end{eqnarray}
By construction, $U_{\theta(b)} b \subseteq \overline{U}_N$ for $b \in \tilde{B}$. So denoting by $F^{0}$ the dominating process associated with $F$, we compute with Proposition \ref{prop:FNull} and with Lemma \ref{lemma:FNullcomp} that
\begin{eqnarray*}
F^{0}(\overline{U}_N)(\omega) &\geq& \sum_{b \in \tilde{B}} F^{0}(U_{\theta(b)} b)(\omega) \geq \kappa^{-1} \, \sum_{b \in \tilde{B}} F^{0}(U_{\theta(b)})(\varphi(b)\omega) \\
&\geq& \kappa^{-1}\, \sum_{b \in \tilde{B}} \| F(U_{\theta(b)})(\varphi(b)\omega) \|_Y  \stackrel{b \in \tilde{B}}{>} \frac{\lambda}{\kappa} \, \left| \bigsqcup_{b \in \tilde{B}} U_{\theta(b)} b \right|
\end{eqnarray*}
for $\mu$-almost every $\omega \in \Omega$ and where $\kappa \geq 1$ is the constant and $\varphi: G  \rightarrow G$ is the homomorphism taken from Inequality (\ref{eqn:tmp}). It follows from Inequality (\ref{eqn:setB1}) that 
\begin{eqnarray} \label{eqn:setB2}
|B(\omega, \lambda, M, N)| &\leq& \frac{\kappa \tilde{\kappa}}{\lambda} \, F^{0}(\overline{U}_N)(\omega) \nonumber \\
&\leq& \frac{\kappa \tilde{\kappa}}{\lambda} \, \left( F^{0}(\overline{U}_N \setminus U_{k_N})(\omega) + F^{0}(U_{k_N})(\omega)\right)
\end{eqnarray}
$\mu$-almost everywhere. Integrating the left hand side of the latter inequality yields
\begin{eqnarray} \label{eqn:setB3}
\int_{\Omega} |B(\omega, \lambda, M, N)| \, d\mu(\omega) = |U_{k_N}| \cdot \mu(D_{\lambda, M, N}),
\end{eqnarray} 
since the action of $G$ on $\Omega$ is $\mu$-preserving. Note further that by the choice of the set $\overline{U}_N$, it is true that $|\overline{U}_N \setminus U_{k_N}| < \delta \,|U_{k_N}|$ and therefore, integrating the right hand side of Inequality~\ref{eqn:setB2}, we obtain with $\gamma:= \sup_{Q \in \mathcal{F}(G)} F^{0}(Q)/|Q| < \infty$ (cf. Proposition~\ref{prop:FNull}) that 
\begin{eqnarray*}
\frac{\kappa \tilde{\kappa}}{\lambda} \, \int_{\Omega} \left( F^{0}(\overline{U}_N \setminus U_{k_N})(\omega) + F^{0}(U_{k_N})(\omega)\right)\,d\mu(\omega) \leq |U_{k_N}| \cdot \frac{\kappa \tilde{\kappa}}{\lambda} \, \left( \gamma\cdot \delta + \frac{\|F^{0}(U_{k_N})\|_{L^1(\RR)}}{|U_{k_N}|} \right).
\end{eqnarray*} 
Combining this fact with the Inequality (\ref{eqn:setB3}), we get with $k_N \geq M$
\[
\mu(D_{\lambda, M, N}) \leq \frac{\kappa \tilde{\kappa}}{\lambda} \, \left( \gamma\cdot\delta + \sup_{l \geq M} \frac{\| F^{0}(U_l) \|_{L^1(\RR)}}{|U_l|} \right)
\]
and with $\delta \rightarrow 0$, we arrive at
\[
\mu(D_{\lambda, M, N}) \leq \frac{\kappa \tilde{\kappa}}{\lambda} \,  \sup_{l \geq M} \frac{\| F^0(U_l) \|_{L^1(\Omega, \mathbb{R})}}{|U_l|}.
\]
Since the left hand side of the latter inequality does not depend on $N$, we can exploit the continuity of the measure $\mu$ as $N \rightarrow \infty$ to finish the proof. 
\end{proof}

In order to prove the pointwise convergence, we have to introduce the notion of 
approximability of bounded, additive processes. In the sequel, we stick to 
finite measure spaces, i.e.\@ $\mu(\Omega) < \infty$. 

\begin{Definition}[Approximable processes] \label{defi:appr}
Let $F:\mathcal{F}(G) \to L^1(\Omega, Y)$ be a bounded, additive process which satisfies
the regularity condition~\eqref{eqn:tmp}. We call $F$ {\em approximable} if there is a
sequence $(F_n)$ of bounded, additive processes on $\mathcal{F}(G)$ with the following
properties.
\begin{itemize}
\item For every $n \in \NN$, $F_n$ takes values in $L^{\infty}(\Omega, Y)$. 
\item For all $n \in \NN$, the process $F-F_n$ satisfies the regularity condition given in~\eqref{eqn:tmp}.
\item For every weak Tempelman F{\o}lner sequence $(U_j)$ with $U_j \subseteq U_{j+1}$, the 
following boundedness condition holds.
For each $n \in \NN$, there is a $j_0 \in \NN$ such that for each $j \geq j_0$, one obtains
\[
F^0_n(U_j)(\omega) \leq n\,|U_j|
\]
for almost every $\omega \in \Omega$, where $F_n^0$ is the associated dominating process for
$F_n$.
\item For every weak Tempelman F{\o}lner sequence $(U_j)$ with $U_j \subseteq U_{j+1}$, the 
sequence $(F_n)$ approximates $F$ in the sense that 
\[
\lim_{n\to\infty} \limsup_{j\to\infty} \frac{\|H_n^0(U_j)\|_{L^1(\Omega,\mathbb{R})}}{|U_j|}
= 0,
\]
where $H_n^0$ is the associated dominating process for $F-F_n$. 
\end{itemize}
\end{Definition}

For some examples for approximable bounded, additive processes, see e.g.\@
\cite{Pogorzelski-PhD}. In particular, the classical integral averages given in the above
list, cf.\@ Examples~\ref{exa:examples}, are approxmiable. 

The next proposition shows that for the approximants $F_n$, the pointwise convergence result holds. 
The proof is based on the Mean Ergodic Theorem~\ref{thm:METSF}.
It is taken from the PhD thesis of the author
of the present paper, cf.\@ \cite{Pogorzelski-PhD}, Proposition~5.16.  

\begin{Proposition} \label{prop:AUX}
Let $F:\mathcal{F}(G) \to L^1(\Omega, Y)$ be an approximable bounded, additive process which satisfies
the regularity condition~\eqref{eqn:tmp} and let $(F_n)$ be a sequence of approximants.
Suppose further that $(U_j)$ is a strong Tempelman F{\o}lner sequence with $U_j \subseteq U_{j+1}$. \\
Then, for every $n \in \NN$, there is some set $\tilde{\Omega} \subseteq \Omega$ of full measure such that
for all $\omega \in \tilde{\Omega}$, the sequence $F_n(U_j)(\omega)/|U_j|$
converges in the topology of $Y$ as $j \to \infty$. 
\end{Proposition}

\begin{proof}
Let $(U_j)$ be an increasing Tempelman F{\o}lner sequence. Fix $n \in \NN$. 
By the assumption of approximability, there is a $j_0(n) \in \NN$ such that
$F^0_n(U_j)(\cdot) \leq n\,|U_j|$ almost-surely 
for all $j \geq j_0(n)$. For the sake of simplicity, we discard the first $j_0(n)$
elements of $U_j$ such that $F_n$ is bounded almost-surely by the constant $n$
along the whole sequence $(U_j)$.

We define the sequence $(S_o)$ of sets in $G$ via $S_o:= U_o u^{-1}$, $o \geq 1$,
where $u \in U_1$ is an arbitrary element. Since  $(U_j)$ is increasing, the 
sequence $(S_o)$ is a nested F{\o}lner sequence. \\
Now, take a sequence $(\varepsilon_k)$ of positive numbers converging to zero. 
For every $k \in \NN$, we set $N(\varepsilon_k):= \lceil \log(\varepsilon_k) /
\log(1- \varepsilon_k) \rceil$ and we choose $\varepsilon_k$-prototiles
$\{ T_i^{\varepsilon_k}\}_{i=1}^{N(\varepsilon_k)}$ taken from the sequence
$S_o$ with $0 < \beta_k < 2^{-N(\varepsilon_k)}
\varepsilon_k$ according to Definition~\ref{defi:STP}. By Theorem~\ref{thm:lindenstr}, we can find a set $\hat{\Omega}
\subseteq \Omega$ with $\mu(\hat{\Omega}) = 1$ 
such that for each $k \in \NN$, for every $1 \leq i \leq N(\varepsilon_k)$,
and for all $\omega \in \hat{\Omega}$, the limits
\begin{eqnarray} \label{eqn:refat2}
S(T_i^{\varepsilon_k})(\omega) := \lim_{j\to \infty} 
|U_j|^{-1} \int_{U_j} F(T_i^{\varepsilon_k}g)(\omega)\,dg
\end{eqnarray}
exist in the topology of the Banach space $Y$. Now, fix
$k \in \NN$. 
By Theorem~\ref{thm:UDT}, we find 
$K = K(\varepsilon_k, \beta_k, T_i^{\varepsilon_k}) \in \NN$ such that 
for every $j \geq K$, we find a decomposition tower
emanating from the set $U_j$.
 Define
\begin{eqnarray*}
\Delta(j,\varepsilon_k, \omega) 
:= \left\| \frac{F(U_j)(\omega)}{|U_j|} - \sum_{i=1}
^{N(\varepsilon_k)} \eta_i(\varepsilon_k)\, \frac{S(T_i^{\varepsilon_k})(\omega)}{|T_i^{\varepsilon_k}|} \right\|_Y
\end{eqnarray*}
for $j \geq K$. In the following, we fix $j 
\geq K$, choose $\eta_0$ as in Definition~\ref{defi:UDT} and fix $0 < \eta < \eta_0$.
Then, we find
\begin{itemize}
\item some $(U_jU_j^{-1}, \eta)$-invariant set $\hat{U}_j$ along 
with
\item an associated uniform decomposition tower $(\Upsilon, \Lambda)$ with 
prototile sets $T_i^{\varepsilon_k}$, $(1 \leq i \leq N(\varepsilon_k))$,
\item a family of finite center sets $\hat{C}^{y}_i (y \in 
\Upsilon)$ for the $\varepsilon_k$-quasi tilings of $\hat{U}_j$, 
\item and for each $y \in \Upsilon$, a family of finite center sets
$C_i^{y, \lambda} (\lambda \in \Lambda)$ for the $\varepsilon_k$-quasi tilings of $U_j$.
\end{itemize}
We will show 
\[
\lim_{l \to \infty} \lim_{j \to \infty}
\Delta(j, \varepsilon_{k_l}, \omega) = 0
\]
almost-surely for a subsequence $(\varepsilon_{k_l})$. To do so, we follow the lines of the proof of Theorem~\ref{thm:METSF}. Fixing $\omega \in \tilde{\Omega}$, we 
obtain by means of the triangle inequality
\[
\Delta(j, \varepsilon_{k}, \omega) \leq \sum_{\ell =1}^5
D_{\ell}(j, \varepsilon_k, \omega),
\]
where the expressions $D_{\ell}(j,\varepsilon_k, \omega)$
are defined as in the proof of 
the mean ergodic theorem. Using the boundedness by $n$ 
and the limit relations~\eqref{eqn:refat2}, we obtain 
\begin{eqnarray} \label{eqn:refat3}
\limsup_{j \to \infty} \sum_{\ell = 2}^5 D_{\ell}(j,\varepsilon_k, \omega)
\leq (24 +26 + 48 + 16)n \, \varepsilon_k = 114n \, \varepsilon_k
\end{eqnarray}
as in the steps (2) to (5) of the mentioned proof
with $C=n$. For $D_1(j,\varepsilon_k, \omega)$, 
we will have to argue in a slightly different way. 
To do so, note first that due to the additivity of
the process, there is no boundary term present, 
i.e.\@ $b \equiv 0$. For $\lambda \in \Lambda$ and
$y \in \Upsilon$, we define 
\[
A^{\varepsilon_k}_{y, \lambda} := \bigcup_{i=1}^{N(\varepsilon_k)} \bigcup_{c \in C_i^{y,\lambda}} T_i^{\varepsilon_k}c.
\]
Further, for $c \in C_i^{y,\lambda}$, we denote by $T_i^{\varepsilon_k}(c)$ a subset of $T_i^{\varepsilon_k}$
with $|T_i^{\varepsilon_k}(c)| \geq (1-\varepsilon)|T_i^{\varepsilon_k}|$ with the property
that the sets $T_i^{\varepsilon_k}(c)c$ are pairwise disjoint for all $1 \leq i\leq N(\varepsilon_k)$, $c \in C_i^{y,\lambda}$ and 
\[
A^{\varepsilon_k}_{y, \lambda} =  \bigsqcup_{i=1}^{N(\varepsilon_k)} \bigsqcup_{c \in C_i^{y,\lambda}} T_i^{\varepsilon_k}(c)c.
\]
Using the additivity of the process $F_n$, we compute
\begin{eqnarray*}
D_1(j, \varepsilon_k, \omega) &\leq& |\Upsilon|^{-1} \,|\Lambda|^{-1}\, \int_{\Upsilon} \int_{\Lambda} \Bigg( 
\frac{\| F_n(U_j \setminus A^{\varepsilon_k}_{y, \lambda} )(\omega)\|_Y}{|U_j|} \\ 
&\quad& \quad +  \quad \frac{\sum_{i=1}^{N(\varepsilon_k)} \sum_{c \in C_i^{y,\lambda}} \|F_n(T_i^{\varepsilon_k}c \setminus T_i^{\varepsilon_k}(c)c)(\omega)\|_Y}{ |U_j|}
\Bigg) \, d\lambda \, dy
\end{eqnarray*}
for every $\omega \in \tilde{\Omega}$. Again by the boundedness of $F_n$, one obtains 
that the function
\[
D_1(\varepsilon_k, \omega) := \limsup_{j \to \infty} D_1(j, \varepsilon_k, \omega)
\]
is bounded by the constant $3n$ for all $\omega \in \tilde{\Omega}$. 
Moreover, the dominated ergodic theorem combined with 
the boundedness of the process $F_n$ in $L^1(\Omega, Y)$ yield
\[
\| D_1(\varepsilon_k, \cdot) \|_{L^1(\Omega, \RR)} \leq 4n\varepsilon_k +  2 n\varepsilon_k = 6n\varepsilon_k.
\]
Note that we used here that the sets $U_j$ are $(1-4\varepsilon_k)$-coverd by the 
sets $A^{\varepsilon_k}_{y, \lambda}$ 
and that the $\varepsilon_k$-disjoint translates $T_i^{\varepsilon_k}c$ are $(1-\varepsilon_k)$ covered by the disjoint translates $ T_i^{\varepsilon_k}(c)c$.
Finally, take a subsequence $\varepsilon_{k_l}$ such that 
\[
\lim_{l \to \infty} D_1(\varepsilon_{k_l}, \omega) = 0
\]
for all $\omega \in \tilde{\Omega} \cap \hat{\Omega}$, where $\hat{\Omega}$
is a set of full measure as well. 
With inequality~\eqref{eqn:refat3}, we arrive at
\[
\lim_{l \to \infty} \limsup_{j \to \infty} \Delta(j, \varepsilon_{k_l}, \omega) = 0
\]
almost-surely. In the same manner as in the proof of Theorem~\ref{thm:METSF},
we conclude that $(F_n(U_j)(\omega)/|U_j|)_j$ is convergent in $Y$ for
almost-every $\omega \in \Omega$.
\end{proof}

Combining the above proposition with the Dominated Ergodic Theorem~\ref{thm:DET},
we can finally prove the almost-sure convergence result for approximable
bounded, additive processes (see also \cite{Pogorzelski-PhD}, Theorem~5.17).

\begin{Theorem}[Convergence of bounded, additive processes] \label{thm:PWET}
Assume that $G$ is a unimodular, amenable group and denote by $(U_j)$ a strong F{\o}lner sequence such that $U_j \subseteq U_{j+1}$ $(j \in \NN)$, which satisfies the Tempelman condition. Let $F: \mathcal{F}(G) \rightarrow L^1(\Omega, Y)$ $(Y \mbox{ reflexive})$ be a bounded, additive process, where $\mu(\Omega)< \infty$. Further, suppose that $F$ is compatible with a family $\{T_g\}_{g \in G}$ of uniformly bounded operators acting weakly measurably on $L^1(\Omega, Y)$ $($i.e.\@ $T_gF(Q) = F(Qg^{-1})$ and the Inequality (\ref{eqn:tmp}) is satisfied for all $g \in G$ and every $Q \in \mathcal{F}(G)\,)$. \\
If, in addition, the process $F$ is approximable, then
\[
\lim_{j \rightarrow \infty} \left\| \frac{F(U_j)}{|U_j|}(\omega) - \overline{F}(\omega) \right\|_Y = 0  
\]
for $\mu$-almost every $\omega \in \Omega$. Further, for every $g \in G$, we have $T_g\overline{F} = \overline{F}$ $\mu$-almost everywhere.  
\end{Theorem}

\begin{proof}
Since $F$ is approximable, for every $n \in \NN$, there is a sequence of processes
\[
F_n: \mathcal{F}^0(\Gamma) \to L^{\infty}(\Omega, Y)
\]
as described in Definition~\ref{defi:appr}.
For $n \in \mathbb{N}$, define
\[
H_n: \mathcal{F}^0(\Gamma) \to L^1(\Omega, Y):
H_n(Q)(\omega) := F(Q)(\omega) - F_n(Q)(\omega).
\]
By Definition~\ref{defi:appr}, the $H_n$ are bounded, additive processes 
satisfying the regularity  condition
given in equality~\eqref{eqn:tmp}. Moreover, 
Lemma~\ref{lemma:FNullcomp} 
shows that the same holds true for the processes $H_n^0$.
By definition of an approximable processes, we have
\[
\lim_{n \to \infty} \limsup_{j \to \infty} \| H^0_n(U_j)/|U_j| \|_{L^1(\Omega,\RR)}  = 0.
\]
By the dominated ergodic theorem, 
Theorem~\ref{thm:DET}, we conclude that there is a constant
$\gamma > 0$ such that
for all $\varepsilon > 0$,
every $n(\varepsilon) \in \NN$ and each  $\lambda >0$, one gets
\begin{eqnarray*}
\mu\left( \left\{ \omega \in \Omega \, \Big| \, \limsup_{j \to \infty} \frac{H^{0}_n(U_j)(\omega)}{|U_j|} > 
\lambda  \right\} \right) &\leq& \frac{\gamma}{\lambda}\,\limsup_{j \to
\infty} \frac{\| H^0_n (U_j) \|_{L^1(\Omega, \RR)}}{|U_j|} \\
&\leq& \frac{\gamma}{\lambda}\, \varepsilon.
\end{eqnarray*}
This shows that 
\begin{eqnarray} \label{eqn:almostthere}
\lim_{n \to \infty} \limsup_{j \to \infty} \left\| \frac{F(U_j)(\omega)}{|U_j|} -
\frac{F_n(U_j)(\omega)}{|U_j|} \right\|_Y = 0
\end{eqnarray}
for $\mu$-almost-every $\omega \in \Omega$. 
Further, it follows from Proposition~\ref{prop:AUX} that for all
$n \in \NN$, there is an element 
$F_n^{*} \in L^1(\Omega, Y)$ such that
\[
\lim_{j\to \infty} \left\| \frac{F_n(U_j)(\omega)}{|U_j|} - F_n^{*}(\omega) \right\|_Y = 0
\]
almost-surely. Inserting this into the limit relation~\eqref{eqn:almostthere} yields
\begin{eqnarray*}
\lim_{n \to \infty} \limsup_{j \to \infty} \left\| \frac{F(U_j)(\omega)}{|U_j|} -
F_n^{*}(\omega) \right\|_Y = 0.
\end{eqnarray*}
Therefore, for almost-every $\omega \in \Omega$, the sequence $\big( F(U_j)(\omega)/|U_j|\big)_j$
is Cauchy in the Banach space $Y$. Hence, it must converge to some element $\bar{F}^{*}$ almost-surely.
By Theorem~\ref{thm:METSF}, the ratios $F(U_j)/|U_j|$ converge in $L^1(\Omega, Y)$ to some
element $F^{*} \in L^1(\Omega, Y)$, and in addition, one obtains $T_gF^{*}(\omega) = F^{*}(\omega)$
for almost-all $\omega \in \Omega$. (To check that the compactness criterion required in the 
abstract mean ergodic theorem holds true, the reader may e.g.\@ refer to 
\cite{DiestelU-91}, Theorem IV.2.1.) Hence, $\bar{F}^{*} = F^{*}$ almost-surely and 
we have finished the proof of the theorem. 
\end{proof}
\section{An IDS model for continuous groups} \label{sec:IDScont}

In the following, we give an application of the ergodic theorems proven in the Sections \ref{sec:MET} and \ref{sec:PWET}. More precisely, we prove convergence of the integrated density of states for a class of random operators on discrete structures possessing a quasi isometry to an amenable group. We remark at this point that the integrated density of states is a well studied topic in mathematical physics. For (possibly random) Schr\"odinger operators on $L^2(\RR^d)$, characteristic properties such as Lifshitz tails, Anderson localization and Lipshitz continuity have been investigated over the past decades, cf.\@ e.g.\@ \cite{Okura-79, Fukushima-81, CarmonaL-90, KirschW-06, CombesHK-07, KirschV-10, GruberLV-11}. Concerning IDS models for abstract, continuous spaces, the literature mainly refers to periodic Schr\"odinger operators on abelian or amenable covering manifolds, see e.g.\@ \cite{KobayashiOS-89, LenzPV-07, LenzPPV-08}. 
For the special case of amenable Lie groups, such as the continuous Heisenberg group, see also \cite{Veselic-01}.

In this chapter, we apply our results to the model of {\sc Lenz} and {\sc Veseli\'c} (\cite{LenzV-09}). In the latter work, the authors prove uniform convergence of the IDS. A related Banach space theorem for discrete groups in a deterministic setting can be found in \cite{PogorzelskiS-11}. With our Banach space valued ergodic Theorems \ref{thm:METSF} and \ref{thm:lindenstr} at hand, we can deal with randomness and with general unimodular, amenable groups. Moreover, we are able to avoid the measure theoretical machinery used in \cite{LenzV-09}. However, for compactness reasons, we need to work with a reflexive Banach space. Therefore, it is convenient for our purposes to consider convergence in $L^p(I)$, where $1 < p < \infty$ and where $I \subset \RR$ is some bounded interval. In the main Theorems \ref{thm:mean} and \ref{thm:pointwise}, we show the convergence of suitable IDS-approximants in $L^1(\Omega, L^p(I))$, and pointwise almost everywhere in $L^p(I)$ respectively.   

\subsection{The model} \label{subsec:model} 

We cite the model of  in \cite{LenzV-09} and we also stick close to their notation.\\   
Assume that $(X,d_X)$ is a locally compact metric space with a countable basis of the topology. Assume further that $G$ is a second countable unimodular amenable group with an invariant metric $d_G$ such that every ball is precompact. 

We assume further that $G$ acts continuously from the right by isometries on $X$ such that the following two properties hold:
\begin{itemize}
\item there exists a right fundamental domain $J^{'}$ with compact closure $J$, which is a countable union of compact sets
\item the map $\Phi: X \rightarrow G: x \mapsto g$, whenever $x \in J^{'}g$, is a {\em quasi isometry}, i.e. there exist $a \geq 1$ and $b \geq 0$ with
\begin{eqnarray*}
\frac{1}{a}\, d_G(\Phi(x), \Phi(y)) - b \leq d_X(x, y) \leq a\, d_G(\Phi(x), \Phi(y)) + b
\end{eqnarray*}
for all $x,y \in X$.
\end{itemize}

For a set $A \subseteq G$ and $r > 0$, we write $A_r := \{g \in G\,|\, d_G(g, G\setminus A) > r \}$, as well as $A^{r}:= \{g \in G \,|\, d_G(g,A) < r \}$ and $\partial^{r}(A) := A^r \setminus A_r$. Analogously, with the metric $d_X$ at hand, we introduce this notation for subsets of the space $X$. 

For some fixed parameter $\eta > 0$, we set $\mathcal{D}$ as the family of $\eta$-uniformly discrete subsets of $X$, i.e.\@
\begin{eqnarray*}
\mathcal{D} := \{A \subset X \,|\, d_X(x,y) \geq \eta, \quad \mbox{for } x,y \in A \mbox{ with } x \neq y\}.
\end{eqnarray*} 
We define the set $\tilde{\mathcal{D}}$ as
\begin{eqnarray*}
\tilde{\mathcal{D}} := \{(A,h) \,|\, A \in \mathcal{D}, h: A \times A \rightarrow \CC^{*}\},
\end{eqnarray*}
where $\CC^{*}$ is an arbitrary compactification of $\CC$. 

This space can be naturally equipped with the vague topology and it is compact, cf.\@ \cite{LenzV-09}. The action of $G$ on $X$ induces an action from the right on $\tilde{\mathcal{D}}$ by $g\cdot(A,h) = (Ag^{-1}, h(xg, yg))$ for $g \in G$ and $(A,h) \in \tilde{\mathcal{D}}$. It is well known that there exists a $G$-invariant probability measure on $\tilde{\mathcal{D}}$, whose topological support will be denoted by $\Omega$. Then, $\Omega$ is a compact subset of $\tilde{\mathcal{D}}$. Note that each element $\omega \in \Omega$ can be written as $\omega := (X(\omega), h_\omega) \in \tilde{\mathcal{D}}$, where $X(\omega)$ is $\eta$-uniformly discrete and $h_{\omega}: X(\omega) \times X(\omega) \rightarrow \CC^{*}$ is a map. Each $X(\omega)$ gives rise to a Hilbert space $\ell^2(X(\omega))$, endowed with the canonical counting measure $\delta_{X(\omega)}:= \sum_{x \in X(\omega)} \delta_x$.  

We will draw attention to the bounded operators $H_{\omega}$ on $\ell^2(X(\omega))$, defined as
\begin{eqnarray*}
(H_{\omega}u)(x) := \sum_{y \in X(\omega)} h_{\omega}(x,y) \, u(y)
\end{eqnarray*}
for each $x \in X(\omega)$. Moreover, we assume that the $H_{\omega}$ are {\em of finite hopping range}, i.e.\@ there exists some number $R > b$ such that for all $\omega \in \Omega$, we have $h_{\omega}(x,y) = 0$ whenever $d_X(x,y) \geq R$. For $g \in G$, we let 
\begin{eqnarray*}
U_g: \ell^2(X(\omega)) \rightarrow \ell^2(X(\omega)): (U_gu)(x) := u(xg)
\end{eqnarray*} 
with adjoint ${U}^{*}_g = U_{g^{-1}}$. With that notion, we assume that the operators $H_{\omega}$ are {\em equivariant}, i.e.\@ 
\begin{eqnarray} \label{eqn:equivar}
{U}^{*}_g \, H_{g\omega} \, {U}_g = H_{\omega}
\end{eqnarray}
for all $g \in G$ and every $\omega \in \Omega$. Also, we need that the $H_{\omega}$ are {\em self-adjoint}. 

For further considerations, we need to restrict and to expand the operators $H_{\omega}$. In light of that, for $Q \in \mathcal{F}(G)$, we denote by $i_Q: \ell^2(X(\omega) \cap (JQ)_R) \rightarrow \ell^2(X(\omega))$ the canonical inclusion operator and by $p_Q: \ell^2(X(\omega)) \rightarrow \ell^2(X(\omega) \cap (JQ)_R)$ the canonical projection operator for $\omega \in \Omega$, where $(JQ)_R$ stands for the $R$-interior of $JQ$, i.e.\@
 \[
 (JQ)_R= \{ x \in JQ \,|\, d_X(x,X \setminus JQ ) \geq R\}.
 \] 

We will now show the $L^p$-existence of the integrated density of states in the above setting. For the Banach space, we set $Z:=Z(I):= L^1(\Omega, L^p(I))$ as the $L^1$-{\em Bochner space} of equivalence classes of functions mapping each $\omega \in \Omega$ to some element $F_{\omega}$ in the space of $p$-integrable functions over some bounded interval $I \subset \RR$, where $1 < p < \infty$. \\

Now for fixed $\omega \in \Omega$, we consider the restricted operators $H^R_{\omega}[Q]: \ell^2(X(\omega) \cap (JQ)_R) \rightarrow \ell^2(X(\omega) \cap (JQ)_R)$, where 
\begin{eqnarray*}
H_{\omega}^R[Q]:= p_{Q} H_{\omega} i_Q    
\end{eqnarray*}
and $Q \in \mathcal{F}(G)$.
Note that since $X(\omega)$ is $\eta$-uniformly discrete, each such $H^R_{\omega}[Q]$ can be described by a quadratic matrix. For those objects, one can define the eigenvalue counting funtion $F_{\omega}(Q)(\cdot) \in L^p(I)$ as
\begin{eqnarray} \label{eqn:EFF}
F_{\omega}(Q)(E)&:=& \#\{ i\in \NN \,|\,\lambda_i \mbox{ is eigenvalue of } H_{\omega}^R[Q] \mbox{ and } \lambda_i \leq E \} \nonumber  \\
&=& \operatorname{tr}\Big( ( \one_{]-\infty, E]} H_{\omega}^R[Q]) \Big) 
\end{eqnarray}  
for $E \in \RR$. Note that for each $\omega \in \Omega$ and every $Q \in \mathcal{F}(G)$, the element $F_{\omega}(Q)(\cdot)$ also belongs to $C_{rb}(\RR)$, the Banach space of all bounded, right-continuous functions, endowed with supremum norm $\|\cdot\|_{\infty}$. 

\begin{Proposition} \label{prop:bounded}
There exists some constant $C> 0$ independent of $\omega$ and $Q$ such that $\|F_{\omega}(Q)(\cdot)\|_{\infty} \leq C\, |Q|$ for every $Q \in \mathcal{F}(G)$ and each $\omega \in \Omega$.  \\
In particular, $F_{(\cdot)}(Q) \in L^{\infty}(\Omega, L^p(I)) \cap L^{1}(\Omega, L^p(I))$ for all $Q \in \mathcal{F}(G)$, every $1 \leq p < \infty$ and every bounded interval $I \subset \RR$.  
\end{Proposition}

\begin{proof}
Since $R > b$ and by the quasi isometry between $G$ and $X$, we obtain the inclusion $(JQ)_R \subseteq J(Q_S)$ for some $0 < S < (R-b)/a$, where as defined above $Q_S = \{g \in Q\,|\, d_G(g, G \setminus Q) \geq S\}$. 
By a similar calculation as in \cite{LenzV-09}, Proposition 3.3, we can find a constant $C>0$ depending only on the parameters $a,b,\eta$ and $S$ such that    
\begin{eqnarray*}
\|F_{\omega}(Q)\|_{\infty} &=&  \operatorname{card}(X(\omega) \cap (JQ)_R) \\
&\leq& \operatorname{card}(X(\omega) \cap J(Q_S)) \\
&\leq& C \, |(Q_S)^S| = C \, |Q|, 
\end{eqnarray*}
where we use the notation that $Q^{\rho}:= \{ g \in G \,|\, d_G(g, Q) \leq \rho \}$ for $\rho > 0$.\\
Using this result, we simply compute
\begin{eqnarray*}
\left( \int_I |F_{\omega}(Q)(E)|^p\,dE \right)^{1/p} &\leq& \left(|I|\cdot \|F_{\omega}(Q)(\cdot)\|^p_{\infty} \right)^{1/p} \\
&\leq& |I|^{1/p} \cdot \|F_{\omega}(Q)(\cdot)\|_{\infty} \\
&\leq& C\,|I|^{1/p}\cdot |Q|
\end{eqnarray*}
for every $\omega \in \Omega$. This finishes the proof of the proposition. 
\end{proof}


\begin{Proposition} \label{prop:equiv}
For each $Q \in \mathcal{F}(G)$ and every $g \in G$, one gets
\begin{eqnarray*}
F_{\omega}(Qg)(\cdot) = F_{g \omega}(Q)(\cdot)
\end{eqnarray*}
for every $\omega \in \Omega$. 
\end{Proposition}

\begin{proof}
Let $\lambda \in \RR$ be an eigenvalue of the operator $H_{\omega}^R[Qg]$ for $Q \in \mathcal{F}(G)$ and $g \in G$ and denote the corresponding eigenfunction as $u \in \ell^2(X(\omega) \cap (JQ)_R g)$ (note that since $G$ acts on $X$ by isometries, we have $(JQg)_R = (JQ)_R g$). Evidently, it is enough to show that $\tilde{u}:= U_g u$ is an eigenfunction of the operator $H_{g\omega}^R[Q]$ for the eigenvalue $\lambda$. To do so, we compute for $b \in X(\omega) \cap (JQ)_R g$

\begin{eqnarray*}
\lambda\, \tilde{u}(bg^{-1}) &=& \lambda \, U_g u(bg^{-1}) \\
&=& \lambda \, u(b) \\
&=& (H_{\omega}^R[Qg] \, u)(b) \\
&=& (H_{\omega}\, i_{Qg} \, u)(b) \\
&=& ({U}_g \, H_{\omega} \, i_{Qg} \, u)(bg^{-1}) \\
&=& (H_{g \omega} \, {U}_g \, i_{Qg} \, u)(bg^{-1}) \\
&=& (H_{g \omega} \, i_Q \, U_g u)(bg^{-1}) \\
&=& (p_Q \, H_{g \omega} \, i_Q \, U_g u)(bg^{-1}) \\
&=& (H_{g \omega}^R[Q] \, \tilde{u})(bg^{-1}).
\end{eqnarray*}
Since $b$ was chosen arbitrarily, this proves the claim. 
\end{proof}

For $g \in G$, define
\[
T_g: L^1(\Omega, L^p(I)) \rightarrow L^1(\Omega, L^p(I)): T_gf(\omega):= f(g^{-1}\omega).  
\]

\begin{Proposition} \label{prop:act}
The corresponding collection $\{T_g\}_{g \in G}$ is a family of uniformly bounded operators that acts 
weakly measurably on $Z=L^1(\Omega, L^p(I))$.
\end{Proposition}

\begin{proof}
The linearity of the $T_g$ follows from construction. By the fact that the action of $G$ on $\Omega$
preserves the probability measure $\mu$, we have $\|T_gf\|_Z = \|f\|_Z$ for all $g \in G$ and each $f \in Z$. 
The weak measurability follows from standard concepts.  
\end{proof}

It now follows from the Propositions \ref{prop:equiv} and \ref{prop:act} that the weakly measurable action of
the family $\{T_g\}_{g \in G}$ implies
\begin{eqnarray}
T_gF_{(\cdot)}(Q) = F_{(\cdot)}(Qg^{-1}) \label{eqn:act2}
\end{eqnarray}
for each $Q \in \mathcal{F}(G)$ and each $g \in G$.  

\begin{Proposition} \label{prop:cpct}
Let $Q \in \mathcal{F}(G)$. 
Then the set $C_{F,Q}:= {\operatorname{co}}\{ F_{(\cdot)}(Qg)\,|\, g \in G\}$ is relatively weakly compact in $Z=L^1(\Omega, L^p(I))$ for each $1 < p < \infty$ and for any bounded interval $I \subset \RR$.  
\end{Proposition}

\begin{proof}
This follows from the general theory for Bochner spaces, see e.g.\@ \cite{DiestelU-91}, Theorem IV.2.1. 
\end{proof}

\begin{Lemma} \label{lemma:almostadditive}
For each $\omega \in \Omega$, the function $F_\omega: \mathcal{F}(G) \rightarrow L^p(I)$ is $b$-almost additive with the tiling-admissible boundary term
\begin{eqnarray*}
b: \mathcal{F}(G) \rightarrow [0, \infty): b(Q) := \tilde{C}\cdot |\partial^{\overline{R}}(Q)|,
\end{eqnarray*} 
where $\tilde{C}$ and $\overline{R}$ are constants depending on $a,b,\eta, R$, and $\partial^{\overline{R}}(Q) := Q^{\overline{R}} \setminus Q_{\overline{R}}$ for $Q \in \mathcal{F}(G)$. \\
Moreover, the same holds true for the function $F:\mathcal{F}(G) \rightarrow L^1(\Omega, L^p(I))$. 
\end{Lemma}

\begin{proof}
For the error estimate, assume that $Q \in \mathcal{F}(G)$ is the disjoint union of the sets $Q_k \in \mathcal{F}(G)$, $1 \leq k \leq m$. We can repeat the arguments of \cite{PogorzelskiS-11}, Theorem 7.4, in order to find positive constants $\tilde{C}, \overline{R} \in \RR$ depending on only $a,b, \eta$ and $R$ such that
\begin{eqnarray*}
&& \left(\int_I \left| \operatorname{tr}(\one_{]-\infty, E]} H_{\omega}^R[Q]) - \sum_{k=1}^m \operatorname{tr}(\one_{]-\infty, E]} H_{\omega}^R[Q_k]) \right|^p \, dE \right)^{1/p} \\
&\quad\quad& \leq |I|^{1/p} \cdot \sup_{E \in \RR}\,\left| \operatorname{tr}(\one_{]-\infty, E]} H_{\omega}^R[Q]) - \sum_{k=1}^m \operatorname{tr}(\one_{]-\infty, E]} H_{\omega}^R[Q_k])  \right| \\
&\quad\quad& \leq \tilde{C} \, \sum_{k=1}^m |\partial^{\overline{R}}(Q_k)|.
\end{eqnarray*}
To conclude the proof, note that the mapping $b$ has all the properties (i)-(iii) of Definition \ref{defi:BT}. 
Therefore, it is a boundary term and since 
\[
\partial^{\overline{R}}(Q) \subseteq \partial_{B_{2\overline{R}}(e)}(Q)
\]
for all $Q \in \mathcal{F}(G)$, it follows from Proposition \ref{prop:tiladmis} that $b$ is also tiling-admissible. \\
By integrating over $\Omega$ and by the fact that $\mu(\Omega)= 1$, we observe that the same claim must hold true for the function $F$. 
\end{proof}

\subsection{Proof of the main theorems}

In this subsection, we prove the main theorems of this section. Thus, using the mean ergodic Theorem \ref{thm:METSF}, we show that for a strong F{\o}lner sequence $\{U_j\}_{j=1}^{\infty}$, the expressions $F_{(\cdot)}(U_j)(\cdot)/|U_j|$ converge to some $\overline{F}$ in $Z=L^1(\Omega, L^p(I))$, where as above, $1 < p < \infty$. 
Moreover, we demonstrate that we can use the generalization of the {\sc Lindenstrauss} ergodic theorem (Theorem \ref{thm:lindenstr}) to establish the $\mu$-almost everywhere convergence for Shulman (tempered) F{\o}lner sequences.

\begin{Theorem} \label{thm:mean}
Assume the model of Subsection \ref{subsec:model} and let some strong F{\o}lner sequence $(U_j)$ be given. Then there is an element $\overline{F} \in Z=L^1(\Omega, L^p(I))$ such that the mean ergodic convergence holds for $F_{(\cdot)}$, i.e.\@ 
\[
\lim_{j \rightarrow \infty} \left\| \frac{F_{(\cdot)}(U_j)(\cdot)}{|U_j|} - {\overline{F}_{(\cdot)}(\cdot)} \right\|_Z  =  \lim_{j \rightarrow \infty} \left| \int_{\Omega} \left(\left\| \frac{F_{\omega}(U_j)}{|U_j|}(\cdot) -  \overline{F}_{\omega}(\cdot)\right\|_{L^p(I)} \right) \,d\mu(\omega)  \right| = 0.
\]
Moreover, for each $g \in G$ we have $\overline{F}_{g\omega} = \overline{F}_{\omega}$ in $L^p(I)$ for $\mu$-almost every $\omega \in \Omega$. 
\end{Theorem}

\begin{proof}
We check that the mapping 
\[
F: \mathcal{F}(G) \rightarrow L^1(\Omega, L^p(I))
\]
as defined in Equality (\ref{eqn:EFF}) satisfies all the assumptions of Theorem \ref{thm:METSF}. By Proposition \ref{prop:bounded}, $F$ is bounded. It follows from Proposition \ref{prop:act} and from the Inequality (\ref{eqn:act2}) that there is a family $\{T_g\}_{g \in G}$ of uniformly bounded operators acting weakly measurably on $Z$ with the property that
$T_gF_{(\cdot)}(Q) = F_{(\cdot)}(Qg^{-1})$ in $Z$ for every $Q \in \mathcal{F}(G)$ and all $g \in G$. Moreover, we infer from Proposition \ref{prop:cpct} that for each $Q \in \mathcal{F}(G)$, the weak closure of the set $C_{F,Q} = {\operatorname{co}}\{F(Qg)\,|\,g \in G\}$ is weakly compact in $Z$. Finally, Lemma \ref{lemma:almostadditive} asserts that $F$ is almost-additive with respect to some tiling admissible boundary term.\\
Consequently, we can apply Theorem \ref{thm:METSF} to obtain the claimed convergence.     
\end{proof}

\begin{Theorem} \label{thm:pointwise}
Assume the model of subsection \ref{subsec:model} and let some strong, tempered F{\o}lner sequence $(U_j)$ be given. Then there is an element $\overline{F} \in Z=L^1(\Omega, L^p(I))$ such that the pointwise ergodic convergence holds for $F_{(\cdot)}$, i.e.\@
\[
\lim_{j \rightarrow \infty} \left\| \frac{F_{\omega}(U_j)(\cdot)}{|U_j|} - {\overline{F}_{\omega}(\cdot)} \right\|_Z
\]
for $\mu$-almost every $\omega \in \Omega$. Furthermore, for each $g \in G$ we have $\overline{F}_{g\omega} = \overline{F}_{\omega}$ in $L^p(I)$ for $\mu$-almost every $\omega \in \Omega$. 
\end{Theorem}

\begin{proof}
Recall that by the considerations of the previous subsection, the linear mappings $T_gf(\omega):= f(g^{-1}\omega)$ $(g \in G)$ act weakly measurably on $L^1(\Omega, L^p(I))$ as a family of uniformly bounded operators. It follows from Proposition \ref{prop:equiv} that for each $g \in G$ and every $Q \in \mathcal{F}(G)$, we have $T_g F_{\omega}(Q) = F_{\omega}(Qg^{-1})$ for all $\omega \in \Omega$. Therefore, we can apply the extended Lindenstrauss ergodic Theorem \ref{thm:lindenstr} which yields that for all $Q \in \mathcal{F}(G)$, the limit 
\begin{eqnarray*}
S(Q)(\omega)&:=& \mbox{Y-}\lim_{j \rightarrow\infty} |U_j|^{-1} \int_{U_j} T_{g^{-1}}F_{\omega}(Q) \, dg \\
&=& \mbox{Y-}\lim_{j \rightarrow\infty} |U_j|^{-1} \int_{U_j} F_{\omega}(Qg) \, dg
\end{eqnarray*}
exists in $Y=L^p(I)$ for every $\omega \in \tilde{\Omega}$, where $\tilde{\Omega} \subseteq \Omega$ is a set of full measure. Note also that for each $Q \in \mathcal{F}(G)$ and for all $g \in G$, we have $S(Qg)(\omega) = S(Q)(\omega)$ for all $\omega \in \hat{\Omega}$, where again,  $\hat{\Omega} \subseteq \tilde{\Omega}$ is a set of full measure. \\
Further, we can infer from Proposition \ref{prop:bounded} that for each $\omega \in \Omega$, there is some constant $C > 0$ such that $\|F_{\omega}(Q)\|_{L^p(I)} \leq C \,|I|^{1/p} \,|Q|$ for all $Q \in \mathcal{F}(G)$. The Lemma \ref{lemma:almostadditive} guarantees that for every element $\omega \in \Omega$, the map $F_{\omega}$ is $b$-almost additive with the tiling-admissible boundary term $b(Q):= \tilde{C}\,|\partial^{\overline{R}}(Q)|$ $(Q \in \mathcal{F}(G))$, where the constants $\tilde{C},\overline{R} > 0$ are chosen accordingly (see above). Hence, we have verified all assumptions which are necessary to apply Corollary \ref{cor:METSF}. More precisely, for a given null sequence $(\varepsilon_k)$ of positive numbers, can find a set $\hat{\Omega} \subseteq \Omega$ of full measure such that for each $\omega \in \hat{\Omega}$, the limit $S(T_i^{\varepsilon_k})(\omega)$ exists and $S(T_i^{\varepsilon_k})(g\omega) =S(T_i^{\varepsilon_k}g)(\omega) = S(T_i^{\varepsilon_k})(\omega)$ for all $g \in G$, all $k \in \NN$ and all $1 \leq i \leq N(\varepsilon_k)$, where the elements $T_i^{\varepsilon_k}$ and $N(\varepsilon_k)$ correspond to the underlying $\varepsilon_k$-quasi tiling of the group. By the boundedness and the $b$-almost additivity of the $F_{\omega}$ $(\omega \in \hat{\Omega})$ we directly apply Corollary \ref{cor:METSF} with $Z=Y$ as the Banach space under consideration and we obtain the convergence in $Y$ for every $\omega \in \hat{\Omega}$. For the proof of the invariance property, let $g \in G$ in order to obtain 
\begin{eqnarray*}
\|\overline{F}_{\omega} -\overline{F}_{g\omega}\|_Z &\leq& \left\| \overline{F}_{\omega} - \sum_{i=1}^{N(\varepsilon_k)} \eta_i(\varepsilon_k) \, \frac{S(T_i^{\varepsilon_k})(\omega)}{|T_i^{\varepsilon_k}|}  \right\|_Z + \left\| \sum_{i=1}^{N(\varepsilon_k)} \eta_i(\varepsilon_k) \, \frac{S(T_i^{\varepsilon_k})(g\omega)}{|T_i^{\varepsilon_k}|} - \overline{F}_{g\omega} \right\|_Z 
\end{eqnarray*}
for all $k \in \NN$ and for every $\omega \in \hat{\Omega}$. Letting $k \rightarrow \infty$ finishes the proof.          
\end{proof}

\begin{Remark}
It follows from the considerations in \cite{LenzV-09} that in the ergodic case, one can express the limit $\overline{F}$ as 
\[
\overline{F}_{\omega}(E) = C \, \int_{\Omega} \operatorname{tr}(\one_J \one_{]-\infty,E]}H_{\omega}) \,d\mu(\omega)
\]
for $E \in \RR$ and for $\mu$-almost every $\omega \in \Omega$, where $C$ stands for a constant depending on $G,X$ and $\mu$ (cf.\@ \cite{LenzV-09}, Section 2). 
\end{Remark}

\section{Appendix} \label{sec:appendix}

We give the proofs of Theorem \ref{thm:LPmax} and of Lemma \ref{lemma:partition}, respectively. We put the proof of Theorem \ref{thm:LPmax} in this appendix because it mainly requires arguments which have been established in \cite{Lindenstrauss-01}. The validity of Lemma \ref{lemma:partition} has already been known in \cite{Emilion-85}. Since we could not find a rigorous proof in the literature, we attach one for the sake of the reader.  

\begin{proof}[Proof of Theorem \ref{thm:LPmax}]
Take $1 \leq p < \infty$ and choose $\lambda > 0$, as well as $f \in L^p(\Omega, Y)$. Note that by means of the triangle inequality, the condition $\|A_j f\|_Y(\omega) > \lambda$ implies that
\begin{eqnarray*}
|U_j|^{-1} \left( \int_{U_j} \left\| (T_{g^{-1}}f)(\cdot) \right\|_Y \,dm_L(g)\right)(\omega) > \lambda
\end{eqnarray*}
for $j \in \NN$. The Inequality (\ref{eqn:tmp}) yields
\[
\kappa \, |U_j|^{-1} \int_{U_j} \|f(\phi(g)\omega)\|_Y \, dm_L(g) > \lambda 
\]
for all $j \in \NN$. It follows from this that it is sufficient to show that there must be a constant $\tilde{C} > 0$ such that for every $\lambda > 0$, all $h \in L^p(\Omega,\RR)$ and for each measurable group homomorphism $\varphi$, we have
\begin{eqnarray} \label{eqn:lin}
\mu(\{ \omega \,|\, M^{\phi}h(\omega) > \lambda \}) \leq  \frac{\tilde{C}}{\lambda^p} \|h\|^p_{L^p(\Omega, \RR)},
\end{eqnarray}
where $M^{\phi}$ is the associated maximal operator, cf.\@ Definition \ref{defi:maxOP}.
The proof of this claim will be an adaption of the proof in \cite{Lindenstrauss-01}. We will sketch the major steps. For a detailed discussion, see also \cite{Pogorzelski-10}, Chapters 6-8.

So fix $n \in \NN$ be some integer and define $F:= \cup_{j=1}^n U_j$. Since $\{U_j\}$ is a F{\o}lner sequence, there must be some $J \in \NN$ such that $|U_J \triangle FU_J| < |U_J|$ which implies that $|\overline{U}| < 2\, |U_J|$, where $\overline{U}:= FU_J$. We now choose $h \in L^p(\Omega, \RR)$, $\omega \in \Omega$, $1 \leq j \leq n$, $\phi: G\rightarrow G$, as well as $\lambda > 0$ and set
\begin{eqnarray*}
B_j := B_j(h,\omega, \lambda, \phi):= \{g \in U_J \,| \, |{A}^{\phi}_jh(\phi(g)\omega)| > \lambda \}
\end{eqnarray*}
and
\begin{eqnarray*}
D_n := \{ \omega \in \Omega \,|\, \max_{1 \leq j \leq n} |{A}_j^{\phi}h(\omega)| > \lambda \}.
\end{eqnarray*}
   
With these sets at hand, it is easy to verify that
\begin{eqnarray} \label{eqn:Bjott}
\left| \bigcup_{j=1}^n B_j(h,\omega, \lambda, \phi) \right| = \int_{U_J} \one_{D_n}(\phi(g)\omega) \, dm_L(g).
\end{eqnarray}
Since $G$ acts on $\Omega$ in a measure preserving manner, we obtain with the Fubini theorem that
\begin{eqnarray} \label{eqn:measureDN}
\mu(D_n) &=& |U_J|^{-1} \int_{\Omega} \int_{U_J} \one_{D_n}(\phi(g)\omega) \, dm_L(g) \, d\mu(\omega).
\end{eqnarray} 
Hence, if we can find some constant $\tilde{C}^{'} > 0$ independent of $n,h, \omega, \lambda$ and $\phi$ such that the so-called {\em transfer inequalities}
\begin{eqnarray} \label{eqn:transfer}
\int_{U_J} \one_{D_n}(\phi(g)\omega) \, dm_L(g) \leq \frac{\tilde{C}^{'}}{\lambda^p}\, \int_{\overline{U}}|h(\phi(g)\omega)|^p \, dm_L(g)
\end{eqnarray}
hold ($\mu$-almost all $\omega$), we can insert Inequality (\ref{eqn:transfer}) into Inequality (\ref{eqn:measureDN}) to obtain with the fact that the action of $G$ preserves
$\mu$ that
\begin{eqnarray*}
\mu(D_n) \leq \frac{2\tilde{C}^{'}}{\lambda^p} \, \|h\|^p_{L^p(\Omega,\RR)}
\end{eqnarray*}
and with $n \rightarrow \infty$, the theorem is proven. So let us sketch the proof of the Inequalities (\ref{eqn:transfer}). The key step is the following observation. Namely, for each $1 \leq j \leq n$ and every $b \in B_j$, we obtain with H{\"o}lder's inequality that
\begin{eqnarray*}
\int_{U_jb} |h(\phi(g)\omega)|^p \, dm_L(g) &\geq& |U_jb|^{-p/q} \left( \int_{U_jb} |h(\phi(g)\omega)| \, dm_L(g) \right)^p \\
&=& |U_jb|^{-p/q} \left( \int_{U_j} |h(\phi(gb)\omega)| \, dm_L(g) \right)^p \Delta(b)^p \\
&=& |U_jb|^{-p/q} |U_j|^p |{A}_j^{\phi}h(\phi(b)\omega)|^p \Delta(b)^p \\
&\stackrel{b \in B_j}{\geq}& \lambda^p |U_j b|^{p - p/q} = \lambda^p \, |U_jb|,
\end{eqnarray*}
where $1/p + 1/q = 1$ and $q = \infty$ for $p=1$ (with the convention that $\frac{1}{\infty} = 0)$ and where $\Delta$ stands for the modular function of $G$. Having this latter inequality at hand, the remaining arguments for the proof of the statement can be carried out in the same manner as demonstrated in \cite{Lindenstrauss-01}, Lemma 2.1 and Theorem 3.2 in \cite{Pogorzelski-10}, Corollary 8.8.
\end{proof}

\begin{proof}[Proof of Lemma \ref{lemma:partition}]
Since $G$ is second countable and Hausdorff and by the outer regularity of the Haar measure, for every $n \in \NN$, one finds a precompact neighbourhood $V_n$ of the unity $\operatorname{id}$ in $G$ such that $\cap_{n \in \NN} V_n = \{\operatorname{id}\}$, $V_{n+1} \subseteq V_n$ and $|V_n| < 2^{-n}$ for all $n \in \NN$. So let $Q \in \mathcal{F}(G)$. We will define a sequence of successively refined partitions of $Q$ in the following manner. For each $n \in \NN$, cover the closure of $Q$ by left-translates of $V_n$. Due to the precompactness of $Q$, we can extract a finite, open subcover $\cup_{i=1}^{K(n)} g^{(n)}_i V_n$ of $Q$. Next, we make the translates of this union disjoint such that 
\[
\bigcup_{i=1}^{K(n)} g^{(n)}_i V_n = \bigsqcup_{i=1}^{K(n)} g^{(n)}_i \tilde{V}^{i}_n
\]    
where $\bigsqcup$ stands for the disjoint union and $\tilde{V}^{i}_n \subseteq V_n$ is (Borel-)measurable (in fact $\tilde{V}^{i}_n \in \mathcal{F}(G)$) for all $1 \leq i \leq K(n)$. Hence, putting $Q^{(n)}_i := Q \cap g^{(n)}_i \tilde{V}^{i}_n$ for $1 \leq i \leq K(n)$, we have $Q = \bigsqcup_{i=1}^{K(n)} Q^{(n)}_i$, as well as $|Q^{(n)}_i| < 2^{-n}$ for every $1 \leq i \leq K(n)$. In light of that, define the partitions 
\begin{eqnarray*}
P_1(Q) &:=& \{ Q^{(1)}_i \,|\, 1 \leq i \leq K(1) \}, \\
P_m(Q) &:=& \{ Q^{(m)}_i \cap \tilde{Q}^{(m-1)}  \,|\,  1 \leq i \leq K(m), \, \tilde{Q}^{(m-1)} \in P_{m-1} \} \quad \mbox{ for } m \geq 1,
\end{eqnarray*} 
for the set $Q$.
It is obvious that $P_{m+1}(Q)$ is finer than $P_{m}(Q)$ for every $m \geq 1$, and we write $P_{m+1}(Q) \geq P_m(Q)$. Further, by construction, $|A| \leq 2^{-m}$ for $A \in P_m$ $(m \geq 1)$. Since $G$ is $\sigma$-compact, it can be exhausted by a countable sequence $(G_n)$ of increasing, compact sets. We set $\tilde{G}_n := G_n \setminus G_{n-1}$ with the convention that $G_{0}:= \{\operatorname{id}\}$ and we repeat the above construction for each $Q= \tilde{G}_n$, $n \geq 1$. Then clearly, for each $m \geq 1$, the expression 
\[
P_m := \bigcup_{n \in \NN} P_m(\tilde{G}_n)
\] 
is a partition of the group satisfying the first two claims of the Lemma \ref{lemma:partition}.  \\

To show the approximation result for $F^{0}$, let $Q \in \mathcal{F}(G)$ and define
\[
F_{P_m}^{0}(Q)(\omega) := \sum_{A \in P_m} \|F(A \cap Q)\|_Y(\omega).
\]  
for $m \geq 1$, where the partitions $P_m$ have been defined above. 
By the triangle inequality, we have $F^{0}_{P_m}(Q) \leq F^{0}_{P_{m+1}}(Q)$ $\mu$-almost surely for all $m \geq 1$. Further, it is clear from the boundedness of $F$ that $\|F^{0}_{P_m}(Q)\|_{L^p(\RR, Y)} \leq C\, |Q|$ for every $m \geq 1$. Hence, we can define
\begin{eqnarray} \label{eqn:limit!}
\overline{F}^{0}(Q)(\omega) := \lim_{m \rightarrow \infty} F_{P_m}^{0}(Q)(\omega)
\end{eqnarray}
for $\mu$-almost every $\omega \in \Omega$ and we set $\overline{F}^{0}(Q)(\omega) = 0$ for the remaining $\omega \in \Omega$. 
We will now show that in fact $F^{0}(Q) = \overline{F}^{0}(Q)$ $\mu$-almost everywhere. 
This shows that the equivalence class $F^{0}(Q)$ is well defined and in particular measurable.
Note that it follows from the definition of $F^{0}$ that for each $m \in \NN$, we have 
\[
F_{P_m}^{0}(Q)(\omega) \leq F^{0}(Q)(\omega)
\]
for $\mu$-almost every $\omega \in \Omega$. This implies $F^{0}(Q) \geq \overline{F}^{0}(Q)$ almost everywhere.  
For the converse inequality, we choose a finite, disjoint union $Q = \bigsqcup_{l=1}^L Q_l$, where $Q_l \in \mathcal{F}(G)$ for $1 \leq l \leq L$. For $\omega \in \Omega$, we obtain with the triangle inequality that 
\begin{eqnarray} \label{eqn:lim1}
\sum_{l=1}^L \|F(Q_l)\|_Y(\omega) &\leq& \sum_{l=1}^L \sum_{A \in P_m} \|F(Q_l \cap A)\|_Y(\omega) \nonumber \\
&\leq& \sum_{\ato{A \in P_m,}{\exists\,l:\, A \cap Q \subseteq Q_l}} \|F(A \cap Q)\|_Y(\omega) + \sum_{l=1}^L \sum_{\ato{A \in P_m, \, \exists\, l_1\neq l:}{A \cap Q_{l}, A \cap Q_{l_1} \neq \emptyset}} \|F(A \cap Q_l)\|_Y(\omega) \nonumber \\
&\leq& F^{0}_{P_m}(Q)(\omega) + \sum_{l=1}^L \sum_{\ato{A \in P_m, \, \exists\, l_1\neq l:}{A \cap Q_{l}, A \cap Q_{l_1} \neq \emptyset}} \|F(A \cap Q_l)\|_Y(\omega)
\end{eqnarray}
for all $m \in \NN$. Further, we have by the boundedness of $F$ that 
\[
\left\| \sum_{l=1}^L \sum_{\ato{A \in P_m, \, \exists\, l_1\neq l:}{A \cap Q_{l}, A \cap Q_{l_1} \neq \emptyset}}  \|F(A \cap Q_l)\|_Y(\cdot)   \right\|_{L^p(\Omega, \RR)} \leq C\, \sum_{l=1}^L |\partial_{V_m}(Q_l)|.
\]
Note that since $\cap_{m \in\NN} V_m = \{\operatorname{id}\}$, it follows from the continuity of the (Haar) measure, as well as by the fact that all involved sets have compact closure, that 
\[
|\partial_{V_m}(Q_l)| = |V_m^{-1}Q_l \cap V_m^{-1}(G\setminus Q_l)| \stackrel{m \rightarrow \infty}{\rightarrow} 0
\] 
for each $1 \leq l \leq L$. In light of that, we can find a subsequence $(m_k)$ such that the second sum in Inequality (\ref{eqn:lim1}) converges to zero almost everywhere. Hence, we deduce from that same inequality that 
\begin{eqnarray*}
\sum_{l=1}^L \|F(Q_l)\|_Y(\omega) &\leq& \limsup_{k \rightarrow \infty} F^{0}_{P_{m_k}}(Q)(\omega) = \overline{F}^{0}(Q)(\omega) 
\end{eqnarray*}
for almost every $\omega \in \Omega$. Hence $F^{0}(Q) \leq \overline{F}^{0}(Q)$ and since $Q \in \mathcal{F}(G)$ was arbitrarily chosen, this shows in fact that $F^{0} = \overline{F}^{0}$. 
\end{proof}

\subsection*{Acknowledgements}

I would like to thank my adviser {\sc Daniel Lenz} for enduring encouragement and for fruitful suggestions concerning the organization of this work. Also, I would like to express thanks to {\sc Fabian Schwarzenberger} for stimulating discussions, as well as for his contributions in the precedent projects which provided a major groundwork for the results of this paper. 
Thanks also go to {\sc Xueping Huang} for helpful remarks about Poisson point processes and absolute continuity.
Moreover, I take this chance to mention gratefully that this work was supported by the German National Academic Foundation (Studienstiftung des deutschen Volkes).

\bibliographystyle{amsalpha}
\bibliography{PS_lit}

\end{document}